\documentclass[12pt]{amsart}
\usepackage{amssymb}
\usepackage{amsmath}
\usepackage{bbm}
\usepackage[dvips]{graphicx}
\usepackage{hyperref}
\usepackage[capitalize,nameinlink,noabbrev,nosort]{cleveref}
\usepackage[dvipsnames]{xcolor}
\usepackage{todonotes}
\usepackage{mathtools}
\usepackage{enumitem}
\usepackage{caption}
\usepackage{subcaption}
\usepackage{fullpage}
\usepackage{mathabx}
\usepackage{ytableau}
\usepackage{tikz}
\usetikzlibrary{arrows}
\usetikzlibrary{decorations.markings}
\usetikzlibrary{tikzmark,decorations.pathreplacing}
\usetikzlibrary{shapes.geometric}
\numberwithin{equation}{section}
\usepackage{float}
\makeatletter
\def\Ddots{\mathinner{\mkern1mu\raise\p@
\vbox{\kern7\p@\hbox{.}}\mkern2mu
\raise4\p@\hbox{.}\mkern2mu\raise7\p@\hbox{.}\mkern1mu}}
\makeatother

\usepackage{etoolbox}
\patchcmd{\endproof}
  {\popQED}
  {\par\popQED}
  {}
  {}

\newcommand{\ds}{\displaystyle}
\DeclareMathOperator{\Sp}{Sp}

\DeclareMathOperator{\GL}{GL}
\DeclareMathOperator{\OO}{SO}

\renewcommand{\sp}{\mathrm{sp}}
\DeclareMathOperator{\oo}{so}

\renewcommand{\OE}{\mathrm{O}}
\renewcommand{\oe}{\mathrm{o}^{\text{even}}}

\DeclareMathOperator{\sgn}{sgn}

\newcommand{\floor}[1]{\left\lfloor #1 \right\rfloor}

\newcommand{\upperRomannumeral}[1]{\uppercase\expandafter{\romannumeral#1}}

\newcommand{\core}[2]{\mathrm{core}_{#2}{(#1)}}
\newcommand{\quo}[2]{\mathrm{quo}_{#2}{(#1)}}
\newcommand{\x}{\bar{x}}
\newcommand{\X}{\overline{X}}

\DeclareMathOperator{\rk}{rk}
\DeclareMathOperator{\rev}{rev}
\DeclareMathOperator{\inv}{inv}

\newtheorem{thm}{Theorem}[section]
\newtheorem{prop}[thm]{Proposition}
\newtheorem{cor}[thm]{Corollary}
\newtheorem{lem}[thm]{Lemma}
\newtheorem{defn}[thm]{Definition}
\newtheorem{rem}[thm]{Remark}

\newtheorem{eg}[thm]{Example}

\title
{Factorization of classical characters twisted by roots of unity: \upperRomannumeral{2}}

\subjclass[2010]{20G05, 20G20, 05A15, 05E05, 05E10}
\keywords{classical groups, Weyl character formula, twisted characters, factorizations, $(z_1,z_2,k)$-asymmetric partitions, generating functions}

\author{Nishu Kumari}
\address{Nishu Kumari, Department of Mathematics, 
Indian Institute of Science, Bangalore  560012, India.}
\email{nishukumari@iisc.ac.in}

\date{\today}

\begin{document}
\begin{abstract}
Fix natural numbers $n \geq 1$, $t \geq 2$ and a primitive $t^{\text{th}}$ root of unity $\omega$. In previous work with A. Ayyer (J. Alg., 2022), we studied the factorization of specialized irreducible characters of  $\GL_{tn}$, $\text{SO}_{2tn+1},$ $  \text{Sp}_{2tn}$ and $\text{O}_{2tn}$ evaluated at elements to $\omega^j x_i$ for $0 \leq j \leq t-1$ and $1 \leq i \leq n$. In this work, we extend the results to the groups $\GL_{tn+m}$ $(0 \leq m \leq t-1)$, $\OO_{2tn+3}$, $\Sp_{2tn+2}$ and $\OE_{2tn+2}$ evaluated at similar specializations: (1) for the $\GL_{tn+m}(\mathbb{C})$ case, we set the first $tn$ elements to $\omega^j x_i$ for $0 \leq j \leq t-1$ and $1 \leq i \leq n$ and the remaining $m$ to $y, \omega y, \dots, \omega^{m-1} y$; (2) for the other three families, the same specializations but with $m=1$. The main results of this paper are a characterization of partitions for which these characters vanish and a factorization of nonzero characters into those of smaller classical groups. Our motivation is the conjectures of Wagh and Prasad (Manuscripta Math., 2020) relating the irreducible representations of $\text{Spin}_{2n+1}$ and $\text{SL}_{2n}$, $\text{SL}_{2n+1}$ and $\Sp_{2n}$ as well as $\text{Spin}_{2n+2}$ and $\Sp_{2n}$. Our proofs use the Weyl character formulas and the beta-sets of $t$-core partitions. Lastly, we give a bijection to prove that there infinitely many $t$-core partitions for which these characters are nonzero. 
\end{abstract}

\maketitle 

\section{introduction}
The irreducible characters of classical Lie groups, namely the general linear, symplectic and orthogonal groups play a prominent role in representation theory. They have been the objects of active research, and numerous results and applications have been established for them. Some of the combinatorial applications are enumerating plane partitions and alternating sign matrices~\cite{okada2020intermediate, ayyer-behrend-2019,CiuKra09}. For background, see \cite{FulHar91}.

The specialized characters satisfy nontrivial relations which are not well-understood from the point of view of representation theory~\cite{CiuKra09,ayyer-fischer-2020,AyyBehFis16}. 
In previous work with Ayyer~\cite{ayyer-2021}, we
considered the irreducible classical characters in $tn$ variables, for $t \geq 2$ a fixed positive integer, specialized to $(\exp(2 \pi \iota k/t) x_j)_{0 \leq k \leq t-1, 1 \leq j \leq n}$, motivated by the works of Littlewood~\cite{littlewood-1950} and  Prasad~\cite{prasad-2016}.
% We also looked at $\GL_{tn+1}$ where we specialize the elements as before and set the last variable to 1.
We showed that such a character is nonzero if and only if the corresponding $t$-core is of some special form, and if it is nonzero, it factors into the characters of smaller groups. 
% They further generalized this result to $tn$ variables,   obtaining similar results. 
% We will think of these as twisted characters, where the twists are by all the $t$'th roots of unity.

% We note in passing that Schur polynomials evaluated at roots of unity and their powers have been considered in \cite{macdonald-2015,rhoades-2010}.

In this work, we first consider the characters of $\GL_{tn+m}$, where we specialize the first $tn$ variables as before and the last $m$ variables as
%add $m$ extra variables
$y,\dots,\omega^{m-1} y$. (See \cref{thm:schur-k}).
We then consider the characters of classical groups $\OO_{2tn+3}$, $\Sp_{2tn+2}$ and $\OE_{2tn+2}$, where an extra variable $y$ is added. In each case, we characterize partitions for which the character value is nonzero in terms of what we call $(z_1,z_2,k)$-asymmetric partitions, where $z_1$, $z_2$ and $k$ are integers which depend on the group.
%We show that the characters are nonzero if and only if the $t$-core of the indexed partition is of special form which we call $(z_1,z_2,k)$-asymmetric partitions. 
(See \cref{sec:summary} for the definitions). These are stated as \cref{thm:odd}, \cref{thm:symp} and \cref{thm:even} respectively. 
%Our proofs are more involved for obvious reasons. 
% For the general linear group, there are finitely many $t$-core partitions for which the twisted character is nonzero, namely partitions of length at most $m$ and first part at most $t-m$. For the other classical characters, there are infinitely many  $t$-core partitions for which the character is nonzero. We will show that these are $t$-cores which can be written in Frobenius coordinates as $(\alpha | \alpha + z)$, where the value of $z$ depends on the group, and which we call $z$-asymmetric partitions. 

 The plan for the rest of the paper is as follows.
 We give all the definitions, statements of our results and illustrative examples in \cref{sec:summary}.
 We give uniform proofs of the factorization  of characters of classical groups of type $B$, $C$ and $D$. To that end, we formulate results on beta sets, generating functions and determinant identities in \cref{sec:back}.
 We prove the Schur factorization result in \cref{sec:schur-k}. 
 We prove the new factorizations of classical characters in \cref{sec:other}. We give all the details for the factorization of characters of group of type $B$, and are much more sketchy about the type $C$ and $D$ character factorization.
 Finally, we prove generating function formulas for $(z_1,z_2,k)$-asymmetric partitions
and $(z_1,z_2,k)$-asymmetric $t$-cores in \cref{sec:gf}. In particular, we will prove \cref{thm:inf-cores} there, showing that there are infinitely many $t$-cores at which the specialized character values are nonzero.

\section{Summary of the main Results}
\label{sec:summary}
A {\em partition} $\lambda = (\lambda_1,\dots,\lambda_{\ell})$ 
% with parts $\lambda_i$ 
is a weakly decreasing sequence of nonnegative integers with all but finite number of parts are zero. 
The number of nonzero parts of $\lambda$ is the {\em length} $\ell(\lambda)$ of $\lambda$. For a partition $\lambda$ of length at most $\ell$,
define the {\em beta-set of $\lambda$} by $\beta(\lambda) \equiv \beta(\lambda,\ell) \coloneqq (\beta_1(\lambda,\ell),\dots,\beta_{\ell}(\lambda,\ell))$ by $\beta_i(\lambda,\ell) \coloneqq \lambda_i+\ell-i$. We will simplify the notation $\beta(\lambda,\ell)$ as $\beta(\lambda)$ whenever $\ell$ is clear from the context.
We use the shorthand notation $a + \lambda$ for the partition $(a + \lambda_1,\dots, a + \lambda_{\ell})$ for $a \geq 0$.

Throughout fix an integer $t \geq 2$ and a primitive $t^{\text{th}}$ root of unity $\omega$, i.e. $\omega^t = 1$.
Let $X = (x_1,\dots,x_n)$ be a tuple of commuting indeterminates for a fixed positive integer $n$. 
Define $\x \coloneqq 1/x$ for an indeterminate $x$ and write $\X = (\x_1,\dots,\x_n)$.
For $j \in \mathbb{Z}$, we set $X^j = (x_1^j,\dots,x_n^j)$, and for $a \in \mathbb{R}$, set $aX = (a x_1,\dots,a x_n)$.

We begin by recalling the Weyl character formulas of all the classical groups at the representation indexed by $\lambda=(\lambda_1,\ldots,\lambda_n)$. (See ~\cite{FulHar91} for more details).
The \emph{Schur polynomial} or the \emph{general linear (type A) character} of $\GL_n$ is defined as
\begin{equation}
\label{gldef}
s_\lambda(X) \coloneqq \frac{\displaystyle\det_{1\le i,j\le n}\Bigl(x_i^{\beta_j(\lambda)}\Bigr)}
{\displaystyle\det_{1\le i,j\le n}\Bigl(x_i^{n - j}\Bigr)}.
\end{equation}
% The denominator is the standard Vandermonde determinant,
% \begin{equation}
% \label{gldenom}
% \det_{1\le i,j\le n}\Bigl(x_i^{n - j}\Bigr)
% = \prod_{1\le i<j\le n}(x_i-x_j).
% \end{equation}
The \emph{odd orthogonal (type B) character} of the group $\OO(2n+1)$ is defined as
\begin{equation}
\label{oodef}
\oo_\lambda(X) \coloneqq
\frac{\ds \det_{1\le i,j\le n}\Bigl(x_i^{\beta_j(\lambda)+1/2}-\x_i^{\beta_j(\lambda)+1/2}\Bigr)}
{\ds \det_{1\le i,j\le n}\Bigl(x_i^{n-j+1/2}-\x_i^{n-j+1/2}\Bigr)} 
= \frac{\ds \det_{1\le i,j\le n}\Bigl(x_i^{\beta_j(\lambda)+1}-\x_i^{\beta_j(\lambda)}\Bigr)}
{\ds \det_{1\le i,j\le n}\Bigl(x_i^{n-j+1}-\x_i^{n-j}\Bigr)}.
\end{equation}
% and the denominator here is
% \begin{equation}
% \label{oodenom}
% \det_{1\le i,j\le n}\Bigl(x_i^{n-j+1/2}-\x_i^{n-j+1/2}\Bigr)
% = \prod_{i=1}^n\bigl(x_i^{1/2}-\x_i^{1/2}\bigr)\,\prod_{1\le i<j\le n}(x_i+\x_i-x_j-\x_j).
% \end{equation}
The \emph{symplectic (type C) character} of the group $\Sp(2n)$ is defined as
\begin{equation}
\label{spdef}
\sp_\lambda(X) \coloneqq
\frac{\ds \det_{1\le i,j\le n}\Bigl(x_i^{\beta_j(\lambda)+1}-\x_i^{\beta_j(\lambda)+1}\Bigr)}
{\ds \det_{1\le i,j\le n}\Bigl(x_i^{n-j+1}-\x_i^{n-j+1}\Bigr)}.
\end{equation}
% and the denominator here is
% \begin{equation}
% \label{spdenom}
% \det_{1\le i,j\le n}\Bigl(x_i^{n-j+1}-\x_i^{n-j+1}\Bigr)
% = \prod_{i=1}^n(x_i-\x_i)\,\prod_{1\le i<j\le n}(x_i+\x_i-x_j-\x_j).
% \end{equation}
Lastly, the \emph{even orthogonal (type D) character} of the group $\OE(2n)$ is defined as
\begin{equation}
\label{oedef}
\oe_\lambda(X) \coloneqq
\frac{\ds 2 \det_{1\le i,j\le n}\Bigl(x_i^{\beta_j(\lambda)}+\x_i^{\beta_j(\lambda)}\Bigr)}
{\displaystyle (1+\delta_{\lambda_n,0})
\det_{1\le i,j\le n}\Bigl(x_i^{n-j}+\x_i^{n-j}\Bigr)},
\end{equation}
where $\delta$ is the Kronecker delta.
% The determinant here factorizes as
% \begin{equation}
% \label{oedenom}
% \det_{1\le i,j\le n}\Bigl(x_i^{n-j}+\x_i^{n-j}\Bigr)
% = \prod_{1\le i<j\le n}(x_i+\x_i-x_j-\x_j).
% \end{equation}
We note the trivial cases: 
\[
s_\lambda(x_1,\ldots,x_n) = \sp_\lambda(x_1,\allowbreak\dots,x_n)  = \oo_\lambda(x_1,\dots,x_n) = \oe_\lambda(x_1,\dots,x_n) = 0, \quad
\text{if } n < \ell(\lambda).
\]
We also note a general relation between even and odd orthogonal characters.
For a partition $ \lambda = (\lambda_1,\dots,\lambda_n)$ of length at most $n$, we have
\begin{equation}
 \label{oeshifted}
\oe_{(\lambda_1+1/2,\ldots,\lambda_n+1/2)}(x_1,\ldots,x_n)=
(-1)^{\sum_{i=1}^n\lambda_i}
\prod_{i=1}^n\bigl(x_i^{1/2}+\x_i^{1/2}\bigr)
\:\oo_\lambda(-x_1,\ldots,-x_n).   
\end{equation}

We recall Macdonald's definition of the $t$-core and $t$-quotient of a partition after setting up some notations.
For a partition $\lambda$ of length at most $\ell$ and $i \in [0,t-1]$,
let $n_{i}(\lambda) \equiv n_{i}(\lambda,m)$ be the number of parts of $\beta(\lambda)$ congruent to $i \pmod{t}$, and  $\beta_j^{(i)}(\lambda)$ for $1 \leq j \leq n_{i}(\lambda)$ be the $n_{i}(\lambda)$ parts of $\beta(\lambda)$ that are congruent to $i \pmod t$ in decreasing order.

\begin{defn}[{\cite[Example~I.1.8]{macdonald-2015}}]
\label{prop:mcd-t-core-quo}
Let $\lambda$ be a partition of length at most $\ell$.

\begin{enumerate}
\item The $\ell$ numbers $tj+i$, where $0 \leq j \leq n_{i}(\lambda)$ and $0 \leq i \leq t-1$, are all distinct. Arrange them in descending order, say $\tilde{\beta}_1>\dots>\tilde{\beta}_\ell$. Then the $t$-core of $\lambda$ has parts $(\core \lambda t)_i=\tilde{\beta}_i-\ell+i$.  
Thus, $\lambda$ is a $t$-core if and only if these $\ell$ numbers $tj+i$, where $0 \leq j \leq n_{i}(\lambda)$ and $0 \leq i \leq t-1$ form its beta-set $\beta(\lambda)$.

\item The parts $\beta_j^{(i)}(\lambda)$ may be written in the form $t\tilde{\beta}_j^{(i)}+i$, $1 \leq j \leq n_{i}(\lambda)$,
where $\tilde{\beta}_1^{(i)}>\dots>\tilde{\beta}_{n_{i}(\lambda)}^{(i)}\geq 0$. 
Let $\lambda_j^{(i)}=\tilde{\beta}_j^{(i)}-n_{i}(\lambda)+j$, 
so that $\lambda^{(i)}=(\lambda_1^{(i)},\dots,\lambda_{n_{i}(\lambda)}^{(i)})$ is a partition. 
Then the $t$-quotient $\quo \lambda t$ of $\lambda$ is a cyclic permutation of $\lambda^{\star} = (\lambda^{(0)},\lambda^{(1)},\dots,\lambda^{(t-1)})$.
The effect of changing $\ell \geq \ell(\lambda)$ is to permute the $\lambda^{(j)}$ cyclically, so that $\lambda^{\star}$ should perhaps be thought of as a `necklace' of partitions.
\end{enumerate}
\end{defn}

% For a partition $\lambda$ of length at most $tn+1$, 
% let $\sigma_{\lambda} \in S_{tn+1}$ be the permutation that rearranges the parts of $\beta(\lambda)$ as follows:
% \begin{equation}
% % \label{sigma-perm}
%  \beta_{\sigma_{\lambda}(j)}(\lambda) \equiv q \pmod t, \quad  \text{ for } \quad \sum_{i=0}^{q-1} n_{i}(\lambda)+1 \leq j \leq \sum_{i=0}^{q} n_{i}(\lambda),   
% \end{equation}
% arranged in decreasing order for each $q \in \{0,1,\dots,t-1\}$.
% For the empty partition,  $\beta(\emptyset,tn+1)=(tn,tn-1,tn-2,\dots,0)$ with
% \[ n_{q}(\emptyset,tn+1) = 
% \begin{cases}
% n+1 & q=0,\\
% n & 1 \leq q \leq t-1,
% \end{cases}
% \]
% and 
% \begin{equation}
% % \label{sigma0}
% \sigma_{\emptyset}={\displaystyle (1,t+1, \dots , nt+1, t, \dots ,nt, \dots,2, \dots ,(n-1)t+2}),
% \end{equation}
% in one-line notation with $\sgn(\sigma_{\emptyset})=(-1)^{\frac{t(t-1)}{2}\frac{n(n+1)}{2}}$.

% Fix $0 \leq m \leq t-1$. 
Let $E = (e_1, \dots, e_m)$ such that $t-1 \geq e_1 > \cdots > e_m \geq 0$. We extend $E$ for enumerating the set $\{0, \ldots, t-1\} \setminus \{e_i\}_{i\in [m]}$ as $\{e_{m+1}< \cdots < e_t\}$, denoted $\bar{E}$, for convenience.
For a partition $\lambda$ of length at most $tn+m$, let $\sigma_{\lambda}^{E}$ be the permutation in $S_{tn+m}$ such that it rearranges parts of $\beta(\lambda)$ in the following way: 

\begin{equation}
\label{sigma-perm-m}
 \beta_{\sigma^E_{\lambda}(j)}(\lambda) \equiv e_q \pmod t, \quad  \text{ for } \quad \sum_{i=1}^{{q-1}} n_{e_i}(\lambda)+1 \leq j \leq \sum_{i=1}^{q} n_{e_i}(\lambda),   
\end{equation}
arranged in decreasing order for each $q \in \{1,\dots,t\}$. For simplicity, we write $\sigma_{\lambda}^{\emptyset}$ as $\sigma_{\lambda}$.
For the empty partition,  $\beta(\emptyset,tn+1)=(tn,tn-1,tn-2,\dots,0)$ with
\[ n_{q}(\emptyset,tn+1) = 
\begin{cases}
n+1 & q=0,\\
n & 1 \leq q \leq t-1,
\end{cases}
\]
and 
\begin{equation}
\label{sigma0m}
\sigma_{\emptyset}={\displaystyle (1,t+1, \dots , nt+1, t, \dots ,nt, \dots,2, \dots ,(n-1)t+2})
\end{equation}
in one-line notation with $\sgn(\sigma_{\emptyset})=(-1)^{\frac{t(t-1)}{2}\frac{n(n+1)}{2}}$.

Recall, $X = (x_1,\dots,x_n)$ and $\omega$ is a primitive $t$'th root of unity. Fix $0 \leq m \leq t-1$. We first consider the specialized Schur polynomial evaluated at elements twisted by the $t$'th roots of unity. We denote the indeterminates by $X,\omega X,\omega^2X, \dots ,\omega^{t-1}X$, $y, \dots , \omega^{m-1} y$.  

\begin{thm}
\label{thm:schur-k}
Fix $0 \leq m \leq t-1$. Let $\lambda$ be a partition of length at most $tn+m$ indexing an irreducible representation of $\GL_{tn+m}$ and $\quo \lambda t = (\lambda^{(0)},\dots,\lambda^{(t-1)})$. Then the $\GL_{tn+m}$-character $s_{\lambda}(X, \omega X,\dots,\omega^{t-1} X, y,\omega y, \dots, \omega^{m-1}y )$ is as follows:
\begin{enumerate}
    \item If $\core \lambda t= \nu \coloneqq (\nu_1,\dots,\nu_m)$ for some $\nu_1\leq t-m$, then
    \begin{equation}
    \label{schur-k}
        \begin{split}
       s_{\lambda}(X, \omega X,\dots,\omega^{t-1} X,  y,\omega y, \dots, & \omega^{m-1}y) 
       =  \sgn(\sigma_{\lambda}^{\beta(\nu)})
       \sgn(\sigma_{\emptyset}^{\beta(\nu)}) \\
    \times    s_{\nu}(1,\omega,&\dots,\omega^{m-1})
       \prod_{i=1}^{m} s_{\lambda^{(\beta_i(\nu))}} (X^t,y^t) \prod_{\substack{j=0\\ j \neq \beta_i(\nu),1 \leq i \leq m}}^{t-1}  s_{\lambda^{(j)}} (X^t). % C+\delta_m = \beta(C)    
        \end{split}
    \end{equation}
    \item Otherwise,
    \begin{equation}
  s_{\lambda}(X, \omega X,\dots,\omega^{t-1} X, y,\omega y, \dots, \omega^{m-1}y) = 0.
    \end{equation}
\end{enumerate}
\end{thm}

In other words, the nonzero $\GL_{tn+m}$ character is the product of $m$ $\GL_{n+1}$ and $(t-m)$ $\GL_{n}$ characters.
For $m=0$ and $m=1$, \cref{thm:schur-k} is proved by Littlewood~\cite[Equation (7.3;3)]{littlewood-1950},~\cite[Chapter VII, Section IX]{littlewood-1950}  and independently by Prasad~\cite[Theorem 2]{prasad-2016},~\cite[Theorem 4.5]{LuPra21} for $t=2$. For $m=0$ and $m=1$ the result is also proven in~\cite{ayyer-2021}. 

% \begin{eg}
% For $t=3$, $m=2$, \cref{thm:schur-k} says that 
% the irreducible character of $\GL_5$ indexed by $\lambda = (a,b,c,d,e)$ is non-zero if and only if $\core \lambda 3$ is $\emptyset$, $(1)$, $(1,1)$. 
% \end{eg}
\begin{eg}
\label{eg:schur}
For $t=2$, $m=1$, \cref{thm:schur-k} 
says that the character of the group $\GL_{3}$ of the representation indexed by the partition $(a,b,c)$, $a \geq b \geq c \geq 0$, evaluated at $(x,-x,y)$ is non-zero if and only if $a$ and $b$ have the same parity or $a$ and $c$ have the opposite parity. If $\core {a,b,c} 2$ is empty, then
\[
s_{(a,b,c)}(x,-x,
y)=
\begin{cases}
-s_{(\frac{a}{2},\frac{b+1}{2})}(x^2,y^2) s_{(\frac{c-1}{2})}(x^2) & \text{$a$  even, $b$ and $c$ odd,}\\
s_{(\frac{a}{2},\frac{c}{2})}(x^2,y^2) s_{(\frac{b}{2})}(x^2) & a,b,c \text{ even,}\\
-s_{(\frac{b-1}{2},\frac{c}{2})}(x^2,y^2) s_{(\frac{a+1}{2})}(x^2) & \text{$a$ and $b$ odd, $c$ even,}
\end{cases}
\]
and if $\core {a,b,c} 2 = (1)$, then
\[
s_{(a,b,c)}(x,-x,y)=
\begin{cases}
y \, s_{(\frac{a-1}{2},\frac{b}{2})}(x^2,y^2) s_{(\frac{c}{2})}(x^2) & \text{$a$ odd, $b$ and $c$ even,}\\
-y \, s_{(\frac{a-1}{2},\frac{c-1}{2})}(x^2,y^2) s_{(\frac{b+1}{2})}(x^2) & a,b,c \text{ odd,}\\
y\,  s_{(\frac{b-2}{2},\frac{c-1}{2})}(x^2,y^2) s_{(\frac{a+2}{2})}(x^2) & \text{ $a$ and $b$ even, $c$ odd.}
\end{cases}
\] 
\end{eg}

We now generalize \cref{thm:schur-k} to other classical characters for $m=1$. We first need some definitions.
The {\em Young diagram} of a partition $\lambda$ is its pictorial representation whose $i^{\text{th}}$ row has $\lambda_i$ left-justified boxes. The {\em conjugate of a partition $\lambda$} is the partition $\lambda'$ obtained by transposing the Young diagram of $\lambda$.  For example, the Young diagram of $\lambda=(4,2,1)$ and its conjugate $\lambda'=(3,2,1,1)$ are 
\begin{figure}[H]
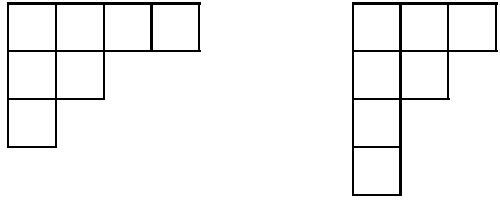

    \centering
    \begin{subfigure}[b]{0.25\textwidth}
         \centering
    \ydiagram{4,2,1} 
     % \caption{$\lambda=(4,2,1)$} 
    \end{subfigure}
     % \hfill
    \begin{subfigure}[b]{0.25\textwidth}
         \centering
    % \qquad \qquad 
\ydiagram{3,2,1,1}
% \caption{$\lambda'=(3,2,1,1)$}
\end{subfigure}
    \caption{$\lambda=(4,2,1)$  and $\lambda'=(3,2,1,1)$}
    \label{fig:my_label}
\end{figure}
% \[
% \ydiagram{4,2,1}
% \qquad \qquad 
% \ydiagram{3,2,1,1} 
% \]
The largest integer $j$ such that
$\lambda_j \geq j$ is the ($\text{\em Frobenius}$) rank of the partition $\lambda$, denoted by $\rk(\lambda)$. The {\em Frobenius coordinates} of a partition $\lambda$, denoted $(\alpha|\beta)$, is a pair of strict partitions of length at most $\rk(\lambda)$ given by $\alpha_i=\lambda_i-i$ and  $\beta_j=\lambda'_j-j$.

\begin{defn}
\label{defn:asym}
Suppose $z_1 > z_2 \geq 0$ and $\lambda$ is a partition of rank $r$. We say $\lambda$ is {\em $(z_1,z_2,k)$-asymmetric} for some $0 \leq k \leq r$, if $\lambda=(\alpha_1,\dots,\alpha_k,\dots,\alpha_r|\alpha_1+z_1,\dots,\widehat{\alpha_k+z_1},\dots,\alpha_r+z_1,z_2)
$, in Frobenius coordinates for some strict partition $\alpha$, where a hat on a coordinate denotes its omission. $($Here $k=0$ means no part is omitted and therefore no part is added$)$. 
If in addition a
$(z_1,z_2,k)$-asymmetric partition is also a $t$-core, we call it a {\em $(z_1,z_2,k)$-asymmetric $t$-core.} We denote the set of $(z_1,z_2,k)$-asymmetric partitions and $(z_1,z_2,k)$-asymmetric $t$-cores by $\mathcal{Q}_{z_1,z_2,k}$ and $\mathcal{Q}^{(t)}_{z_1,z_2,k}$ respectively.
\end{defn}

Note that the $(z,0,0)$-asymmetric partition is the $z$-asymmetric partition defined in~\cite[Definition 2.9]{ayyer-2021}. Recall that a partition $\lambda$ is $z$-asymmetric if $\lambda=(\alpha|\beta)$ where $\beta_i=\alpha_i+z$ for $1 \leq i \leq \rk(\lambda)$. 

To state our results, define, for $\lambda = (\lambda_1,\dots,\lambda_n)$, the reverse of $\lambda$ as $\rev(\lambda) = (\lambda_n,\dots,\lambda_1)$. Moreover, if $\mu = (\mu_1,\dots,\mu_j)$ is a partition such that $\mu_1 \leq \lambda_n$, then we write the concatenated partition $(\lambda,\mu) = (\lambda_1,\dots,\lambda_n,\mu_1,\dots,\mu_j)$. 
% Recall that the {\em $q$-number} of $n$, denoted $[n]_q$ is $\frac{q^n-1}{q-1}$.

For the odd orthogonal case, we take $G = \OO_{2tn+3}$, the orthogonal group of $(2tn+3) \times (2tn+3)$ square matrices. For a partition $\lambda$,
% is a partition of length at most $tn+1$ and 
if $\core \lambda t$ is either $(2,0,k)$- or $(2,1,k)$-asymmetric  
    % $\in \mathcal{Q}^{(t)}_{2,0,k} \cup \mathcal{Q}^{(t)}_{2,1,k}$ 
    for some $k \in [\rk(\core \lambda t)]$, then by \cref{cor:z=0}, there exists a unique $i_0 \in [0, \floor{\frac{t-1}{2}}]$ such that 
    \begin{equation*}
    n_{i}(\lambda)+n_{t-1-i}(\lambda) =
    \begin{cases}
    2n+1+\delta_{i_0,\frac{t-1}{2}} & \text{ if } i=i_0,\\
    2n & \text{ otherwise},
    \end{cases}
    \quad  \quad  0 \leq i \leq  \floor{\frac{t-1}{2}}.
\end{equation*}
    For such a partition $\lambda$, let
\begin{equation}
\label{epsilon}
\epsilon_1(\lambda) \coloneqq
\left(\ds 
\sum_{i=t-i_0}^{t-1} n_{i}(\lambda) \right)
+ \ds \sum_{i=\floor{\frac{t+1}{2}}}^{t-1} \left( \binom{n_{i}(\lambda)+1}{2} 
+ t n (n_{i}(\lambda)-n) \right),
% \epsilon_1(\lambda) \coloneqq
% \left(\ds 
% \sum_{i=t-i_0}^{t-1} n_{i}(\lambda) \right)
% + \ds \sum_{i=\floor{\frac{t+1}{2}}}^{t-1} 
% t n (n_{i}(\lambda)-n) 
% + \sum_{q=0}^{\floor{\frac{t-2}{2}}} \left( \frac{n_{t-1-q}(\lambda)
% (n_{t-1-q}(\lambda)+1)}
% {2} \right)
\end{equation}
and $\ds \pi^{(1)}_i \coloneqq  \lambda^{(t-1-i)}_1 +
\left(\lambda^{(i)}, 0,\dots,0, -\rev(\lambda^{(t-1-i)})\right)$ has $2n+\delta_{i,i_0}$ parts for 
$0 \leq i \leq \floor{\frac{t-1}{2}}$. We note that the empty partition is vacuously $(2,0,0)$-asymmetric with $i_0=0$. Our result for the factorization of odd orthogonal characters is as follows.
\begin{thm}
\label{thm:odd}
Let $\lambda$ be a partition of length at most $tn+1$.
% and $\core \lambda t=(\alpha|\beta)$, in Frobenius coordinates. 
Then the odd orthogonal character $\oo_{\lambda}(X,\, \omega X, \dots,\omega^{t-1}X,y)$ is as follows:
\begin{enumerate}
%     \item If $\core \lambda t$ is empty, then 
%     \begin{equation}
%     \begin{split}
%         \oo_{\lambda}( X,\, \omega X, \dots, & \omega^{t-1}X,y)  \\
%         =  \sgn(\sigma_{\lambda}) & \frac{(-1)^{\epsilon_1} }{(y-1)}
%         \Big( y^{-t(\lambda_1^{(t-1)})+1}  s_{\pi_{0}^{(1)}}(X^t,\X^t,y^t)  -
% y^{t(\lambda_1^{(t-1)})} 
% s_{\pi_{0}^{(1)}}(X^t,\X^t, \bar{y}^t) \Big) \\
%  \times  &
%  \prod_{i=1}^{\floor{\frac{t-2}{2}}} 
%  s_{\pi^{(1)}_i} (X^t,\X^t) 
%         \times \begin{cases}
%      \oo_{\lambda^{\left(\frac{t-1}{2}\right)}}(X^t)   
%      & t \text{ is odd},\\
% 1 & t \text{ is even}.
%         \end{cases}
%         \end{split}
%     \end{equation}
    \item If $\core \lambda t$ is either $(2,0,k)$ or $(2,1,k)$-asymmetric 
    % $\in \mathcal{Q}^{(t)}_{2,0,k} \cup \mathcal{Q}^{(t)}_{2,1,k}$ 
    for some $k \in [\rk(\core \lambda t)]$ and $i_0=\frac{t-1}{2}$,
    % $\alpha_k+1 \equiv \frac{t-1}{2} \pmod{t}$, 
    then 
    \begin{equation}
    \begin{split}
        \oo_{\lambda}(X,\, \omega X, \dots,&\omega^{t-1}X,y)\\
        = (-1)^{\epsilon_1(\lambda)+n} & \sgn(\sigma_{\lambda})
        \oo_{\left(\frac{t-1}{2}\right)}(y)
        % \frac{(y^t-1)}{y^{(t-1)/2}  (y-1)}
        \oo_{\lambda^{\left(\frac{t-1}{2}
        \right)}}(X^t,y^t) \times  \prod_{i=0}^{\frac{t-3}{2}} s_{\pi^{(1)}_i} (X^t,\X^t).
        \end{split}
    \end{equation}
    \item If $\core \lambda t$ is either $(2,0,k)$ or $(2,1,k)$-asymmetric 
    % =(\alpha|\beta) \in \mathcal{Q}_{2,0,k} \cup \mathcal{Q}_{2,1,k}$ 
    for some $k \in [\rk(\core \lambda t)]$ and 
    % $\alpha_k+1 \not \equiv \frac{t-1}{2} \pmod{t}$
    $i_0 \not = \frac{t-1}{2}$, then 
    \begin{equation}
    \label{odd}
    \begin{split}
        \oo_{\lambda}( X,\, \omega X, \dots,  \omega^{t-1}X,y) &\\ 
        =  
        \sgn(\sigma_{\lambda}) (-1)^{\epsilon_1(\lambda)} &
        \frac{ \Big( y^{-\mu^{(1)}_{i_0}+1}   s_{\pi_{i_0}^{(1)}}(X^t,\X^t,y^t)  -
y^{\mu^{(1)}_{i_0}} 
s_{\pi_{i_0}^{(1)}}(X^t,\X^t, \bar{y}^t) \Big)  }{(y-1)} \\
 \times & \prod_{\substack{i=0\\i \neq i_0}}^{\floor{\frac{t-2}{2}}} s_{\pi^{(1)}_i} (X^t,\X^t) 
        \times \begin{cases}
     \oo_{\lambda^{\left(\frac{t-1}{2}\right)}}(X^t)   & t \text{ odd},\\
1 & t \text{ even},
        \end{cases} 
        \end{split}
    \end{equation}
where $\mu^{(1)}_{i_0}=t(\lambda_1^{(t-1-i_0)}
+n_{(t-1-i_0)}(\lambda)-n)-i_0$.
% \[
% i_0 = \begin{cases}
% (\alpha_k+1) \pmod{t}  & \text{if }  \in [0,\floor{\frac{t-1}{2}}],\\
% t-1-(\alpha_k+1) \pmod{t}  & \text{otherwise}.
% \end{cases}
% \]  
\item If neither of the above conditions hold, then
\begin{equation}
    \oo_{\lambda}( X,\, \omega X, \dots,  \omega^{t-1}X,y) =0.
\end{equation}
\end{enumerate}
\end{thm}
\begin{rem}
    We note that the first factor on the right side of \eqref{odd} is a Laurent polynomial and  approaches to $(1-{2 \mu^{(1)}_{i_0}}) s_{\pi_{i_0}^{(1)}}(X^t,\X^t,1)$ as $y \to 1$.
    % $-y^{-\mu^{(1)}_{i_0}+1}(1+y+\dots+y^{2 \mu^{(1)}_{i_0} -2}) s_{\pi_{i_0}^{(1)}}(X^t,\X^t,1)$.
\end{rem}
% \begin{rem}
% The factorization of $\oo_{\lambda}( X,\, \omega X, \dots,  \omega^{t-1}X)$ was considered in~\cite{ayyer-2021}. We will not recover the result~\cite[Theorem 2.17]{ayyer-2021} by substituting $y=0$ as these are not universal characters. 
% \end{rem}
% Note: Either $\core \lambda 2$ is empty or 
% \[
% \core \lambda 2 = \begin{cases}
% (m,m-2,\dots,0|m,m-2,\dots,0) & \text{if } m \text{ is even},\\
% (m,m-2,\dots,1|m,m-2,\dots,1) & \text{if } m \text{ is odd}.
% \end{cases}
% \]
% So, either $(2,0,1)$ or $(2,1,1)$-asymmetric.
\begin{eg}
For $t=2$, \cref{thm:odd} says that the character of the group $\OO_{7}$ of the representation indexed by the partition $(a,b,c)$, $a \geq b \geq c \geq 0$, evaluated at $(x,-x,y)$ is non-zero.
If $\core {a,b,c} 2$ is empty, which is $(2,0,0)$-asymmetric and $i_0=0$, then
\begin{multline*}
\oo_{(a,b,c)}(x,-x,y)\\
= 
\begin{cases}
-\frac{y^{2-c}}{y-1}
s_{\left(\frac{a+c-1}{2},\frac{b+c}{2}\right)}
(x^2,\x^2,y^2)
+\frac{y^{c-1}}{y-1} s_{\left(\frac{a+c-1}{2},\frac{b+c}{2}\right)}
(x^2,\x^2,\bar{y}^2) 
& \text{$a$  even, $b$ and $c$ odd,}\\
\frac{y^{1-b}}{y-1}
s_{\left(\frac{a+b}{2},\frac{b+c}{2}\right)}
(x^2,\x^2,y^2)
-\frac{y^{b}}{y-1} s_{\left(\frac{a+b}{2},\frac{b+c}{2}\right)}
(x^2,\x^2,\bar{y}^2) & a,b,c \text{ even,}\\
-\frac{y^{-a}}{y-1}
s_{\left(\frac{a+b}{2},\frac{a+c+1}{2}\right)}
(x^2,\x^2,y^2)
+\frac{y^{a+1}}{y-1} s_{\left(\frac{a+b}{2},\frac{a+c+1}{2}\right)}
(x^2,\x^2,\bar{y}^2) 
& \text{$a$ and $b$ odd, $c$ even}.
\end{cases}
\end{multline*}
If $\core {a,b,c} 2 = (1)=(0|0)$, which is $(2,0,1)$-asymmetric and $i_0=1$, then
\begin{multline*}
\oo_{(a,b,c)}(x,-x,y) \\= 
\begin{cases}
-\frac{y^{-a}}{y-1} \, s_{(\frac{a+c-1}{2},\frac{a-b-1}{2})}
(x^2,\x^2,y^2) 
+ \frac{y^{a+1}}{y-1} 
s_{(\frac{a+c-1}{2},\frac{a-b-1}{2})}
(x^2,\x^2,\bar{y}^2)
& \text{$a$ odd, $b$ and $c$ even,}  \\
\frac{y^{-a}}{y-1} \, s_{(\frac{a+b}{2},\frac{a-c}{2})}
(x^2,\x^2,y^2) -
\frac{y^{a+1}}{y-1} s_{(\frac{a+b}{2},\frac{a-c}{2})}
(x^2,\x^2,\bar{y}^2)
& a,b,c \text{ odd,}\\
-\frac{y^{1-b}}{y-1}  s_{(\frac{a+b}{2},\frac{b-c-1}{2})}
(x^2,\x^2,y^2) +
\frac{y^{b}}{y-1} s_{(\frac{a+b}{2},\frac{b-c-1}{2})}
(x^2,\x^2,\bar{y}^2)
& \text{ $a$ and $b$ even, $c$ odd.}
\end{cases}
\end{multline*}
If $a$ and $c$ are even, and $b$ is odd, then $\core {a,b,c} 2 = (2,1)=(1|1)$, which is $(2,1,1)$-asymmetric and $i_0=0$, and in this case,
\[
\oo_{(a,b,c)}(x,-x,y)= \frac{y^{3}}{y-1}  s_{(\frac{a-2}{2},\frac{b-1}{2},\frac{c}{2})}
(x^2,\x^2,y^2) -
\frac{y^{-2}}{y-1} s_{(\frac{a-2}{2},\frac{b-1}{2},\frac{c}{2})}
(x^2,\x^2,\bar{y}^2).
\]
Lastly, if $a$ and $c$ are odd, and $b$ is even, then $\core {a,b,c} 2 = (3,2,1)=(2,0|2,0)$, which is $(2,0,1)$-asymmetric and $i_0=1$, then
\[
\oo_{(a,b,c)}(x,-x,y)= \frac{y^{-a}}{y-1}  s_{(\frac{a-c-2}{2},\frac{a-b-1}{2})}
(x^2,\x^2,y^2) -
\frac{y^{a+1}}{y-1} s_{(\frac{a-c-2}{2},\frac{a-b-1}{2})}
(x^2,\x^2,\bar{y}^2). 
\]
\end{eg}
% \begin{eg}
% Let $\lambda$ be a partition of length at most $2n+1$. Then the odd orthogonal character $\oo_{\lambda}(X,-X,y)$ is given by
%     \begin{equation*}
%     \begin{split}
%         \oo_{\lambda}(X,-X,y) 
%         =  & \frac{(-1)^
%         {
%         \frac{n_{1}(\lambda)(n_{1}(\lambda)+1)}{2} 
% } \sgn(\sigma_{\lambda})}{(y-1)} \\ &
%      \times   \Big( y^{-2(\lambda_1^{(1)}
% +n_{1}(\lambda)-n)+1}  s_{\pi_{0}^{(1)}}(X^2,\X^2,y^2)  -
% y^{2(\lambda_1^{(1)} 
% +n_{1}(\lambda)-n)} 
% s_{\pi_{0}^{(1)}}(X^2,\X^2, \bar{y}^2) \Big),  \end{split}
%     \end{equation*}
% where 
% $\ds \pi^{(1)}_0=  \lambda^{(1)}_1 +
% \left(\lambda^{(0)}, 0, -\rev(\lambda^{(1)})\right)$ has $2n+1$ parts.
% \end{eg}

For the symplectic case, we take $G = \Sp_{2tn+2}$, the symplectic group of $(2tn+2) \times (2tn+2)$ matrices. 
If $\lambda$ is a partition such that $\core \lambda t$ is either $(3,0,k)$ or $(3,2,k)$-asymmetric  
    for some $k \in [\rk(\core \lambda t)]$, then by \cref{cor:z=1}, there exists a unique $i_0 \in [0, \floor{\frac{t-2}{2}}] \cup \{t-1\}$ such that
    \begin{equation*}
\begin{split}
    n_{i}(\lambda)+n_{t-2-i}(\lambda) =& 
    \begin{cases}
    2n+1+\delta_{i_0,\frac{t-2}{2}} & \text{ if } i=i_0,\\
    2n & \text{ otherwise},
    \end{cases}
    \quad \text{  for} \quad  0 \leq i \leq  \floor{\frac{t-2}{2}}, \\ 
    \text{and} \quad n_{t-1}(\lambda) =& \begin{cases}
    n+1 & \text{ if } i=i_0,\\
    n & \text{ otherwise}.
    \end{cases} 
\end{split}
\end{equation*}
    % such that \eqref{n_{t-1}} holds. 
    % We call it $i_0(\lambda)$. 
    For such a partition $\lambda$, let
    \begin{equation}
        \label{epsilon22}
        \epsilon_2(\lambda) \coloneqq \ds
           \sum_{i=t-i_0}^{t-1} n_{i-1}(\lambda) 
+ \ds \sum_{i=\floor{\frac{t+2}{2}}}^{t-1} \left( \binom{n_{i-1}(\lambda)+1}{2}+(t-1) n (n_{i-1}(\lambda)-n)      \right),
%         \sum_{q=0}^{\floor{\frac{t-2}{2}}}  
% \frac{n_{t-2-q}(\lambda)
% (n_{t-2-q}(\lambda)+1)}
% {2}  + \sum_{i=t-i_0}^{t-1} n_{i-1}(\lambda) 
% + \ds \sum_{i=\floor{\frac{t+2}{2}}}^{t-1} (t-1) n (n_{i-1}(\lambda)-n).
    \end{equation}
and $\ds \pi^{(2)}_i \coloneqq  \lambda^{(t-2-i)}_1 +
\left(\lambda^{(i)}, 0,\dots,0, -\rev(\lambda^{(t-2-i)})\right)$ has $2n+\delta_{i,i_0}$ parts for 
$0 \leq i \leq \floor{\frac{t-3}{2}}$. We note that the empty partition is vacuously $(3,0,0)$-asymmetric with $i_0=0$. 
\begin{thm}
\label{thm:symp}
Let $\lambda$ be a partition of length at most $tn+1$. The symplectic character $\sp_{\lambda}(X,\, \omega X, \dots,\omega^{t-1}X,y)$ is as follows:
\begin{enumerate}
% \item If $\core \lambda t$ is empty, then 
% \begin{equation}
%     \begin{split}
%         \sp_{\lambda}( X,\, \omega X, \dots,  \omega^{t-1}X,y)  
%         =  \sgn(\sigma_{\lambda}) 
%         \frac{(-1)^{\epsilon_2} }{(y-1)} & \sp_{\lambda^{(t-1)}}(X^t)  \\
%       \times  \Big( y^{-t(\lambda_1^{(t-2)})+1}  s_{\pi_{0}^{(1)}}(X^t,\X^t,y^t)  - &
% y^{t(\lambda_1^{(t-2)})} 
% s_{\pi_{0}^{(1)}}(X^t,\X^t, \bar{y}^t) \Big) \\
%  \times  \prod_{i=1}^{\floor{\frac{t-3}{2}}}
%  s_{\pi^{(2)}_i} (X^t,\X^t) 
%       &   \times \begin{cases}
%      \oo_{\lambda^{\left(\frac{t-2}{2}\right)}}(X^t)   & t \text{ is even},\\
% 1 & t \text{ is odd}.
%         \end{cases}
%         \end{split}
%     \end{equation}
    \item If $\core \lambda t$ is either $(3,0,k)$ or $(3,2,k)$-asymmetric, 
    % \in \mathcal{Q}^{(t)}_{3,0,k} \cup \mathcal{Q}^{(t)}_{3,2,k}$ 
    for some $k \in [\rk(\core \lambda t)]$ and $i_0=t-1$,
    % $\alpha_k+1 \equiv t-1 \pmod{t}$,
    then
    \begin{equation}
    \begin{split}
    \sp_{\lambda}(X,\, \omega X, \dots,\omega^{t-1}X,y)= (-1)^{\epsilon_2(\lambda)} \sgn(\sigma_{\lambda}) \, \sp_{(t-1)}(y) \,
    % \frac{(y^{2t}-1)}{y^t(y-\bar{y})} 
    & \sp_{\lambda^{(t-1)}}(X^t, y^t)
    \\ \times  \prod_{i=0}^{\floor{\frac{t-3}{2}}} 
        s_{\pi_i^{(2)}} (X^t, \X^t) &
\times 
\begin{cases}
\oo_{\lambda^{\left(\frac{t-2}{2}\right)}} (X^t) 
& t \text{ even,}\\
1 & t \text{ odd}.
\end{cases} 
\end{split}
    \end{equation}
    \item If $\core \lambda t$ is either $(3,0,k)$ or $(3,2,k)$-asymmetric, 
    % \in \mathcal{Q}^{(t)}_{3,0,k} \cup \mathcal{Q}^{(t)}_{3,2,k}$ 
    for some $k \in [\rk(\core \lambda t)]$ and $i_0=\frac{t-2}{2}$,
    % $\alpha_k+1 \equiv t-1 \pmod{t}$,
    then
    \begin{equation}
    \begin{split}
        \sp_{\lambda}(X,\, \omega X, \dots,\omega^{t-1}X,y)= (-1)^{\epsilon_2(\lambda)+n} \sgn(\sigma_{\lambda}) \,
        \sp_{\left(\frac{t-2}{2}\right)}(y) \,
        % \frac{(y^t-1)}{y^{t/2}(y-\bar{y})}  
        &
        \sp_{\lambda^{(t-1)}}(X^t) \\ \times \prod_{i=0}^{\floor{\frac{t-3}{2}}} s_{\pi_i^{(2)}} (X^t, \X^t) &
\times 
\oo_{\lambda^{\left(\frac{t-2}{2}\right)}} (X^t,y^t).
\end{split}
    \end{equation}
    \item 
    If $\core \lambda t$ is either $(3,0,k)$ or $(3,2,k)$-asymmetric, 
    % \in \mathcal{Q}^{(t)}_{3,0,k} \cup \mathcal{Q}^{(t)}_{3,2,k}$ 
    for some $k \in [\rk(\core \lambda t)]$ and $i_0 \neq t-1, \frac{t-2}{2}$,
    % $\alpha_k+1 \equiv t-1 \pmod{t}$,
    then
    % \begin{multline}
    \begin{equation}
    \label{symp}
    \begin{split}
        \sp_{\lambda}( X,\, \omega X, \dots,  \omega^{t-1}X,y) & \\
        =  \sgn(\sigma_{\lambda}) (-1)^{\epsilon_2(\lambda)} \,
      & \frac{ \Big(y^{-\mu^{(2)}_{i_0}}  s_{\pi_{i_0}^{(2)}}(X^t,\X^t,y^t)   -
y^{\mu^{(2)}_{i_0}} 
s_{\pi_{i_0}^{(2)}}(X^t,\X^t, \bar{y}^t) \Big) }{(y-\bar{y})} \\
 \times  \prod_{\substack{i=0\\i \neq i_0}}^{\floor{\frac{t-3}{2}}}
& s_{\pi^{(2)}_i} (X^t,\X^t) 
      \times  \sp_{\lambda^{(t-1)}}(X^t)     \times \begin{cases}
     \oo_{\lambda^{\left(\frac{t-2}{2}\right)}}(X^t)   & t \text{ is even},\\
1 & t \text{ is odd}, 
        \end{cases}
        % \end{multline}
        \end{split}
    \end{equation}
where $\mu^{(2)}_{i_0}=t(\lambda_1^{(t-2-i_0)}
+n_{(t-2-i_0)}(\lambda)-n)-i_0$.    
    \item If none of the above conditions hold, then
    \begin{equation}
        \sp_{\lambda}(X,\, \omega X, \dots,\omega^{t-1}X,y)=0.
    \end{equation}
\end{enumerate}
\end{thm}
\begin{rem}
    We note that the first factor on the right side of \eqref{symp} is a Laurent polynomial and  approaches to $-{\mu^{(2)}_{i_0}} s_{\pi_{i_0}^{(2)}}(X^t,\X^t,1)$ as $y \to 1$.
    % $-y^{-\mu^{(1)}_{i_0}+1}(1+y+\dots+y^{2 \mu^{(1)}_{i_0} -2}) s_{\pi_{i_0}^{(1)}}(X^t,\X^t,1)$.
\end{rem}

\begin{eg}
\label{eg:symp}
For $t=2$, \cref{thm:symp} 
says that the character of the group $\Sp_{6}$ of the representation indexed by the partition $(a,b,c)$, $a \geq b \geq c \geq 0$, evaluated at $(x,-x,y)$ is non-zero if and only if $a$ and $b$ have the same parity or $a$ and $c$ have the opposite parity same as in \cref{eg:schur}. If $\core {a,b,c} 2$ is empty, which is $(3,0,0)$-asymmetric and $i_0=0$, then
\[
\sp_{(a,b,c)}(x,-x,y)=
\begin{cases}
\oo_{(\frac{a}{2},\frac{b+1}{2})}(x^2,y^2) \sp_{(\frac{c-1}{2})}(x^2) & \text{$a$  even, $b$ and $c$ odd,}\\
\oo_{(\frac{a}{2},\frac{c}{2})}(x^2,y^2) \sp_{(\frac{b}{2})}(x^2) & a,b,c \text{ even,}\\
-\oo_{(\frac{b-1}{2},\frac{c}{2})}(x^2,y^2) \sp_{(\frac{a+1}{2})}(x^2) & \text{$a$ and $b$ odd, $c$ even,}
\end{cases}
\]
and if $\core {a,b,c} 2 = (1)=(0|0)$, which is $(3,0,1)$-asymmetric and $i_0=1$, then
\[
\sp_{(a,b,c)}(x,-x,y)=
\begin{cases}
 (y+\bar{y}) \sp_{(\frac{a-1}{2},\frac{b}{2})}(x^2,y^2) \oo_{(\frac{c}{2})}(x^2) & \text{$a$ odd, $b$ and $c$ even,}\\
- (y+\bar{y}) \sp_{(\frac{a-1}{2},\frac{c-1}{2})}(x^2,y^2) \oo_{(\frac{b+1}{2})}(x^2) & a,b,c \text{ odd,}\\
(y+\bar{y}) \sp_{(\frac{b-2}{2},\frac{c-1}{2})}(x^2,y^2) \oo_{(\frac{a+2}{2})}(x^2) & \text{ $a$ and $b$ even, $c$ odd.}
\end{cases}
\] 
\end{eg}

% \begin{eg}
% Let $\lambda$ be a partition of length at most $2n+1$. The symplectic character $\sp_{\lambda}(X,-X,y)$ is given by
% \begin{enumerate}
% \item If $\core \lambda 2$ is empty, then
% \begin{equation}
%         \sp_{\lambda}( X,-X,y)  
%         =  \sgn(\sigma_{\lambda}) (-1)^{\epsilon_2}
%          \sp_{\lambda^{(1)}}(X^2)  
%       \oo_{\lambda^{(0)}}(X^2,y^2)  
%     \end{equation}
%     \item If $\core \lambda 2 =(1)$, then
%     \begin{equation*}
%     \sp_{\lambda}(X,-X,y)=
%         (-1)^{\epsilon_2} \sgn(\sigma_{\lambda})  \sp_{\lambda^{(1)}}(X^2, y^2) 
% \oo_{\lambda^{(0)}} (X^2) 
%     \end{equation*}
%     \item Otherwise
%     \begin{equation}
%         \sp_{\lambda}(X,-X,y)=0.
%     \end{equation}
% \end{enumerate}
% \end{eg} 
For the even orthogonal case, we take $G = \OE_{2tn+2}$, the orthogonal group of $(2tn+2) \times (2tn+2)$ square matrices. If $\lambda$ is a partition such that $\core \lambda t$ is $(1,0,k)$-asymmetric for some $k \in [\rk(\core \lambda t)]$, then by \cref{cor:z=-1}, there exists a unique $i_0 \in [\floor{\frac{t}{2}}]$ such that 
\begin{equation*}
 n_0(\lambda)=n, \quad 
    n_{i}(\lambda)+n_{t-i}(\lambda) =
    \begin{cases}
    2n+1+\delta_{i_0,\frac{t}{2}} & \text{ if } i=i_0\\
    2n & \text{ otherwise}
    \end{cases}
    \quad \quad  1 \leq i \leq  \floor{\frac{t}{2}}.
\end{equation*}
% \eqref{n_{t-1}} holds. 
For such a partition $\lambda$, let
\begin{equation}
\label{epsilon3}
\epsilon_3(\lambda) \coloneqq
\left(\ds 
\sum_{i=t+1-i_0}^{t-1} n_{i}(\lambda) \right)
+ \ds \sum_{i=\floor{\frac{t+2}{2}}}^{t-1}  \left( \binom{n_i(\lambda)}{2} +
(t-1) n (n_{i}(\lambda)-n) \right),
% \left(\ds 
% \sum_{i=t+1-i_0}^{t-1} n_{i}(\lambda) \right)
% + \ds \sum_{i=\floor{\frac{t+2}{2}}}^{t-1} 
% (t-1) n (n_{i}(\lambda)-n) 
% + \sum_{q=0}^{\floor{\frac{t-1}{2}}} \left( \frac{n_{t-q}(\lambda)
% (n_{t-q}(\lambda)-1)}
% {2} \right)
\end{equation}
and $\ds \pi^{(3)}_i=  \lambda^{(t-i)}_1 +
\left(\lambda^{(i)}, 0,\dots,0, -\rev(\lambda^{(t-i)})\right)$ has $2n+\delta_{i,i_0}$ parts for 
$1 \leq i \leq \floor{\frac{t-1}{2}}$. Note that for $(1,0,0)$-asymmetric $t$-cores, $i_0=0$.
\begin{thm}
\label{thm:even}
Let $\lambda$ be a partition of length at most $tn+1$. The even orthogonal character $\oe_{\lambda}(X,\, \omega X, \dots,\omega^{t-1}X,y)$ is as follows:
\begin{enumerate}
    \item If $\core \lambda t$ is 
    % symplectic
    $(1,0,0)$-asymmetric, then 
    \begin{equation}
    \begin{split}
   \oe_{\lambda}(X,\omega X, \dots ,\omega^{t-1}X,y) =
 (-1)^{\epsilon_3(\lambda)} \, \sgn(\sigma_{\lambda}) \oe_{\lambda^{(0)}}(X^t,y^t)  \prod_{i=1}^{\floor{\frac{t-1}{2}}} 
 s_{\pi_i^{(3)}}(X^t,\X^t) \\
 \times \begin{cases}
%  {\frac{1+\delta_{\lambda_n^{(t/2)}}}{2}}
(-1)^{\sum_i \lambda_i^{(t/2)}} \oo_{\lambda^{(t/2)}}(-X^t)& t \text{ even,}\\
1 & t \text{ odd.}
\end{cases}        
\end{split}
    \end{equation}
\item If $\core \lambda t$ is $(1,0,k)$-asymmetric for some $k \in [\rk(\core \lambda t)]$ 
% =(\alpha|\beta) \in \mathcal{Q}^{(t)}_{1,0,k}$ 
and 
% $t$-even 
% $i_0 \coloneqq$ 
% $\alpha_k+1 \equiv t/2 \pmod{t}$, 
$i_0=\frac{t}{2}$, then 
\begin{equation}
    \begin{split}
    \oe_{\lambda}(X,\omega X, \dots ,\omega^{t-1}X,y) =     (-1)^{\epsilon_3(\lambda)}  \sgn(\sigma_{\lambda}) \, &
% (1+\delta_{\lambda^{(t/2)}_{n+1},0}) 
\oe_{\left(\frac{t}{2}\right)}(y)
% (y^{t/2}+\bar{y}^{t/2})
\, 
\oe_{\lambda^{(0)}}(X^t)  \\ \times  \prod_{q=1}^{\frac{t-2}{2}}  s_{\pi_q^{(3)}}(X^t,\X^t) 
\times &(-1)^{ \sum_i \lambda_i^{(t/2)}}  \oo_{\lambda^{(t/2)}}(-X^t,-y^t).
    \end{split}
\end{equation}
\item If $\core \lambda t$ is $(1,0,k)$-asymmetric for some $k \in [\rk(\core \lambda t)]$ 
% =(\alpha|\beta) \in \mathcal{Q}^{(t)}_{1,0,k}$ 
and 
% $t$-even 
% $i_0 \coloneqq$ 
% $\alpha_k+1 \equiv t/2 \pmod{t}$, 
$i_0 \neq \frac{t}{2}$, then 
\begin{equation}
    \begin{split}
     \oe_{\lambda}(X,\omega X, \dots ,\omega^{t-1}X,y) =     (-1)^{\epsilon_3(\lambda)+n} \, \sgn(\sigma_{\lambda}) \oe_{\lambda^{(0)}}(X^t) & \\
\times \left( y^{-\mu^{(3)}_{i_0}} s_{\pi_{i_0}^{(3)}}(X^t,\X^t,y^t) +
y^{\mu^{(3)}_{i_0}}
s_{\pi_{i_0}^{(3)}}(X^t,\X^t, \bar{y}^t) \right) & \prod_{\substack{j=1 \\ j \neq i_0}}^{\floor{\frac{t-1}{2}}} s_{\pi_j^{(3)}}(X^t,\X^t) \\  \times 
 \begin{cases}
%  (1+\delta_{\lambda^{(t/2)}_n,0}) \frac{\oe_{\lambda^{(t/2)}+1/2}(X^t)}{\ds \prod_{i=1}^n (x_i^{t/2}+\bar{x}_i^{t/2})} 
 (-1)^{\sum_i \lambda_i^{(t/2)}} 
 \oo_{\lambda^{(t/2)}}(-X^t) & t \text{ even},\\
1 & t \text{ odd},
 \end{cases} &
    \end{split}
\end{equation}
where $\mu^{(3)}_{i_0}=t(\lambda_1^{(t-i_0)}
+n_{(t-i_0)}(\lambda)-n)-i_0$.
\item If none of the above conditions hold, then
\begin{equation*}
\oe_{\lambda}(X,\omega X, \dots ,\omega^{t-1}X,y) =0.    
\end{equation*}
\end{enumerate}
\end{thm}
\begin{eg}
For $t=2$, \cref{thm:even} 
says that the character of the group $\OE_{6}$ of the representation indexed by the partition $(a,b,c)$, $a \geq b \geq c \geq 0$, evaluated at $(x,-x,y)$ is non-zero if and only if $a$ and $b$ have the same parity or $a$ and $c$ have the opposite parity same as in \cref{eg:schur} and \cref{eg:symp}. If $\core {a,b,c} 2$ is empty, which is $(1,0,0)$-asymmetric, then
\[
\oe_{(a,b,c)}(x,-x,y)=
\begin{cases}
(-1)^{\frac{c+1}{2}}\oe_{(\frac{a}{2},\frac{b+1}{2})}(x^2,y^2) \oo_{(\frac{c-1}{2})}(-x^2) & \text{$a$  even, $b$ and $c$ odd,}\\
(-1)^{\frac{b}{2}} \oe_{(\frac{a}{2},\frac{c}{2})}(x^2,y^2) \oo_{(\frac{b}{2})}(-x^2) & a,b,c \text{ even,}\\
(-1)^{\frac{a+1}{2}} \oe_{(\frac{b-1}{2},\frac{c}{2})}(x^2,y^2) \oo_{(\frac{a+1}{2})}(-x^2) & \text{$a$ and $b$ odd, $c$ even,}
\end{cases}
\]
and if $\core {a,b,c} 2 = (1)=(0|0)$, which is $(1,0,1)$-asymmetric and $i_0=1$, then
\[
\oe_{(a,b,c)}(x,-x,y)=
\begin{cases}
(-1)^{\frac{a+b-1}{2}} (y+\bar{y}) \, \oo_{(\frac{a-1}{2},\frac{b}{2})}(-x^2,-y^2) \, \oe_{(\frac{c}{2})}(x^2) & \text{$a$ odd, $b$ and $c$ even,}\\
(-1)^{\frac{a+c}{2}} (y+\bar{y}) \, \oo_{(\frac{a-1}{2},\frac{c-1}{2})}(-x^2,-y^2) \, \oe_{(\frac{b+1}{2})}(x^2) & a,b,c \text{ odd,}\\
(-1)^{\frac{b+c-3}{2}} (y+\bar{y}) \,  \oo_{(\frac{b-2}{2},\frac{c-1}{2})}(-x^2,-y^2) \, \oe_{(\frac{a+2}{2})}(x^2) & \text{ $a$ and $b$ even, $c$ odd.}
\end{cases}
\] 
\end{eg}

\begin{rem}
The factorization of characters of classical groups of type $B$, $C$ and $D$ specialized with $tn$ variables are considered  in~\cite{ayyer-2021}. We will not recover the results~\cite[Theorem 2.11, Theorem 2.15, Theorem 2.17]{ayyer-2021} by substituting $y=0$ in the above factorization results as these are Laurent polynomials in $\mathbb{C}[x_1,x_1^{-1},\dots,x_n,x_n^{-1},y,y^{-1}]$. $($See~\cite{koike1987young}$)$.   
\end{rem}

It is natural to ask if there are infinitely many $(z_1,z_2,k)$-asymmetric $t$-cores. Our last result answers this question in a special case. For $z_1>z_2$, let 
\[
\mathcal{Q}^{(t)}_{z_1,z_2} \coloneqq \ds \bigcup_k \mathcal{Q}^{(t)}_{z_1,z_2,k}.
\]

\begin{thm} 
\label{thm:inf-cores}
There are infinitely many $t$-cores $\mathcal{Q}^{(t)}_{z+2,0} \cup \mathcal{Q}^{(t)}_{z+2,z+1}$ for $t \geq z$.
\end{thm}

This is proved in \cref{sec:gf}.
\section{Background results}
\label{sec:back}
% In \cref{sec:betasets}, we study beta sets of $(z_1,z_2,k)$-asymmetric partitions defined in \cref{defn:asym}. We will derive generating functions for such partitions and prove that there are infinitely many $t$-cores in \cref{sec:gf}. Lastly, we will derive determinant identities for block matrices in \cref{sec:det}.

\subsection{Properties of beta sets}
\label{sec:betasets}
We use the shorthand
notations $[m]=\{1,\dots,m\}$, $[m_1, m_2]=\{m_1,\dots,m_2\}$ and $m_+ := \max(m,0)$.
We first recall a useful property of the beta numbers. 
For a partition $\lambda$ of length at most $m$, we see by~\cite[Equation (3.1)]{ayyer-2021}:
\begin{equation}
\label{no-parts-partition=core}
 n_{i}(\lambda,m)=n_{i}(\core \lambda t, m), 
 \quad 0 \leq i \leq t-1.  
\end{equation}
\begin{lem}[{\cite[Lemma 3.10]{ayyer-2021}}] 
If $\lambda$ is a $t$-core of length at most $tn$, then  \begin{equation}
\label{rnk}
 \rk(\lambda) =\sum_{i=0}^{t-1}(n_{i}(\lambda)-n)_+.
\end{equation}
\end{lem}
\begin{lem}
\label{lem:rank-core}
If $\lambda$ is a $t$-core of length at most $tn+1$, then  
\begin{equation}
\label{rnk1}
 \rk(\lambda) =(n_{0}(\lambda)-n-1)_++\sum_{i=1}^{t-1}(n_{i}(\lambda)-n)_+.
\end{equation}
\end{lem}

\begin{proof} If $\ell(\lambda) \leq tn$, then using $n_0(\lambda,tn+1)=n_{t-1}(\lambda,tn)+1$ and $n_{i}(\lambda,tn+1)=n_{i-1}(\lambda,tn)$, $1 \leq i \leq t-1$ in \cref{rnk}, we see that the result holds.
% If $n_{i}(\lambda) = n$ for all $0 \leq i \leq t-1$ but , then $\beta(\lambda)=(
% tn-1,tn-2,\dots,1,0)$ which implies $\lambda$ is an empty partition. So, the result holds in this case.
% Otherwise, a
Assume $\ell(\lambda) = tn+1$. Since $\lambda$ is a $t$-core, $n_0(\lambda)=0$. Let $1 \leq i_k < \dots < i_1 \leq t-1$ 
% $\{i_1,i_2,\dots,i_k\}_{>} \subset \{1,\dots,t-1\}$ 
such that $n_{i_j}(\lambda)>n \text{ for } 1 \leq j \leq k$.  
Since $\lambda$ is a $t$-core, the parts of $\beta(\lambda)$ greater than $tn$ for each $j$ are:
\[
i_j+tn<i_j+t(n+1)<\dots<i_j+t(n_{i_j}(\lambda)-1).
\]
If $r$ is the number of parts of  $\beta(\lambda)$ greater than  $tn$, then
\[ 
r= \sum_{j=1}^k(n_{i_j}(\lambda)-n)=\sum_{i=1}^{t-1}(n_{i}(\lambda)-n)_+.
\]
Moreover, $\beta_{r}(\lambda)$ is the smallest part of  $\beta(\lambda)$ greater than  $tn$ and is therefore equal to $i_k+tn$. This implies $\ds \lambda_{r}=\beta_{r}(\lambda)-(tn+1-r) = tn+i_k-(tn+1-r)=i_k+r-1 \geq r$ and 
$\lambda_{r+1} \leq tn-(tn-r) \leq r$,
 which implies the rank of $\lambda$ is $r$.
\end{proof}

\begin{lem}[{\cite[Corollary 3.7]{ayyer-2021}}]
\label{lem:sym-tn}
Let $\lambda$ be a partition of length at most $tn$. Then $\core \lambda t$ is $(1,0,0)$-asymmetric if and only if 
\[n_i(\lambda,tn)+n_{t-2-i}(\lambda,tn)=2n, \,\,\, 0 \leq i \leq t-2, \quad  n_{t-1}(\lambda,tn)=n+1.\]
\end{lem}
\begin{lem}
\label{lem:symp-1}
Let $\lambda$ be a partition of length at most $tn+1$. Then $\core \lambda t$ is $(1,0,0)$-asymmetric if and only if 
\[n_0(\lambda,tn+1)=n+1, \quad n_i(\lambda,tn+1)+n_{t-i}(\lambda,tn+1)=2n, \,\,\,1 \leq i \leq t-1.\]
\end{lem}
\begin{proof} As $\ell(\lambda) \leq tn+1 \leq t(n+1)$, considering $\lambda$ with $\ell(\lambda) \leq tn+t$, we get by \cref{lem:sym-tn}:
\[
n_{i}(\lambda,tn+t)+n_{t-2-i}(\lambda,tn+t)=2n+2, \,\,\, 0 \leq i \leq t-2, \quad n_{t-1}(\lambda,tn+t)=n+1.\]
Now the proof of the lemma follows by noting:
\[
n_i(\lambda,tn+1)=\begin{cases}
n_{t-1}(\lambda,tn+t) & i=0,\\
n_{i-1}(\lambda,tn+t)-1 & 1 \leq i \leq t-1.
\end{cases}
\]
\end{proof}

Recall the definitions $\mathcal{Q}_{z_1,z_2,k}$ and $\mathcal{Q}^{(t)}_{z_1,z_2,k}$ from \cref{defn:asym}.
% of $(z_1,z_2,k)$-asymmetric partition from \cref{defn:asym}. Let $\mathcal{Q}_{z_1,z_2,k}$ be the set of $(z_1,z_2,k)$-asymmetric partitions and $\mathcal{Q}^{(t)}_{z_1,z_2,k}$ be the set of $(z_1,z_2,k)$-asymmetric $t$-cores.
\begin{lem}
\label{condd}
Let $\lambda
% =(\alpha|\beta)
$ be a partition of length at most $\ell$ and rank $r$.
Then the following statements are equivalent.
\begin{enumerate}
    \item 
    \label{condd1QQ} 
    $\lambda \in \mathcal{Q}_{z_1,z_2,k}$.
\item 
\label{condd2QQ}
 $\beta(\lambda,\ell)$ is obtained from the sequence $(\alpha_1+\ell,\dots,\alpha_r+\ell, \ell-1,\dots,1,0)$
by deleting the numbers $\ell-1-z_2> \ell-1-z_1-\alpha_r > \dots> \ell-1-z_1-\alpha_{k+1} > \ell-1-z_1-\alpha_{k-1} > \dots > \ell-1-z_1-\alpha_1$.             
\end{enumerate}
\end{lem}  
\begin{proof} 
First, note that $\lambda \in \mathcal{Q}_{z_1,z_2,k}$ if and only if $\lambda$ is of the form
\begin{multline*}
    \lambda=(\alpha_1+1,\dots,\alpha_r+r,\underbrace{r,\dots,r}_{z_2},\underbrace{r-1,\dots,r-1}_{\alpha_r+z_1-z_2-1},\underbrace{r-2,\dots,r-2}_{\alpha_{r-1}-\alpha_{r}-1},\dots,\underbrace{k,\dots,k}_{\alpha_{k+1}-\alpha_{k+2}-1},\\
    \underbrace{k-1,\dots,k-1}_{\alpha_{k-1}-\alpha_{k+1}-1},
    \underbrace{k-2,\dots,k-2}_{\alpha_{k-2}-\alpha_{k-1}-1} \dots,\underbrace{1,\dots,1}_{\alpha_{1}-\alpha_{2}-1}).
\end{multline*}
In that case, its beta set reads as:
\begin{multline*}
    \beta(\lambda,\ell)=(\alpha_1+\ell,\dots,\alpha_r+\ell,
    \underbrace{\ell-1,\dots,\ell-z_2}_{z_2},
    \underbrace{\ell-z_2-2,\dots,
    \ell-(\alpha_r+z_1)}_{\alpha_r+z_1-z_2-1},\\
    \widehat{\ell-\alpha_r-z_1-1},
    \underbrace{\ell-\alpha_r-z_1-2,\dots,\ell-(\alpha_{r-1}+z_1)}_{\alpha_{r-1}-\alpha_{r}-1},
    \widehat{\ell-\alpha_{r-1}-z_1-1},\\
    \dots,  
    \widehat{\ell-\alpha_{k+2}-z_1-1},
    \underbrace{\ell-\alpha_{k+2}-z_1-2,\dots,\ell-(\alpha_{k+1}+z_1)}_{\alpha_{k+1}-\alpha_{k+2}-1},
    \widehat{\ell-\alpha_{k+1}-z_1-1},\\ \underbrace{\ell-\alpha_{k+1}-z_1-2,\dots,\ell-\alpha_k-z_1,\dots,\ell-(\alpha_{k-1}+z_1)}_{\alpha_{k-1}-\alpha_{k+1}-1},
    \widehat{\ell-\alpha_{k-1}-z_1-1}\\
    \underbrace{\ell-\alpha_{k-1}-z_1-2,\dots,\ell-(\alpha_{k-2}+z_1)}_{\alpha_{k-2}-\alpha_{k-1}-1},\widehat{\ell-\alpha_{k-2}-z_1-1},
    \dots,\widehat{\ell-\alpha_{2}-z_1-1},\\
    \underbrace{\ell-\alpha_{2}-z_1-2,\dots,\ell-(\alpha_{1}+z_1)}_{\alpha_{1}-\alpha_{2}-1},
    \widehat{\ell-\alpha_{1}-z_1-1},\underbrace{\ell-\alpha_{1}-z_1-2,\dots,0}_{\ell-\alpha_{1}-z_1-1}).
\end{multline*} 
% where a hat on a coordinate denotes its omission from the tuple.
So, \cref{condd1QQ} and \cref{condd2QQ} are equivalent. 
% See \cref{fig:z-asymmetric}(a) and (b) for the last few rows of the Young diagram of a $z$-asymmetric partition and its beta set.
\end{proof}
\begin{lem} 
\label{lem:converse: sym}
Let $\lambda$ be a $t$-core of length at most $tn+1$ and $0 < z+2 \leq t+2$. Then for $i_0 \in [0,\floor{\frac{t-z-1}{2}}] \cup [t-z,t-1]$,
\begin{equation}
\begin{split}
  \label{n_{t-1}} 
    n_{i}(\lambda)+n_{t-z-1-i}(\lambda) =& 
    \begin{cases}
    2n+1+\delta_{i_0,\frac{t-z-1}{2}} & \text{ if } i=i_0,\\
    2n & \text{ otherwise},
    \end{cases}
    \quad \text{  for} \quad  0 \leq i \leq  \floor{\frac{t-z-1}{2}}, \\ 
    \text{and} \quad n_{i}(\lambda) =& \begin{cases}
    n+1 & \text{ if } i=i_0,\\
    n & \text{ otherwise},
    \end{cases} \quad \text{  for} \quad t-z \leq i \leq t-1,
\end{split}
\end{equation}
if and only if $\lambda \in \mathcal{Q}^{(t)}_{z+2,0,k} \cup \mathcal{Q}^{(t)}_{z+2,z+1,k}$ for some $1 \leq k \leq \rk(\lambda)$.
\end{lem}
\begin{proof} 
Assume \eqref{n_{t-1}} holds for $\lambda$. 
% If $\lambda$ is the empty partition, then it belongs to $\mathcal{P}_{z,t}$ vacuously. Let $\lambda$ be non-empty. 
Suppose we have $0 \leq i_m < i_{m-1} < \dots < i_1 \leq t-z-1$ 
such that $n_{i_j}(\lambda)>n$ for all $j \in [m]$. Since $\lambda$ is a $t$-core, for each $j$, the parts of  $\beta(\lambda)$ greater than and equal to $tn$ are:
\[
i_j+t(n_{i_j}(\lambda)-1)>\dots>i_j+t(n+1)>i_j+tn
\]
Note that by \cref{lem:rank-core}, the rank $r$ of $\lambda$ is same as the number of parts of  $\beta(\lambda)$ greater than $tn$. Let $\gamma_s$, $1 \leq s \leq r$ be the sequence of these parts greater than $tn$ arranged in decreasing order. Note that $\gamma_s=\alpha_s+tn+1$ for some $\alpha_s>0$, $1 \leq s \leq r$.
Since $n_{t-z-1-i_j}(\lambda) \leq n$ for $j \in [m]$, $i_j \neq \frac{t-z-1}{2}$, the parts of $\beta(\lambda)$ lesser than $tn$ are obtained from the sequence $(tn-1,tn-2,\dots,0)$ by deleting the numbers
\[
t(n_{t-z-1-i_j}(\lambda)+1)-i_j-z-1< t(n_{t-z-1-i_j}(\lambda)+2)-z-i_j-1< \dots< tn-i_j-z-1.
\]

Suppose $i_0 \in [0,t-z-1]$. Then either $n_{i_0}(\lambda) >n$, or $n_{t-z-1-i_0}(\lambda)>n$. If $n_0(\lambda) \geq n$, then $tn \in \beta(\lambda)$, and the deleted numbers are $tn-z-1$, $2tn-z-1-\gamma_s$,  $s \in [r]$, $\gamma_s \neq i_0+t(n_{i_0}(\lambda)-1)$ or $t-z-1-i_0+t(n_{t-z-1-i_0}(\lambda)-1)$. So,   
$\beta(\lambda,tn+1)$ is obtained from the sequence $(\alpha_1+tn+1,\dots,\alpha_r+tn+1,tn,\dots,1,0)$
by deleting the numbers $tn-z-1>tn-z-2-\alpha_r > \dots> tn-z-2-\alpha_{k+1} > tn-z-2-\alpha_{k-1} > \dots > tn-z-2-\alpha_1$. Therefore by \cref{condd}, $\lambda \in \mathcal{Q}_{z+2,z+1,k}$. Here $k$ is the position of $t(n_{i_0}(\lambda)-1)+i_0$ or $t-z-1-i_0+t(n_{t-z-1-i_0}(\lambda)-1)$ in $\beta(\lambda,tn+1)$, because their counterpart $2tn-z-1-t(n_{i_0}(\lambda)-1)-i_0
=t\,n_{t-z-1-i_0}(\lambda)-i_0-z-1$ or $2tn-z-1-t+z+1+i_0-t(n_{t-z-1-i_0}(\lambda)-1)=t(n_{i_0}(\lambda)-1)+i_0$ weren't removed from the sequence  $(tn-1,tn-2,\dots,0)$. If $n_0(\lambda) < n$, then
$\beta(\lambda,tn+1)$ is obtained from the sequence $(\alpha_1+tn+1,\dots,\alpha_r+tn+1, tn-1,\dots,1,0$)
by deleting the numbers $tn-z-2-\alpha_r > \dots> tn-z-2-\alpha_{k+1} > tn-z-2-\alpha_{k-1} > \dots > tn-z-2-\alpha_1$. Therefore by \cref{condd}, $\lambda \in \mathcal{Q}_{z+2,0,k}$.

Suppose $i_0 \in [t-z, t-1]$. In this case $n_{i_0}(\lambda)=n+1$. If $n_0(\lambda) \geq n$, then $tn \in \beta(\lambda)$. 
% rank=$r_1+r_2+1$. $r_2<$
and $\beta(\lambda,tn+1)$ is obtained from the sequence $(\alpha_1+tn+1,\dots,\alpha_r+tn+1, tn,tn-1,\dots,1,0$)
by deleting the numbers $tn-z-1> tn-z-2-\alpha_r > \dots> tn-z-2-\alpha_{k+1} > tn-z-2-\alpha_{k-1} > \dots > tn-z-2-\alpha_1$. So by \cref{condd}, $\lambda \in \mathcal{Q}_{z+2,z+1,k}$.
If $n_0(\lambda) \leq n$, then
$\beta(\lambda,tn+1)$ is obtained from the sequence $(\alpha_1+tn+1,\dots,\alpha_r+tn+1, tn-1,\dots,1,0$)
by deleting the numbers $tn-z-2-\alpha_r > \dots> tn-z-2-\alpha_{k+1} > tn-z-2-\alpha_{k-1} > \dots > tn-z-2-\alpha_1$. So, $\lambda \in \mathcal{Q}_{z+2,0,k}$.

Conversely, suppose $\lambda \in \mathcal{Q}_{z+2,0,k}$ and $\rk(\lambda)=r$.
By \cref{condd},  $\beta(\lambda)$ is obtained from the sequence $(\alpha_1+tn+1,\dots,\alpha_r+tn+1, tn-1,\dots,1,0$) by deleting the numbers 
$tn-z-2-\alpha_r > \dots >tn-z-2-\alpha_{k+1} >tn-z-2-\alpha_{k-1} > \dots > tn-z-2-\alpha_1$. 
% Since $n_{i}(\emptyset,tn) = n$ for all $i$, \eqref{n_{t-1}} trivially holds for the empty partition.
Note that if $tn-z-2-\alpha_i \equiv \theta_i \pmod{t}$, then $\alpha_i+tn+1  \equiv t-z-1-\theta_i \pmod{t}$ for $i \in [r]$. In that case $n_{t-z-1-\theta_i}(\lambda)$ increases by one and $n_{\theta_i}(\lambda)$ decreases by one. Since $\lambda$ is a $t$-core and $tn \not \in \beta(\lambda)$, $\theta_i$, for all $i \in [r]$, $i \neq k$ can not be equal to $t-z-1$. If
\[
i_0 = \begin{cases}
t-z-1-\theta_k   & \text{if } t-z-1-\theta_k \in [0,\floor{\frac{t-z-1}{2}}] \cup [t-z,t-1], \\
\theta_k & \text{otherwise},
\end{cases}
\]
then it is suffices to show that $\theta_i \in [0,t-z-2]$, for each $i \in [r] \setminus \{k\}$, to prove \eqref{n_{t-1}}.
We prove this successively in reverse order starting from $\theta_r$ and going all the way to $\theta_1$. Since $\lambda$ is a $t$-core, if $tn-z-2-\alpha_r$ does not occur in $\beta(\lambda)$, then neither does $tn-z-2-\alpha_r+t$. 
Since $tn-z-2-\alpha_r$ is the largest number deleted  from $(tn-1,\dots,0)$ to get $\beta(\lambda)$, $tn-z-2-\alpha_r+t \geq tn$. So,
$\alpha_r+z+2 \in [z+2,t]$; and $\theta_r \in [0,t-z-2]$.
There is nothing to show if $\theta_{r-1}=\theta_r$. So, 
assume $\theta_{r-1} \neq \theta_{r}$. Similarly, since $\lambda$ is a $t$-core, if $tn-z-2-\alpha_{r-1}$ does not occur in $\beta(\lambda)$, then neither does $tn-z-2-\alpha_{r-1}+t$. 
Since $tn-z-2-\alpha_{r-1}$ is the largest number congruent to $\theta_{r-1}$ deleted from $(tn-1,\dots,0)$ to get $\beta(\lambda)$, 
$\alpha_{r-1}+z+2 \in [z+2,t]$ and $\theta_{r-1} \in [0,t-z-2]$.
Proceeding in this manner, $\theta_i \in [0,t-z-2]$ for all $i \in [r] \setminus \{k\}$. 

It is easy to see that $\lambda \in \mathcal{Q}_{z+2,z+1,k}$ implies \eqref{n_{t-1}} using similar arguments.

\end{proof}

% \begin{defn}
% Associate $\lambda$ to $i_0(\lambda)$. For empty $t$-core, assume $i_0(\lambda)=0.$
% \end{defn}

The following three corollaries are needed in the proofs of the main results. These are easily shown by applying \cref{lem:converse: sym} for $z = -1,0,1$ respectively, and using the facts that $\ell(\core \lambda t) \leq \ell(\lambda) \leq tn+1$ 
and \eqref{no-parts-partition=core} for $m=tn+1$.

\begin{cor}
\label{cor:z=-1}
Let $\lambda$ be a partition of length at most $tn+1$, and $i_0 \in [1,\floor{\frac{t}{2}}]$. Then
\begin{equation}
  \label{n_{-1}} 
  n_0(\lambda)=n, \quad 
    n_{i}(\lambda)+n_{t-i}(\lambda) =
    \begin{cases}
    2n+1+\delta_{i_0,\frac{t}{2}} & \text{ if } i=i_0\\
    2n & \text{ otherwise}
    \end{cases}
    \quad \text{for} \quad  1 \leq i \leq  \floor{\frac{t}{2}}, 
\end{equation}
if and only if $\core \lambda t \in \mathcal{Q}^{(t)}_{1,0,k}$ for some $1 \leq k \leq \rk(\core \lambda t)$.
\end{cor}

\begin{cor}
\label{cor:z=0}
Let $\lambda$ be a partition of length at most $tn+1$, and $i_0 \in [0,\floor{\frac{t-1}{2}}]$. Then
\begin{equation}
  \label{n_{00}} 
    n_{i}(\lambda)+n_{t-1-i}(\lambda) = 
    \begin{cases}
    2n+1+\delta_{i_0,\frac{t-1}{2}} & \text{ if } i=i_0\\
    2n & \text{ otherwise}
    \end{cases}
    \quad \text{for} \quad  0 \leq i \leq  \floor{\frac{t-1}{2}}, 
\end{equation}
if and only if $\core \lambda t \in \mathcal{Q}^{(t)}_{2,0,k} \cup \mathcal{Q}^{(t)}_{2,1,k}$ for some $1 \leq k \leq \rk(\core \lambda t)$.
\end{cor}

\begin{cor}
\label{cor:z=1}
Let $\lambda$ be a partition of length at most $tn+1$, and $i_0 \in [0,\floor{\frac{t-2}{2}}] \cup \{t-1\}$. Then 
\begin{equation}
\begin{split}
  \label{n_{11}} 
    n_{i}(\lambda)+n_{t-2-i}(\lambda) =& 
    \begin{cases}
    2n+1+\delta_{i_0,\frac{t-2}{2}} & \text{ if } i=i_0\\
    2n & \text{ otherwise}
    \end{cases}
    \quad \text{for} \quad  0 \leq i \leq  \floor{\frac{t-2}{2}}, \\ 
    \text{and} \quad n_{t-1}(\lambda) =& \begin{cases}
    n+1 & \text{ if } i_0=t-1\\
    n & \text{ otherwise}
    \end{cases} 
\end{split}
\end{equation}
if and only if $\core \lambda t \in \mathcal{Q}^{(t)}_{3,0,k} \cup \mathcal{Q}^{(t)}_{3,2,k}$ for some $1 \leq k \leq \rk(\core \lambda t)$.
\end{cor}
\subsection{Determinantal Identities}
\label{sec:det} 
Let $\lambda$ be a partition with $\ell(\lambda) \leq tn+1$.
Recall, $\beta_j^{(p)}(\lambda),$ for $0 \leq p \leq t-1$,
$1 \leq j \leq n_{p}(\lambda)$
are the parts of $\beta(\lambda)$ that are congruent to 
$p$ modulo $t$, arranged in decreasing order.
Additionly, for $q \in \mathbb{Z} \cup (\mathbb{Z}+1/2)$,
define $n \times n_{p}(\lambda)$ matrices
\begin{equation}
\label{def AB}
A_{p,q}^{\lambda}=\left(
x_i^{\beta^p_j(\lambda)+q}
\right)_{\substack{1 \leq i \leq n\\1 \leq j \leq n_p(\lambda)}}, 
\bar{A}_{p,q}^{\lambda}=\left(
\x_i^{\beta^p_j(\lambda)+q}
\right)_{\substack{1 \leq i \leq n\\1 \leq j \leq n_p(\lambda)}},
\end{equation}
and $1 \times n_{p}(\lambda)$ matrices
\begin{equation}
\label{def 1AB}
B_{p,q}^{\lambda}=\left(
y^{\beta^p_j(\lambda)+q}
\right)_{1 \leq j \leq n_p(\lambda)}, 
\bar{B}_{p,q}^{\lambda}=\left(
\bar{y}^{\beta^p_j(\lambda)+q}
\right)_{1 \leq j \leq n_p(\lambda)}.
\end{equation}
The corresponding matrices for the empty partition are denoted by
\begin{equation}
\label{def AB0}
A_{p,q}=\left(
x_i^{t(n+\delta_{p,0}-j)+p+q}
\right)_{\substack{1 \leq i \leq n\\1 \leq j \leq n+\delta_{p,0}}}, 
\bar{A}_{p,q}=\left(
\x_i^{t(n+\delta_{p,0}-j)+p+q}
\right)_{\substack{1 \leq i \leq n\\1 \leq j \leq n+\delta_{p,0}}},
\end{equation}
and
\begin{equation}
\label{def 1AB0}
B_{p,q}=\left(
y^{t(n+\delta_{p,0}-j)+p+q}
\right)_{1 \leq j \leq n+\delta_{p,0}}, \qquad
\bar{B}_{p,q}=\left(
y^{t(n+\delta_{p,0}-j)+p+q}
\right)_{1 \leq j \leq n+\delta_{p,0}}.
\end{equation}
In all cases, whenever $q=0$, we will omit it. For example, we will write $A^{\lambda}_{p}$ instead of $A^{\lambda}_{p,0}$.
Recall, the $t$-quotient of $\lambda$ is given by
$\quo \lambda t = (\lambda^{(0)},\dots, \lambda^{(t-1)})$.
We write down alternate formulas for the classical characters using the relation(s):
% \cref{prop:mcd-t-core-quo}(2),
\[
t\beta_j(\lambda^{(p)})=\beta_j^{(p)}(\lambda)-p, \quad 1\leq j\leq n.
\]
Recall, $X^t = (x_1^t,\dots,x_n^t)$. 
% and $Y^t = (x_1^t,\dots,x_n^t,y^t)$. 
If $n_{p}(\lambda) = n$, then using the above notations and formulas/definitions \eqref{gldef}--\eqref{oedef} respectively, we have:
\begin{eqnarray}
\label{E1} & \ \  \text{Schur polynomial}:  \hspace{3.5em}
s_{\lambda^{(p)}}(X^t) & =  \frac{\det A^{\lambda}_{p,-p}}
  {\det \left(x_i^{t(n - j)}\right)_{1 \leq i,j \leq n}},
   \\
  \label{E2} & \ \ \text{symplectic character}:  \hspace{2em}
  \sp_{\lambda^{(p)}}(X^t) & = \frac{\det \left(A^{\lambda}_{p,t-p}-\bar{A}^{\lambda}_{p,t-p}\right)}
  {\det  \left(x_i^{t(n+1 - j)} - \x_i^{t(n+1 - j)}\right)_{1 \leq i,j \leq n}},
  \\
 \label{E3} &\ \ \text{odd orthogonal character}:  \hspace{0.2em} \oo_{\lambda^{(p)}}(X^t)& =  \frac{\det \left(A^{\lambda}_{p,t-p}-\bar{A}^{\lambda}_{p,-p}\right)}
 {\det  \left(x_i^{t(n+1 - j)} - \x_i^{t(n - j)}\right)_{1 \leq i,j \leq n}}, \\
\label{E4}  &\ \ \text{even orthogonal character}:   \oe_{\lambda^{(p)}}(X^t)& =  \frac{2 \det \left(A^{\lambda}_{p,-p}+\bar{A}^{\lambda}_{p,-p}\right)}
  {(1+\delta_{n,0}) \det  \left(x_i^{t(n- j)} - \x_i^{t(n - j)}\right)_{1 \leq i,j \leq n}}.
\end{eqnarray}
If $n_{p}(\lambda) = n+1$, then using formulas \eqref{gldef}--\eqref{oedef} respectively, we have: 

\begin{eqnarray}
   \label{F1} & \text{Schur polynomial}:  \hspace{3.5em} s_{\lambda^{(p)}}(X^t,y^t)
 & = \frac{\det \left(
  \begin{array}{c}
       A^{\lambda}_{p,-p}  \\\hline 
       B^{\lambda}_{p,-p} 
  \end{array} \right)} 
  {\det \left(
  \begin{array}{c}
       A_0  \\\hline 
       B_0
  \end{array} \right)},\\
  \label{F2} & \text{symplectic character}:   \hspace{2em} \sp_{\lambda^{(p)}}(X^t,y^t)
 & = \frac{  \det  \left(
\begin{array}{c}
    A_{p,t-p}^{\lambda}
    -\bar{A}^{\lambda}_{p,t-p}   \\[0.2cm]
    \hline\\[-0.3cm]
     B_{p,t-p}^{\lambda}
     -\bar{B}^{\lambda}_{p,t-p}
\end{array}
 \right)}
 {\det  \left(
\begin{array}{c}
    A_{0,t}-
    \bar{A}_{0,t}   \\[0.2cm]
   \hline \\[-0.3cm] 
     B_{0,t}-
     \bar{B}_{0,t} 
\end{array}
 \right)},\\
\label{oo-new-1} & \text{odd orthogonal character}:  \hspace{0.2em} \oo_{\lambda^{(p)}}(X^t,y^t)
 & = \frac{  \det  \left(
\begin{array}{c}
    A_{p,t-p}^{\lambda}-
    \bar{A}^{\lambda}_{p,-p}   \\[0.2cm]
    \hline\\[-0.3cm]
     B_{p,t-p}^{\lambda}-
     \bar{B}^{\lambda}_{p,-p}
\end{array}
 \right)}
 {\det  \left(
\begin{array}{c}
    A_{0,t}-\bar{A}_{0}   \\[0.2cm]
   \hline \\[-0.3cm] 
     B_{0,t}-\bar{B}_{0} 
\end{array}
 \right)},\\
  \label{F4} & \text{even orthogonal character}:  \oe_{\lambda^{(p)}}(X^t,y^t) 
  &= \frac{ 2 \det  \left(
\begin{array}{c}
    A_{p,-p}^{\lambda}+\bar{A}^{\lambda}_{p,-p}   \\[0.2cm]
    \hline\\[-0.3cm]
     B_{p,-p}^{\lambda}+\bar{B}^{\lambda}_{p,-p}
\end{array}
 \right)}
 { (1+\delta_{n+1,0}) \det  \left(
\begin{array}{c}
    A_0+\bar{A}_0   \\[0.2cm]
   \hline \\[-0.3cm] 
     B_0+\bar{B}_0 
\end{array}
 \right)}.
\end{eqnarray}

\begin{lem}[{\cite[Lemma 3.14]{ayyer-2021}}]
\label{lem:s-new}
Let $\lambda$ be a partition of length at most $tn$ with 
$\quo \lambda t = (\lambda^{(0)},\dots,\lambda^{(t-1)})$. 
If $p,q \in \{0,1,\dots,t-1\}$ such that 
$n_{p}(\lambda)+n_{q}(\lambda)=2n$, then we define
$\rho_{p,q}=\lambda^{(p)}_1 + (\lambda^{(q)}, 0, -\rev(\lambda^{(p)}))$,
where we pad $0'$s in the middle so that $\rho_{p,q}$ is of length $2n$. 
Then the Schur function $s_{\rho_{p,q}}(X^t,{\X}^t)$ can be written as
\begin{equation}
\label{lemmaa}
 s_{\rho_{p,q}}(X^t,{\X}^t)
 = \frac{(-1)^{\frac{n_{p}(\lambda)(n_{p}(\lambda)-1)}{2}} }
 {(-1)^{\frac{n(n-1)}{2}} }
\frac{  \det \left( \begin{array}{c|c}
   A^{\lambda}_{q,-q}  & \bar{A}^{\lambda}_{p,t-p} \\[0.2cm]
   \hline \\[-0.3cm]
   \bar{A}^{\lambda}_{q,-q}  & A^{\lambda}_{p,t-p}
\end{array}\right)}
{\det \left( \begin{array}{c|c}
   A_{q,-q}  & \bar{A}_{p,t-p} \\[0.2cm]
   \hline \\[-0.3cm]
   \bar{A}_{q,-q}  & A_{p,t-p}
\end{array}\right)}.
\end{equation}
\end{lem}

% \begin{proof}
% $p,q \neq 0$. $\ell(\lambda) \leq tn+1$.
% \end{proof}

\begin{lem}
\label{lem:s-new-1}
Let $\lambda$ be a partition of length at most $tn+1$ and $0 \leq p,q \leq t-1$. If
% If $p,q \in \{1,\dots,t-1\}$ such that 
$n_{p}(\lambda)+n_{q}(\lambda)=2n+1$, then define
$\rho^*_{p,q}=\lambda^{(p)}_1 + (\lambda^{(q)}, 0, -\rev(\lambda^{(p)}))$,
where we pad $0'$s in the middle so that $\rho^*_{p,q}$ is of length $2n+1$. 
Then the Schur function $s_{\rho^*_{p,q}}(X^t,{\X}^t,y^t)$ can be written as
\begin{equation}
\label{lemmaa1}
 s_{\rho^*_{p,q}}(X^t,{\X}^t,y^t)
 =\frac
 {(-1)^{\frac{n_{p}(\lambda)(n_{p}(\lambda)-1)}{2}}
 y^{t(\lambda_1^{(p)}+n_{p}(\lambda))}}
 {V\left(X^t,{\X}^t,y^t\right) } 
 \det \left(\begin{array}{c|c}
  A_{q,-q}^{\lambda}  
  & \bar{A}_{p,t-p}^{\lambda}
  \\\\\hline\\
    \bar{A}_{q,-q}^{\lambda}  
   & A_{p,t-p}^{\lambda} 
   \\\\\hline\\
   B_{q,-q}^{\lambda}  
  & \bar{B}_{p,t-p}^{\lambda}
\end{array}
\right),
\end{equation}
where $V \left(X^t,{\X}^t,y^t\right) \coloneqq \ds \prod_{1 \leq i < j \leq n} (x_i^t-x_j^t)(x_i^t-\x_j^t) (x_j^t-\x_i^t) (\x_i^t-\x_j^t) \ds \prod_{i=1}^n (x_i^t-y^t) (x_i^t-\x_i^t) (\x_i^t-y^t)$.
% is the Vandermonde determinant.
\end{lem}
\begin{proof}
We think of the first $n_{q}(\lambda)$ parts of $\rho^*_{p,q}$ as coming from $\lambda^{(q)}$, and the rest from $\lambda^{(p)}$.
Using the Schur polynomial expression \eqref{gldef} for $s_{\rho^*_{p,q}}(X^t,{\X}^t,y^t)$, 
the numerator in the expression is
\[
\det \left(\begin{array}{c|c}
  \left(
  x_i^{t(\lambda^{(p)}_1+\lambda^{(q)}_j+2n+1-j)}
  \right)_{\substack{1 \leq i \leq n\\1 \leq j \leq n_{q}(\lambda)}}  
  & \left(
  x_i^{t(\lambda^{(p)}_1-\lambda^{(p)}_{2n+2-j}+2n+1-j)}
  \right)_{\substack{1 \leq i \leq n\\ n_{q}(\lambda)+1 \leq j \leq 2n+1}}  \\\\\hline\\
   \left(\bar{x}_i^{t(\lambda^{(p)}_1+\lambda^{(q)}_j+2n+1-j)}\right)_{\substack{1 \leq i \leq n\\1 \leq j \leq n_{q}(\lambda)}}  
   & \left(\bar{x}_i^{t(\lambda^{(p)}_1-\lambda^{(p)}_{2n+2-j}+2n+1-j)}\right)_{\substack{1 \leq i \leq n\\n_{q}(\lambda)+1 \leq j \leq 2n+1}}  \\\\\hline\\
   \left(y^{t(\lambda^{(p)}_1+\lambda^{(q)}_j+2n+1-j)}\right)_{1 \leq j \leq n_{q}(\lambda)}  & \left(y^{t(\lambda^{(p)}_1-\lambda^{(p)}_{2n+2-j}+2n+1-j)}\right)_{n_{q}(\lambda)+1 \leq j \leq 2n+1}
\end{array}
\right). \]
Multiplying row $i$ in the top block and middle block of the numerator by  $\bar{x}_i^{t(\lambda_1^{(p)}+n_{p}(\lambda))}$ and ${x}_i^{t(\lambda_1^{(p)}+n_{p}(\lambda))}$ respectively, 
for each $i=1,2,\dots,n$, the last row by   $\bar{y}^{t(\lambda_1^{(p)}+n_{p}(\lambda))}$ and then reversing the last $n_{p}(\lambda)$ columns, the numerator equals
\begin{equation}
\begin{split}
(-1)^{\frac{n_{p}(\lambda)(n_{p}(\lambda)-1)}{2}} {y}^{t(\lambda_1^{(p)}+n_{p}(\lambda))} \det \left(\begin{array}{c|c}
  \left(
  x_i^{t(\lambda^{(q)}_j+n_q-j)}
  \right)_{\substack{1 \leq i \leq n\\1 \leq j \leq n_{q}(\lambda)}}  
  & \left(
  \x_i^{t(\lambda^{(p)}_{j}+n_p+1-j)}
  \right)_{\substack{1 \leq i \leq n\\ 1 \leq j \leq n_p(\lambda)}}  
  \\\\\hline\\
   \left(\bar{x}_i^{t(\lambda^{(q)}_j+n_q-j)}\right)_{\substack{1 \leq i \leq n\\1 \leq j \leq n_{q}(\lambda)}}  
   & \left(
  x_i^{t(\lambda^{(p)}_{j}+n_p+1-j)}
  \right)_{\substack{1 \leq i \leq n\\ 1 \leq j \leq n_p(\lambda)}}  
   \\\\\hline\\
   \left(
  y^{t(\lambda^{(q)}_j+n_q-j)}
  \right)_{1 \leq j \leq n_{q}(\lambda)}  & \left(
  \bar{y}^{t(\lambda^{(p)}_{j}+n_p+1-j)}
  \right)_{1 \leq j \leq n_{p}(\lambda)}
\end{array}
\right)\\
= (-1)^{\frac{n_{p}(\lambda)(n_{p}(\lambda)-1)}{2}} {y}^{t(\lambda_1^{(p)}+n_{p}(\lambda))} \det \left(\begin{array}{c|c}
  A_{q,-q}^{\lambda}  
  & \bar{A}_{p,t-p}^{\lambda}
  \\\\\hline\\
    \bar{A}_{q,-q}^{\lambda}  
   & A_{p,t-p}^{\lambda} 
   \\\\\hline\\
   B_{q,-q}^{\lambda}  
  & \bar{B}_{p,t-p}^{\lambda}
\end{array}
\right). 
\end{split} 
\end{equation}
% Using the Schur polynomial expression \eqref{gldef}, 
% the denominator of $s_{\rho_{p,q}}(X^t,{\X}^t)$ is
% \[\det \left(\begin{array}{c}
% (x_i^{t(2n+1-j)})_{\substack{1 \leq i \leq n\\1 \leq j \leq 2n+1}}\\\\\hline\\
% (\x_i^{t(2n+1-j)})_{\substack{1 \leq i \leq n\\1 \leq j \leq 2n+1}}\\\\\hline\\
% (y^{t(2n+1-j)})_{1 \leq j \leq 2n+1}    
% \end{array}\right)=  V(X^t,\X^t,y^t)
% (-1)^{\frac{n(n-1)}{2}} \det \left(\begin{array}{c|c|c}
%   A_{0}  
%   & 1 &\bar{A}_{0}
%   \\&&\\\hline&&\\
%     \bar{A}_{0} & 1 
%   & A_{0} 
%   \\&&\\\hline&&\\
%   B_{0}  
%   & 1 &\bar{B}_{0}
% \end{array}
% \right).
% \]
Hence, \eqref{lemmaa1} holds. 
\end{proof}
\begin{lem}[{\cite[Lemma 2]{kratt-1999}}] 
\label{lem:DI}
The following identities hold true.
\begin{multline*}
  \det \left(x_i^{n+2-j}-\x_i^{n+2-j}\right)_{1 \leq i,j \leq n+1}\\
  = (-1)^n 
x_{n+1}^{-n}  (x_{n+1}-\x_{n+1})
\prod_{i=1}^{n} (x_i-x_{n+1})(\x_i-x_{n+1}) \det\left(x_i^{n+1-j}-\x_i^{n+1-j}\right)_{1 \leq i,j \leq n}. \end{multline*}
\begin{multline*}
\det\left(x_i^{n+1-j}+\x_i^{n+1-j}\right)_{1 \leq i,j \leq n+1} \\= (-1)^n 
x_{n+1}^{-n}  
\prod_{i=1}^{n} (x_i-x_{n+1})(\x_i-x_{n+1})  \det \left(x_i^{n-j}+\x_i^{n-j}\right)_{1 \leq i,j \leq n}.   
\end{multline*}
\begin{multline*}
 \det(x_i^{n-j+3/2}+\x_i^{n-j+3/2})_{1 \leq i,j \leq n+1} \\= (-1)^n 
x_{n+1}^{-n-1/2} (x_{n+1}+1)
\prod_{i=1}^{n} (x_i-x_{n+1})(\x_i-x_{n+1}) \det(x_i^{n-j+1/2}+\x_i^{n-j+1/2})_{1 \leq i,j \leq n}.   
\end{multline*}
\begin{multline*}
    \det(x_i^{n-j+3/2}-\x_i^{n-j+3/2})_{1 \leq i,j \leq n+1} \\= (-1)^n 
x_{n+1}^{-n-1/2}  (x_{n+1}-1)
\prod_{i=1}^{n} (x_i-x_{n+1})(\x_i-x_{n+1})
 \det(x_i^{n-j+1/2}-\x_i^{n-j+1/2})_{1 \leq i,j \leq n}.
\end{multline*}
% \begin{equation*}
%  \det(x_i^{n-j+1/2}+\x_i^{n-j+1/2})_{1 \leq i,j \leq n} 
% = \frac{1}{2} \prod_{i=1}^{n} (x_{i}^{1/2}+\x_{i}^{1/2}) \det\left(x_i^{n-j}+\x_i^{n-j}\right)_{1 \leq i,j \leq n}.   
% \end{equation*}
\end{lem}
\begin{lem}
\label{lem:matrix}
For $i = 1,\dots, k$, fix $\ell_i$, $m_i \in \mathbb{Z}^+$ such that $1+ \ell_1 + \cdots + \ell_k = m_1 + \cdots + m_k = d$. Let $S_i$ and $T_i$ be matrices of order $1 \times m_i$ and $\ell_i \times m_i$ respectively. Define a $(k+1) \times k$ block matrix
\[
U_k \coloneqq \left(
\begin{array}{ccccc}
S_1 & S_2 & \dots & S_k\\
T_1 &&& \\
 & T_2  & & \text{\huge0}\\
 && \ddots \\
  \text{\huge0} &    &   & T_k
\end{array}
\right).
\]
\begin{enumerate}
    \item If for some  $i_0 \in [k]$, 
    \[
     m_j=\begin{cases}
\ell_j+1 & j=i_0,\\
\ell_j & \text{otherwise},
\end{cases} \quad  1 \leq j \leq k,
    \]
 then 
    \begin{equation}
    \label{dtw}
      \det(U_k)= (-1)^{\sum_{i < i_0} \ell_i} \det \left(
 \begin{array}{c}
      S_{i_0} \\\hline
      T_{i_0}
 \end{array}
 \right) \ds \prod_{\substack{i=1\\i \neq i_0}}^k \det \left(
      T_i 
 \right).  
    \end{equation}
    \item Otherwise 
    \[
    \det(U_k) =  0.
    \]
\end{enumerate}
\end{lem}

\begin{proof}
It is easy to see that the lemma holds in the case when 
$m_1 \geq \ell_1+1$.  
%then $\det(U_k) =0$ Similarly if $m_1<\ell_1$, then $\det(U_k) =  0.$
%If $m_1=\ell_1+1$, then \eqref{dtw} holds.
If $m_1 \leq \ell_1$, then 
%we prove the lemma by induction on $k$.
%Permuting the first $\ell+1$ rows cyclically, 
applying the blockwise row operations $R_1 \leftrightarrow R_2$,
we see that
\[
\det U_k= 
\begin{cases}
0 & m_1<\ell_1,\\
(-1)^{\ell_1} U_{k-1} & m_1=\ell_1.
\end{cases}
\]
Proceeding recursively in the case $m_1=\ell_1$, \eqref{dtw} holds. This completes the proof. 
\end{proof}

% \tcr{change k to w}
\begin{lem}[{\cite[Lemma 3.17]{ayyer-2021}}] 
\label{lem:det-blockmatrix}
Suppose $u_1,\dots,u_k$ are positive integers summing up to $kn$.
Further, let $\left(\gamma_{i,j}\right)_{1 \leq i \leq k, 1 \leq j \leq k+1}$ 
be a matrix of parameters such that $\gamma_{i,k+1}=\gamma_{i,k}$,  $1 \leq i \leq k$ and $\Gamma$ be the square matrix consisting of its first $k$ columns.
Let $U_{j}$ and $V_{j}$ be matrices of order $n \times u_j$ for $j \in [k]$.
Finally, define a $kn \times kn$ matrix 
with $k \times k$ blocks as
\[
\ds \Pi
\coloneqq \left(
\begin{array}{c|c}
 \left(\gamma_{i,2j-1}U_{j}-\gamma_{i,2j}V_{j}\right)_{\substack{1 \leq i\leq k\\1 \leq j\leq \floor{\frac{k+1}{2}}}}
 & 
 \left(\gamma_{i,2k+2-2j}U_{j}-\gamma_{i,2k+1-2j}V_{j}\right)_{\substack{1 \leq i\leq k \\\floor{\frac{k+3}{2}} \leq j\leq k}}
 \end{array}
 \right).
\] 

\begin{enumerate}
\item If $u_p+u_{k+1-p} \neq 2n$  for some $p \in [k]$, then $\det \Pi = 0$.  

\item Else if $u_p+u_{k+1-p}=2n$ for all $p \in [k]$, then
 \begin{equation}
   \label{lem2}
 \det \Pi
 = (-1)^{\Sigma} 
 (\det \Gamma)^n
 \prod_{i=1}^{\floor{\frac{k+1}{2}}}
 \det W_i, 
 \end{equation}
where
\[
W_i= \begin{cases}
\left(\begin{array}{c|c}U_i  & -V_{k+1-i} \\\hline
-V_i  & U_{k+1-i}
\end{array}\right) & 1 \leq i \leq \floor{\frac{k}{2}}, \\
\left( U_{\frac{k+1}{2}}-V_{\frac{k+1}{2}}  \right) 
& \text{$k$ odd and $i = \frac{k+1}{2}$},
\end{cases}
\]
and
\[
\Sigma=
% \sum_{i=1}^{\floor{\frac{k}{2}}}(n+u_{i})+
\begin{cases}
0 &  k \text{ even},\\
% n \ds \sum_{i=1}^{\frac{k-1}{2}} u_i = 
n \ds \sum_{i=\frac{k+3}{2}}^{k} u_i
&  k \text{ odd}.
\end{cases} 
\]
\end{enumerate}

\end{lem}
\begin{lem} 
\label{lem:det-blockmatrix-1}
Suppose $u_1,\dots,u_k$ are positive integers summing up to $kn+1$.
Further, let $\left(\gamma_{i,j}\right)_{1 \leq i \leq k, 1 \leq j \leq k+1}$ 
be a matrix of parameters such that $\gamma_{i,k+1}=\gamma_{i,k}$,  $1 \leq i \leq k$ and $\Gamma$ be the square matrix consisting of its first $k$ columns.
Let $M_{j}$ and $N_{j}$ be matrices of order $1 \times u_j$ for $j \in [k]$, and $U_{j}$ and $V_{j}$ be matrices of order $n \times u_j$ for $j \in [k]$.
Finally, define a $(kn+1) \times (kn+1)$ matrix 
with $(k+1) \times k$ blocks as
\[
\ds \Delta
\coloneqq \left(
\begin{array}{c|c}
\left(M_j \pm N_j\right)_{1 \leq j\leq \floor{\frac{k+1}{2}}}
& \left(M_j \pm N_j\right)_{\floor{\frac{k+3}{2}} \leq j\leq k}
\\\\\hline\\
 \left(\gamma_{i,2j-1}U_{j}-\gamma_{i,2j}V_{j}\right)_{\substack{1 \leq i\leq k\\1 \leq j\leq \floor{\frac{k+1}{2}}}}
 & 
 \left(\gamma_{i,2k+2-2j}U_{j}-\gamma_{i,2k+1-2j}V_{j}\right)_{\substack{1 \leq i\leq k \\\floor{\frac{k+3}{2}} \leq j\leq k}}
 \end{array}
 \right).
\] 
\begin{enumerate}
\item If for some $1 \leq i_0 \leq \floor{\frac{k+1}{2}},$ 
\[ u_j+u_{k+1-j}=\begin{cases}
2n+1+\delta_{i_0,\frac{k+1}{2}} & j=i_0,\\
2n & \text{otherwise},
\end{cases}
\]
then
\begin{equation}
    \det \Delta
 = (-1)^{\chi} 
 (\det \Gamma)^n \det \left( \begin{array}{c}
      O_{i_0}  \\\hline 
      W_{i_0}
 \end{array} \right)
 \prod_{\substack{i=1\\i \neq i_0}}^{\floor{\frac{k+1}{2}}}
 \det W_i,
\end{equation}
where
\[
O_i= \begin{cases}
\left(\begin{array}{c|c}
   M_i \pm N_i  & M_{k+1-i} \pm N_{k+1-i} 
\end{array} \right) & \text{if } 1 \leq i \leq \floor{\frac{k}{2}},\\
\left(M_{\frac{k+1}{2}} \pm N_{\frac{k+1}{2}}\right) & \text{$k$ odd and $i = \frac{k+1}{2}$},
\end{cases}
\]
\[
W_i= \begin{cases}
\left(\begin{array}{c|c}U_i  & -V_{k+1-i} \\\hline
-V_i  & U_{k+1-i}
\end{array}\right) & 1 \leq i \leq \floor{\frac{k}{2}}, \\
\left( U_{\frac{k+1}{2}}-V_{\frac{k+1}{2}}  \right) 
& \text{$k$ odd and $i = \frac{k+1}{2}$},
\end{cases} 
\]
and
\[
\chi=
% \sum_{i=1}^{\floor{\frac{k}{2}}}(n+u_{i})+
% \left(1-\delta_{i_0,\frac{k+1}{2}}\right) 
\left(\ds \sum_{i=k+2-i_0}^{k} u_i \right)
+ \ds \sum_{i=\floor{\frac{k+3}{2}}}^{k} k n u_i.
% = 
% \begin{cases}
% \ds \sum_{i=k+2-i_0}^{k} u_i &  k \text{ even},\\
% (n+1) \left(\ds \sum_{i=k+2-i_0}^{k} u_i \right) + n \left(\ds \sum_{i=\floor{\frac{k+3}{2}}}^{k+1-i_0} u_i \right) &  k \text{ odd}.
% \end{cases}
\]
\item Otherwise \[
\det \Delta =0.
\]
\end{enumerate}

\end{lem}

\begin{proof}
Consider the permutation $\zeta$ in $S_{kn+1}$ which rearranges the columns of $\Delta$ blockwise in the following order: $1, k, 2, k-1, \dots$. 
In other words, $\zeta$ can be written in one-line notation as
\begin{multline*}  
  \zeta = (\underbrace{1,\dots,u_1}_{u_1},\underbrace{kn-u_k+2,\dots,kn+1}_{u_k}, \\
  \underbrace{u_1+1,\dots,u_1+u_2}_{u_2},\underbrace{kn-u_k-u_{k-1}+2,\dots,kn+1-u_k}_{u_{k-1}},\dots).
  \end{multline*}
Then, the number of inversions of $\zeta$ is
\begin{equation}
    % \begin{split}
        \label{inversion}
 \inv(\zeta)
 = \sum_{i=\floor{\frac{k+3}{2}}}^k u_i(kn+1-(u_1+\dots+u_{k+1-i})-(u_i+\dots+u_k)).
    % \end{split}
\end{equation}
 Here we note that
 \begin{equation}
\label{mk}
  \det \Delta =
  \sgn (\zeta)
  \det\left(\begin{array}{c} 
  M_{j''} \pm N_{j''} \\\\\hline\\
 \gamma_{i,j}U_{j''}-\gamma_{i,j'}V_{j''}
 \end{array}\right)_{1 \leq i,j\leq k},   
\end{equation}
 where 
\[
j'= j-(-1)^j \quad \text{and} \quad 
j''=\begin{cases}
 \frac{j+1}{2} & j \text{ odd}, \\
 k+1-\frac{j}{2} & j \text{ even}.
 \end{cases}
\] 
Now we see that
\begin{equation*}
\left(\begin{array}{c}
M_{j''} \pm N_{j''} \\\\\hline\\
 \gamma_{i,j}U_{j''}-\gamma_{i,j'}V_{j''}
 \end{array}\right)_{1 \leq i,j\leq k} 
 = \left( \begin{array}{c|c}
     1 & 0 \\\\ \hline\\
    0  & (\gamma_{i,j} I_n)_{1 \leq i,j\leq k}
 \end{array} \right) \times
\left(
\begin{array}{ccccc}
O_1 & O_2 && O_{\floor{\frac{k+1}{2}}}\\
W_1 &&& \\
 & W_2  & &\text{\huge0} \\
 &     & \ddots \\
\text{\huge0} &    &   & W_{\floor{\frac{k+1}{2}}}
\end{array}
\right).    
\end{equation*}
\noindent
Since the matrix $(\gamma_{i,j} I_n)_{1 \leq i,j \leq k}$ is the tensor product $\Gamma \otimes I_n$,
\[
\det \left( \begin{array}{c|c}
     1 & 0 \\\\ \hline\\
    0  & (\gamma_{i,j} I_n)_{1 \leq i,j\leq k}
 \end{array} \right)=\left(\det \Gamma \right)^n.
\]
Therefore, 
\begin{equation}
 \label{mk2}
 \det \left(\begin{array}{c}
M_{j''} \pm N_{j''} \\\\\hline\\
 \gamma_{i,j}U_{j''}-\gamma_{i,j'}V_{j''}
 \end{array}\right)_{1 \leq i,j\leq k}  =\left(\det \Gamma \right)^n 
 \det \left(
\begin{array}{ccccc}
O_1 & O_2 && O_{\floor{\frac{k+1}{2}}}\\
W_1 &&& \\
 & W_2  & &\text{\huge0} \\
 &     & \ddots \\
\text{\huge0} &    &   & W_{\floor{\frac{k+1}{2}}}
\end{array}
\right).  
\end{equation}
If 
\[ u_j+u_{k+1-j}=\begin{cases}
2n+1+\delta_{i_0,\frac{k+1}{2}} & j=i_0,\\
2n & \text{ otherwise},
\end{cases}
\]
for some $1 \leq i_0 \leq \floor{\frac{k+1}{2}},$ then using \cref{lem:matrix} in \eqref{mk2} and substituting in \eqref{mk}, we have
\begin{equation}
    \det \Delta
 = 
 (-1)^{\chi} 
 (\det \Gamma)^n \det \left( \begin{array}{c}
      O_{i_0}  \\\hline 
      W_{i_0}
 \end{array} \right)
 \prod_{\substack{i=1\\i \neq i_0}}^{\floor{\frac{k+1}{2}}}
 \det W_i.
\end{equation}
Otherwise, using \cref{lem:matrix}, the determinant of the last matrix in \eqref{mk2} is zero. Hence, by \eqref{mk},
     \[
\det \Delta = 0.
\]
This completes the proof.
\end{proof}

\section{Schur Factorization}
\label{sec:schur-k}
In this section, we give a proof of \cref{thm:schur-k}. 
\begin{lem} 
\label{lem:corek}
Fix $0 \leq m \leq t-1$ and $0 \leq \nu_m \leq \dots \leq \nu_1\leq t-m$. Let $\lambda$ be a partition of length at most $tn+m$. Then 
\[\core \lambda t=(\nu_1,\dots,\nu_m) \text{ if and only if } n_i(\lambda)=
\begin{cases}
n+1 & \text{ if } i=\nu_i+m-i \text{ for some } i\\
n & \text{ otherwise}.
\end{cases}
\]
\end{lem}
\begin{proof}
It is obvious that $\core \lambda t=(\nu_1,\dots,\nu_m)$ if and only if $\beta(\core \lambda t)=(\nu_1+tn+m-1,\dots,\nu_m+tn,tn-1,\dots,0)$. This further implies that $\core \lambda t=(\nu_1,\dots,\nu_m)$ if and only if  
\[ n_i(\lambda)=
\begin{cases}
n+1 & \text{ if } i=\nu_i+m-i \text{ for some } i\\
n & \text{ otherwise}.
\end{cases}
\]
\end{proof}

\begin{proof}[Proof of \cref{thm:schur-k}] 
By the Definition \eqref{gldef}, we see that the desired Schur polynomial is
\begin{equation}
\label{num-1}
   s_{\lambda}(X,\omega X, \dots ,\omega^{t-1}X,y, \dots, \omega^{m-1}y) 
   = \frac{\det \left(
    \begin{array}{c}
    \left( \left( 
    (\omega^{p-1}x_i)^{\beta_{j}(\lambda)}
    \right)_{\substack{1 \leq i\leq n \\ 1\leq j\leq tn+m}}
    \right)_{1\leq p \leq t} \\[0.4cm]  
      \hline     \\[-0.3cm]
       \left( (\omega^{i-1} y)^{\beta_{j}(\lambda)} \right)_{\substack{1 \leq i\leq m \\1 \leq j \leq tn+m}}
    \end{array} 
    \right)}
    {\det \left( \begin{array}{c}  
    \left( \left( (\omega^{p-1}x_i)^{tn+m-j}
    \right)_{\substack{1 \leq i\leq n \\ 1\leq j\leq tn+m}}
    \right)_{1\leq p \leq t} \\[0.4cm]  
      \hline     \\[-0.3cm]
    \left( (\omega^{i-1} y)^{tn+m-j} \right)_{\substack{1 \leq i\leq m \\1 \leq j \leq tn+m}} \end{array} \right)}.
\end{equation}
% \tcr{Permuting rows of the determinant in the numerator, we see that the numerator is} 
% \begin{equation}
% \label{num-1}
%         \det  \left(
%     \begin{array}{c} 
%     \left( \left( 
%     (\omega^{p-1}x_i)^{\beta_{j}(\lambda)}
%     \right)_{\substack{1 \leq i\leq n+1 \\ 1\leq j\leq tn+k}}
%     \right)_{1\leq p \leq k} \\[0.4cm]  
%       \hline     \\[-0.3cm]
%       \left( \left( 
%     (\omega^{p-1}x_i)^{\beta_{j}(\lambda)}
%     \right)_{\substack{1 \leq i\leq n \\ 1\leq j\leq tn+k}} \right)_{k+1 \leq p \leq t}
%       \end{array} 
%     \right)
% \end{equation}
We first consider the case when $n_{c}(\lambda) > n+1$ for some $0 \leq c \leq t-1$. 
Permuting the columns of the matrix in the numerator in \eqref{num-1} by $\sigma^c_{\lambda}$ from \eqref{sigma-perm-m} $(m=1,e_1=c)$, we see that the numerator in the right hand side of \eqref{num-1} is
\begin{equation}
\sgn(\sigma^c_{\lambda}) \det \left( \begin{array}{c|c}
     (\omega^{(p-1)(c)} A^{\lambda}_{c})_{1 \leq p \leq t}    &  (\omega^{(p-1)(j-1)} A^{\lambda}_{j-1})_{\substack{1 \leq p \leq t\\ 1 \leq j \leq t \\ j \neq c+1}}    \\\\ \hline\\
    (\omega^{(p-1)(c)} B^{\lambda}_{c})_{1 \leq p \leq m}     &  (\omega^{(p-1)(j-1)} B^{\lambda}_{j-1})_{\substack{1 \leq p \leq m \\ 1 \leq j \leq t \\ j \neq c+1}} 
    \end{array} \right),    
\end{equation}
where $A^{\lambda}_s
 =\left(x_i^{\beta^{(s)}_j(\lambda)}\right)_{\substack{1 \leq i \leq n \\ 1 \leq j \leq n_s(\lambda)}}
$ and $B^{\lambda}_s
 =\left(y^{\beta^{(s)}_j(\lambda)}
 \right)_{1 \leq j \leq n_s(\lambda)}.
$ 
% are defined in \eqref{def AB} and \eqref{def 1AB} respectively. 
For $p \in [t],$ multiplying the rows of the $p^{\text{th}}$'th block by $\omega^{(1-p)c}$ and for $p \in [m]$ and multiplying the row of the $(t+p)^{\text{th}}$ block by $\omega^{(1-p)c}$, we get
\begin{equation}
    \sgn(\sigma^c_{\lambda}) \det \left( \begin{array}{c|c}
     ( A^{\lambda}_{c})_{1 \leq p \leq t}    &  (\omega^{(p-1)(j-c-1)} A^{\lambda}_{j-1})_{\substack{1 \leq p \leq t\\ 1 \leq j \leq t \\ j \neq c+1}}    \\\\ \hline\\
    (B^{\lambda}_{c})_{1 \leq p \leq m}     &  (\omega^{(p-1)(j-c-1)} B^{\lambda}_{j-1})_{\substack{1 \leq p \leq m \\ 1 \leq j \leq t \\ j \neq c+1}} 
    \end{array} \right)
\end{equation}
% \begin{equation}
% \sgn(\sigma^v_{\lambda}) \det \left( \begin{array}{c|c}
%      (A^{\lambda}_{v})_{1 \leq p \leq k}    &  (\omega^{(p-1)(j-v-1)} A^{\lambda}_{j-1})_{\substack{1 \leq p \leq k\\ 1 \leq j \leq t \\ j \neq v+1}}    \\\\ \hline\\
%     ( B^{\lambda}_{v})_{k+1 \leq p \leq t}     &  (\omega^{(p-1)(j-v-1)} B^{\lambda}_{j-1})_{\substack{k+1 \leq p \leq t \\ 1 \leq j \leq t \\ j \neq v+1}} 
%     \end{array} \right),    
% \end{equation}
Applying the blockwise row operations $R_i \rightarrow R_i-R_1$ for $i \in [2,t]$, $R_i \rightarrow R_i-R_{t+1}$ for $i \in [t+2,t+m]$ and then permuting rows $R_i$, $i \in [2,t+1]$ cyclically, we see that the numerator is 
\begin{equation}
\label{iszero}
\sgn(\sigma^c_{\lambda}) (-1)^{t-1} 
\det \left( \begin{array}{c|c}
A^{\lambda}_{c} & (A^{\lambda}_{j-1})_{\substack{1 \leq j \leq t \\ j \neq c+1}} \\\\\hline\\
B^{\lambda}_{c} & (B^{\lambda}_{j-1})_{\substack{1 \leq j \leq t \\ j \neq c+1}}
\\\\\hline\\
(0)_{2 \leq p \leq t}    &  ((\omega^{(p-1)(j-c-1)}-1) A^{\lambda}_{j-1})_{\substack{2 \leq p \leq t\\ 1 \leq j \leq t \\ j \neq c+1}}    
\\\\ \hline\\
(0)_{2 \leq p \leq m}     &  ((\omega^{(p-1)(j-c-1)}-1) B^{\lambda}_{j-1})_{\substack{2 \leq p \leq m \\ 1 \leq j \leq t \\ j \neq c+1}} 
    \end{array} \right).    
\end{equation}
Since $n_{c}(\lambda) > n+1$, the determinant in \eqref{iszero} is zero. Substituting in \eqref{num-1}, we see that the required Schur polynomial vanishes. Now consider the case when
$n_{i}(\lambda) \leq n+1$ for all $i \in [0,t-1]$. Since $\sum_i n_i(\lambda)=tn+m$,
using pigeonhole principle there exist $\{e_1,\dots,e_m\} \subset [0,t-1]$ such that $n_{e_i}(\lambda)=n+1$, $i \in [m]$. Let $e_i=\nu_i+m-i$ for all $i \in [m]$.  
Permuting the columns of the determinant in the numerator of \eqref{num-1} by $\sigma^{E}_{\lambda}$ from \eqref{sigma-perm-m}, we see that the numerator is
\begin{equation}
\sgn(\sigma^{\nu+\delta_m}_{\lambda}) \det \left( \begin{array}{c|c}
     (\omega^{(p-1)(e_j)} A^{\lambda}_{e_j})_{\substack{1 \leq p \leq t\\1 \leq j \leq m}} &
     (\omega^{(p-1)(j-1)} A^{\lambda}_{j-1})_{\substack{1 \leq p,j \leq t\\ j \neq e_1+1,\dots,e_m+1}}
     \\\\ \hline\\
    (\omega^{(p-1)(e_j)} B^{\lambda}_{e_j})_{\substack{1 \leq p \leq m \\1 \leq j \leq m}} &
     (\omega^{(p-1)(j-1)} B^{\lambda}_{j-1})_{\substack{1 \leq p,j \leq t\\ j \neq e_1+1,\dots,e_m+1}}
    \end{array} \right)
    % = \sgn(\sigma_{\lambda}) \det \left( \begin{array}{c}
    %  (\omega^{(p-1)(j-1)} C^{\lambda}_{j-1})_{\substack{1 \leq p \leq k \\1 \leq j \leq t}}    \\\\ \hline\\
    % (\omega^{(p-1)(j-1)} A^{\lambda}_{j-1})_{\substack{k+1 \leq p \leq t \\1 \leq j \leq t}} 
    % \end{array} \right),    
\end{equation}
% where $C^{\lambda}_s=(x_i^{\beta^{(s)}_j(\lambda)})_{\substack{1 \leq i \leq n+1 \\ 1 \leq j \leq n_s}}$ and $x_{n+1}=y$. 
Consider the permutation $\sigma^*$ in $S_{tn+m}$ which rearranges the rows of the numerator blockwise as: $1, t+1, 2, t+2, \dots, m, t+m, m+1, \dots, t$. Then it can be seen that the numerator is
% In other words,  can be written in one line notation
% as We note that
% \[
% \left( \begin{array}{c}
%      (\omega^{(p-1)(j-1)} C^{\lambda}_{j-1})_{\substack{1 \leq p \leq k \\1 \leq j \leq t}}    \\\\ \hline\\
%     (\omega^{(p-1)(j-1)} A^{\lambda}_{j-1})_{\substack{k+1 \leq p \leq t \\1 \leq j \leq t}} 
%     \end{array} \right)
%     = \left( 
%     \begin{array}{c}
%      \left(\gamma_{i,j} I_{n+1 \times n_j}\right)_{\substack{1 \leq i \leq k \\1 \leq j \leq t}}  \\\\\hline\\
%     \left(\gamma_{i,j} I_{ n \times n_j}\right)_{\substack{k+1 \leq i \leq t \\1 \leq j \leq t}}
%     \end{array}
% \right) 
% \left(
% \begin{array}{cccccc}
% C^{\lambda}_0  \\
% & \ddots  &  & & \text{\huge0}\\
% & & C^{\lambda}_{k-1} & \\
% & &  & A^{\lambda}_{k} \\
% & \text{\huge0} &  & & \ddots \\
% &  &   & &  & A^{\lambda}_{t-1} 
% \end{array}
% \right).
% \]
\begin{equation}
\label{nu}
 \sgn(\sigma^*) \sgn(\sigma^{\nu+\delta_m}_{\lambda}) \det(\Gamma_m(\lambda)) \det
\left(
\begin{array}{cccccc}
\left( \begin{array}{c}
     A^{\lambda}_{e_1}  \\\hline
     B^{\lambda}_{e_1}
\end{array} \right) \\
& \ddots  &  & & \text{\huge0}\\
& & \left( \begin{array}{c}
     A^{\lambda}_{e_m}  \\\hline
     B^{\lambda}_{e_m}
\end{array} \right) & \\
& &  & A^{\lambda}_{e_{m+1}} \\
& \text{\huge0} &  & & \ddots \\
&  &   & &  & A^{\lambda}_{e_t} 
\end{array}
\right),    
\end{equation}
where the set $\{e_{m+1}< \cdots < e_t\}$ is same as $\{0, \ldots, t-1\} \setminus \{e_i\}_{i\in [m]}$ and 
\[ \Gamma_m(\lambda) \coloneqq 
\left( 
    \begin{array}{c|c}
     \left(\omega^{(i-1)(e_j)} I_{n+1 \times n+1}\right)_{\substack{1 \leq i \leq m \\1 \leq j \leq m}}  &
     \left(\omega^{(i-1)(j-1)} I_{n+1 \times n}\right)_{\substack{1 \leq i \leq m \\1 \leq j \leq t\\j \neq e_1+1,\dots,e_m+1
     }}
     \\\\\hline\\
    \left(\omega^{(i-1)(e_j)} I_{ n \times n+1}\right)_{\substack{m+1 \leq i \leq t \\1 \leq j \leq m}} &
    \left(\omega^{(i-1)(j-1)} I_{ n \times n }\right)_{\substack{m+1 \leq i \leq t \\1 \leq j \leq t\\
    j \neq e_1+1,\dots,e_m+1}}
    \end{array}
\right). 
\]
% and
% \[
% \left(\gamma_{i,j}\right)_{1 \leq i,j \leq t} \coloneqq  \left( \begin{array}{c|c}
%  (\omega^{(i-1)(d_j)})_{\substack{1 \leq i \leq t\\1 \leq j \leq m}} &
%      (\omega^{(i-1)(j-1)})_{\substack{1 \leq i,j \leq t\\ j \neq d_1+1,\dots,d_m+1}} 
% \end{array} \right).
% \]
We note that the last determinant in \eqref{nu} is non-zero if and only if
\[
n_j(\lambda)=\begin{cases}
n+1 & \text{ if } j=e_1, \dots, e_m,\\
n & \text{otherwise}.
\end{cases}
\]
So, by \cref{lem:corek}, we see that the Schur polynomial is non-zero if and only if $\core \lambda t=(\nu_1,\dots,\nu_m)$. In this case, the numerator is
\begin{equation}
\label{nums}
  \sgn(\sigma^*) \sgn(\sigma^{\nu+\delta_m}_{\lambda}) \det(\Gamma_m(\lambda)) \prod_{i=1}^{m} \det \left( \begin{array}{c}
     A^{\lambda}_{e_i}  \\\hline
     B^{\lambda}_{e_i}
\end{array} \right) \prod_{\substack{i=0\\i \neq e_i, i \in [m]}}^{t-1} \det A^{\lambda}_{i}.  
\end{equation}
Permuting the columns $C_{sn+s},\dots,C_{tn+m}$ of $\Gamma_m(\lambda)$ cyclically in succession for $s=1,\dots,m$ and then rows in the similar way, we have
\[ \det \Gamma_m(\lambda) =  
\left( 
    \begin{array}{c|c|c}
     \left(\omega^{(i-1)(e_j)} I_{n \times n}\right)_{\substack{1 \leq i \leq t \\1 \leq j \leq m}}  & \left(\omega^{(i-1)(j-1)} I_{n \times n}\right)_{\substack{1 \leq i \leq t \\1 \leq j \leq t\\j \neq e_1+1,\dots,e_m+1}} & 
     \text{\huge0}
     \\&&\\\hline&&\\
   \text{\huge0} & \text{\huge0} & \left(\omega^{(i-1)(e_j)} \right)_{1 \leq i,j \leq m} 
    \end{array}
\right). 
\]
Finally, we evaluate $\det \Gamma_m(\lambda)$ at the empty partition and note that 
\begin{equation}
    \frac{\det \Gamma_m(\lambda)}
    {\det \Gamma_m(\emptyset)} = 
    % \pm 
    % \frac{\left(\omega^{(i-1)(d_j)} \right)_{1 \leq i,j \leq m}}{\left(\omega^{(i-1)(j-1)} \right)_{1 \leq i,j \leq m} } = 
    s_{\core \lambda t}(1,\omega,\dots,\omega^{m-1}).
\end{equation}
Since the denominator in \eqref{num-1} is same as the numerator evaluated at the empty partition. Evaluating \eqref{nums} for the empty partition and substituting in \eqref{schur-k} completes the proof.
\end{proof}
\section{Factorization of Other classical characters}
\label{sec:other}
In this section, we prove \cref{thm:odd}, \cref{thm:symp} and \cref{thm:even}. Recall, the matrices 
$A_{p,q}^{\lambda}$,  $\bar{A}_{p,q}^{\lambda}$ from \eqref{def AB} and $B_{p,q}^{\lambda}$,  $\bar{B}_{p,q}^{\lambda}$ from \eqref{def 1AB}.

\subsection{odd orthogonal} Consider the $(tn+1) \times (tn+1)$ block matrix
\begin{equation}
    \label{delta1}
    \Delta_1 \coloneqq \left( 
    \begin{array}{c}
(B_{q-1,1}^{\lambda} - \bar{B}_{q-1,0}^{\lambda})_{1 \leq q \leq t}  \\\\\hline \\
\left( 
\omega^{(p-1)q} A^{\lambda}_{q-1,1}
-\bar{\omega}^{(p-1)(q-1)}\bar{A}^{\lambda}_{q-1,0}
\right)_{1\leq p,q\leq t}
    \end{array}
    \right). 
\end{equation}
Substituting $M_j=B_{j-1,1}^{\lambda}$, $N_j=\bar{B}_{j-1,0}^{\lambda}$, $U_j=A_{j-1,1}^{\lambda}$, $V_j=\bar{A}_{j-1,0}^{\lambda}$ for $1 \leq j \leq t$ and
\[
\gamma_{i,j} = \begin{cases}
\omega^{\frac{(i-1)(j+1)}{2}} & j \text{ odd }\\
\omega^{-\frac{(i-1)(j-2)}{2}} & j \text{ even } 
\end{cases}
\]
in \cref{lem:det-blockmatrix-1} proves the following corollary.
\begin{cor}
\label{cor:delta1}

\begin{enumerate}
    \item If \begin{equation}
    n_{i}(\lambda)+n_{t-1-i}(\lambda) = 
    \begin{cases}
    2n+1+\delta_{i_0,\frac{t-1}{2}} & \text{ if } i=i_0,\\
    2n & \text{ otherwise},
    \end{cases}
    \quad \quad  0 \leq i \leq  \floor{\frac{t-1}{2}}, 
\end{equation}
for some $0 \leq i_0 \leq \floor{\frac{t-1}{2}}$, then \begin{equation}
\label{delta31}
    \det \Delta_1
 = (-1)^{\chi_1} 
 (\det \Gamma)^n \det \left( \begin{array}{c}
      O_{i_0}^{(1)}  \\[0.1cm]
     \hline \\[-0.4cm]
      W_{i_0}^{(1)}
 \end{array} \right)
 \prod_{\substack{i=0\\i \neq i_0}}^{\floor{\frac{t-1}{2}}}
 \det W_i^{(1)},
\end{equation}
where 
\[
O_i^{(1)}= \begin{cases}
\left(\begin{array}{c|c}
   B_{i,1}^{\lambda} - \bar{B}_{i,0}^{\lambda}  & B_{t-1-i,1}^{\lambda}  - \bar{B}_{t-1-i,0}^{\lambda}  
\end{array} \right) & \text{if  } 0 \leq i \leq \floor{\frac{t-2}{2}}\\
\left(B_{\frac{t-1}{2},1}^{\lambda}  - \bar{B}_{\frac{t-1}{2},0}^{\lambda} \right) & \text{$t$ odd and $i = \frac{t-1}{2}$},
\end{cases}
\]
\[
W_i^{(1)}= \begin{cases}
\left(\begin{array}{c|c}
A_{i,1}^{\lambda}  & -\bar{A}_{t-1-i,0}^{\lambda} \\[0.1cm]
     \hline \\[-0.4cm]
-\bar{A}_{i,0}^{\lambda}  & A_{t-1-i,1}^{\lambda}
\end{array}\right) 
& 0 \leq i \leq \floor{\frac{t-2}{2}}, \\
\left( A_{\frac{t-1}{2},1}^{\lambda} - \bar{A}_{\frac{t-1}{2},0}^{\lambda}  \right) 
& \text{$t$ odd and $i = \frac{t-1}{2}$},
\end{cases}
\]
and
\[
\chi_{1}
=
% \sum_{q=1}^{\floor{\frac{t}{2}}} 
% \left(n+n_{q-1}(\lambda)\right)+
\left(\ds \sum_{i=t+1-i_0}^{t} n_{i-1}(\lambda) \right)
+ \ds \sum_{i=\floor{\frac{t+3}{2}}}^{t} t n n_{i-1}(\lambda).
\]
\item Otherwise 
\begin{equation}
\label{delta3-other}
    \det \Delta_1=0.
\end{equation}
\end{enumerate}
\end{cor}
\begin{proof}[Proof of \cref{thm:odd}]
By \eqref{oodef}, we see that the numerator of desired odd orthogonal character is given by
\begin{equation}
\label{odd-1}
    % \frac{
    \det \left(
    \begin{array}{c}
    \left( \left( 
    (\omega^{p-1}x_i)^{\beta_{j}(\lambda)+1}-(\bar{\omega}^{p-1}\x_i)^{\beta_{j}(\lambda)}
    \right)_{\substack{1 \leq i\leq n \\ 1\leq j\leq tn+1}}
    \right)_{1\leq p \leq t}  \\[0.4cm]  
      \hline     \\[-0.3cm]
       \left( 
    y^{\beta_{j}(\lambda)+1}-\bar{y}^{\beta_{j}(\lambda)}
    \right)_{1\leq j\leq tn+1}
    \end{array} 
    \right).
    % }
    % { \det \left(
    % \begin{array}{c}
    % \left( \left( 
    % (\omega^{p-1}x_i)^{tn+2-j}-(\bar{\omega}^{p-1}\x_i)^{tn+1-j}
    % \right)_{\substack{1 \leq i\leq n \\ 1\leq j\leq tn+1}}
    % \right)_{1\leq p \leq t}  \\  
    %   \hline     \\[-0.3cm]
    %  \left( y^{tn+2-j}-\bar{y}^{tn+1-j} \right)_{1\leq j\leq tn+1}
    % \end{array} 
    % \right)}.
\end{equation}
Permuting the columns of the matrix in \eqref{odd-1} by $\sigma_{\lambda}$ from \eqref{sigma-perm-m} $(m=1,d_1=0)$ and then the $t+1$ row blocks of the numerator cyclically, the numerator is
% \begin{equation}
%     \sgn(\sigma_{\lambda}) \det \left(
%     \begin{array}{c|c|c|c}
%     A_{0,1}^{\lambda}-\bar{A}^{\lambda}_0 & A_{1,1}^{\lambda}-\bar{A}^{\lambda}_1 & \dots & A_{t-1,1}^{\lambda}-\bar{A}^{\lambda}_{t-1} \\\hline&&&\\
%   \omega A_{0,1}^{\lambda}-\bar{A}^{\lambda}_0 & \omega^2 A_{1,1}^{\lambda}-\omega^{t-1} \bar{A}^{\lambda}_1 & \dots & A^{\lambda}_{t-1,1}-\omega \bar{A}^{\lambda}_{t-1} \\\hline&&&\\
%     \vdots &  \vdots &  \dots &  \vdots \\\hline&&&\\
%   \omega^{t-1} A_{0,1}^{\lambda}-\bar{A}^{\lambda}_0 &  \omega^{t-2} A_{1,1}^{\lambda}-\omega \bar{A}^{\lambda}_1 &  \dots &  A^{\lambda}_{t-1,1}-\omega^{t-1} \bar{A}^{\lambda}_{t-1} 
%     \\\hline&&&\\
%     B_{0,1}^{\lambda}-\bar{B}^{\lambda}_0 & B_{1,1}^{\lambda}-\bar{B}^{\lambda}_1 & \dots & B_{t-1,1}^{\lambda}-\bar{B}^{\lambda}_{t-1}
%     \end{array} \right).
% \end{equation}
% where $A_m^{\lambda}=\left(
% x_i^{\beta^m_j(\lambda)}
% \right)_{\substack{1 \leq i \leq n\\1 \leq j \leq n_m(\lambda)}}$, 
% $\bar{A}_m^{\lambda}=\left(
% \x_i^{\beta^m_j(\lambda)}
% \right)_{\substack{1 \leq i \leq n\\1 \leq j \leq n_m(\lambda)}}$,
% $B_m^{\lambda}=\left(
% y^{\beta^m_j(\lambda)}
% \right)_{1 \leq j \leq n_m(\lambda)}$, 
% $\bar{B}_m^{\lambda}=\left(
% \bar{y}^{\beta^m_j(\lambda)}
% \right)_{1 \leq j \leq n_m(\lambda)}$.
\begin{equation}
\label{odd-2}
    \sgn(\sigma_{\lambda}) (-1)^{t} \det \left(
    \begin{array}{c|c|c|c}
    B_{0,1}^{\lambda}-\bar{B}^{\lambda}_0 & B_{1,1}^{\lambda}-\bar{B}^{\lambda}_1 & \dots & B_{t-1,1}^{\lambda}-\bar{B}^{\lambda}_{t-1}
     \\\hline&&&\\
    A_{0,1}^{\lambda}-\bar{A}^{\lambda}_0 & A_{1,1}^{\lambda}-\bar{A}^{\lambda}_1 & \dots & A_{t-1,1}^{\lambda}-\bar{A}^{\lambda}_{t-1}
     \\\hline&&&\\
   \omega A_{0,1}^{\lambda}-\bar{A}^{\lambda}_0 & \omega^2 A_{1,1}^{\lambda}-\omega^{t-1} \bar{A}^{\lambda}_1 & \dots & A^{\lambda}_{t-1,1}-\omega \bar{A}^{\lambda}_{t-1} \\\hline&&&\\
    \vdots &  \vdots &  \dots &  \vdots \\\hline&&&\\
   \omega^{t-1} A_{0,1}^{\lambda}-\bar{A}^{\lambda}_0 &  \omega^{t-2} A_{1,1}^{\lambda}-\omega \bar{A}^{\lambda}_1 &  \dots &  A^{\lambda}_{t-1,1}-\omega^{t-1} \bar{A}^{\lambda}_{t-1} 
    \end{array} \right),
\end{equation}
where $A_{p,q}^{\lambda}$,  $\bar{A}_{p}^{\lambda}$, 
$B_{p,q}^{\lambda}$ and $\bar{B}_{p}^{\lambda}$ are defined in \eqref{def AB} and \eqref{def 1AB}. We note that the matrix in \eqref{odd-2} is $\Delta_1$ defined in \eqref{delta1}. We use \cref{cor:delta1} to get the determinant. Since the denominator in \eqref{oodef} is same as its numerator evaluated at the empty partition and $n_0(\emptyset,tn+1)=n+1$, $n_i(\emptyset,tn+1)=n$ for all $i \in [1,t-1]$, evaluating the numerator in \eqref{odd-2} and then using \eqref{delta31}, we see that the denominator of the desired odd orthogonal character is  

\begin{equation}
\label{dnm-11}
\begin{split}
   \sgn(\sigma_{\emptyset}) 
 (-1)^{t} & (-1)^{\chi^{(0)}_1} 
 (\det \Gamma)^n 
 \det \left( \begin{array}{c|c}
     B_{0,1} - \bar{B}_{0,0} 
     &  B_{t-1,1}  - \bar{B}_{t-1,0}
     \\[0.2cm]
     \hline \\[-0.3cm]
     A_{0,1}  & -\bar{A}_{t-1,0} \\[0.2cm]
     \hline \\[-0.3cm]
-\bar{A}_{0,0}  & A_{t-1,1}
 \end{array} \right) \\
\times & \prod_{q=1}^{\floor{\frac{t-2}{2}}}
 \det \left(\begin{array}{c|c}
A_{q,1} & -\bar{A}_{t-1-q,0} \\[0.2cm]
     \hline \\[-0.3cm]
-\bar{A}_{q,0}  & A_{t-1-q,1}
\end{array}\right) \times 
 \begin{cases}
\det \left( A_{\frac{t-1}{2},1} - \bar{A}_{\frac{t-1}{2},0}  \right)  & t \text{ is odd},\\
1 & t \text{ is even},
 \end{cases}
 \end{split}
\end{equation}
where 
\[
\chi^{(0)}_1 = \ds \sum_{i=\floor{\frac{t+3}{2}}}^{t} t n^2.
\]
\noindent
If $\core \lambda t \not \in \mathcal{Q}^{(t)}_{2,0,k} \cup \mathcal{Q}^{(t)}_{2,1,k}$ for all $k \in [\rk(\core \lambda t)]$, then by \cref{cor:z=0} and \eqref{delta3-other}, the numerator in \eqref{odd-2} is 0. So,
\[
\oo_{\lambda}(X,\, \omega X, \dots,\omega^{t-1}X,y)=0.
\]
If $\core \lambda t \in \mathcal{Q}^{(t)}_{2,0,k} \cup \mathcal{Q}^{(t)}_{2,1,k}$ for some $1 \leq k \leq \rk(\core \lambda t)$, then we use \cref{cor:z=0} and \eqref{delta31} to factorize the numerator in \eqref{odd-2}.

\underline{Case 1}. If $t$ is odd and $i_0=\frac{t-1}{2}$, then the numerator is 
\begin{equation}
\label{num-12}
 \sgn(\sigma_{\lambda}) 
 (-1)^{t} (-1)^{\chi_1} 
 (\det \Gamma)^n 
 \det \left( \begin{array}{c}
      B_{\frac{t-1}{2},1}^{\lambda}  - \bar{B}_{\frac{t-1}{2},0}^{\lambda}  \\[0.2cm]
     \hline \\[-0.3cm]
      A_{\frac{t-1}{2},1}^{\lambda} - \bar{A}_{\frac{t-1}{2},0}^{\lambda} 
 \end{array} \right)
 \prod_{q=0}^{\frac{t-3}{2}}
 \det \left(\begin{array}{c|c}
A_{q,1}^{\lambda} 
& -\bar{A}_{t-1-q,0}^{\lambda} \\[0.2cm]
     \hline \\[-0.3cm]
-\bar{A}_{q,0}^{\lambda}  
& A_{t-1-q,1}^{\lambda}
\end{array}\right).
\end{equation}
By \cref{lem:DI}, we have  
\begin{multline*}
\det \left( \begin{array}{c|c}
     B_{0,1} - \bar{B}_{0,0} 
     &  B_{t-1,1}  - \bar{B}_{t-1,0}
     \\[0.2cm]
     \hline \\[-0.3cm]
     A_{0,1}  & -\bar{A}_{t-1,0} \\[0.2cm]
     \hline \\[-0.3cm]
-\bar{A}_{0,0} & A_{t-1,1}
 \end{array} \right) \\
% \[
% = (-1)^{\frac{n(n-1)}{2}}\prod_i  x_i \left(y^{1-tn} V(X^t,\X^t,y^t) + y^{tn}  V(X^t,\X^t,\bar{y}^t) \right)
% \]
% \[
% = (-1)^{\frac{n(n-1)}{2}}\prod_i  x_i V(X^t,\X^t) \left(y^{1-tn} \prod_i (x_i^t-y^t)(\x_i^t-y^t) + y^{tn} \prod_i (x_i^t-\bar{y}^t)(\x_i^t-\bar{y}^t) \right)
% \] 
= \bar{y}^{tn}  (y  - 1) 
 \prod_{i=1}^n \left( (x_i^t-y^t)(\x_i^t-y^t)  \right)
\det \left( \begin{array}{c|c}
     A_{0,1}  & -\bar{A}_{t-1,0} \\[0.2cm]
     \hline \\[-0.3cm]
-\bar{A}_{0,0} & A_{t-1,1}
 \end{array} \right)  
\end{multline*}
and 
\begin{multline*}
\det \left( \begin{array}{c}
      B_{0,\frac{t+1}{2}}  - \bar{B}_{0,\frac{t-1}{2}} \\[0.2cm]
     \hline \\[-0.3cm]
      A_{0,\frac{t+1}{2}} - \bar{A}_{0,\frac{t-1}{2}}
 \end{array} \right) \\
%  = (y x_i)^{\frac{1}{2}}
%  \det \left( \begin{array}{c}
%       \left(y^{t(n-j+3/2)}-  \bar{y}^{t(n-j+3/2)}\right)_{1 \leq j \leq n+1} \\\hline
%       \left(x_i^{t(n-j+3/2)} - \x_i^{t(n-j+3/2)}\right)_{1 \leq i,j \leq n+1}
%  \end{array} \right)\]
%  \[
 =
 (-1)^{n} y^{-tn+(1-t)/2}  (y^t-1)
\prod_{i=1}^{n}\left( (x_i^t-y^t)(\x_i^t-y^t) \right)
 \det \left( A_{\frac{t-1}{2},1} - \bar{A}_{\frac{t-1}{2},0}  \right).
\end{multline*}
Substituting in \eqref{dnm-11}, we see that the denominator in this case is
\begin{equation}
\label{dnm-12}
\begin{split}
   \sgn(\sigma_{\emptyset}) 
 (-1)^{t+n}  (-1)^{\chi^{(0)}_1} 
 (\det \Gamma)^n
 \frac{y-1}{y^{(1-t)/2}  (y^t-1)} & \prod_{q=0}^{{\frac{t-3}{2}}}
 \det \left(\begin{array}{c|c}
A_{q,1} & -\bar{A}_{t-1-q,0} \\[0.2cm]
     \hline \\[-0.3cm]
-\bar{A}_{q,0}  & A_{t-1-q,1}
\end{array}\right)  \\
& \times 
\det \left( \begin{array}{c}
      B_{0,\frac{t+1}{2}}  - \bar{B}_{0,\frac{t-1}{2}} \\[0.2cm]
     \hline \\[-0.3cm]
      A_{0,\frac{t+1}{2}} - \bar{A}_{0,\frac{t-1}{2}}
 \end{array} \right).  
 \end{split}
\end{equation}

For $q \in [0,\frac{t-3}{2}]$, multiplying by $x_i^{-q-1}$ to the $i^{\text{th}}$ row in upper blocks and by $\x_i^{-q}$ to the $i^{\text{th}}$ row in lower blocks for $i \in [n]$, both in numerator and denominator, and then by \cref{lem:s-new}, we have 
\begin{equation}
\label{scur}
\frac{\det \left(\begin{array}{c|c}
A_{q,1}^{\lambda} 
& -\bar{A}_{t-1-q,0}^{\lambda} \\[0.2cm]
     \hline \\[-0.3cm]
-\bar{A}_{q,0}^{\lambda}  
& A_{t-1-q,1}^{\lambda}
\end{array}\right)}
{\det \left(\begin{array}{c|c}
A_{q,1} 
& -\bar{A}_{t-1-q,0} \\[0.2cm]
     \hline \\[-0.3cm]
-\bar{A}_{q,0}  
& A_{t-1-q,1}
\end{array}\right)}
% = 
% \frac{\det \left(\begin{array}{c|c}
% A_{q,-q}^{\lambda} 
% & -\bar{A}_{t-1-q,q+1}^{\lambda} \\\hline
% -\bar{A}_{q,-q}^{\lambda}  
% & A_{t-1-q,q+1}^{\lambda}
% \end{array}\right)}
% {\det \left(\begin{array}{c|c}
% A_{q,-q} 
% & -\bar{A}_{t-1-q,q+1} \\\hline
% -\bar{A}_{q,-q}  
% & A_{t-1-q,q+1} 
% \end{array}\right)} 
= \frac{(-1)^{\frac{n_{t-1-q}(\lambda)
(n_{t-1-q}(\lambda)+1)}
{2}}}
{(-1)^{\frac{n
(n+1)}
{2}}} s_{\pi_q^{(1)}}(X^t,\X^t).
\end{equation}
Taking the ratio of \eqref{num-12} and \eqref{dnm-12} and using \eqref{sigma0m}, \eqref{scur} and \eqref{oo-new-1}, we see that the required odd orthogonal character is 
\[
 \sgn(\sigma_{\lambda})  
 (-1)^{\epsilon+n} \frac{y-1}{y^{(1-t)/2}  (y^t-1)}  \prod_{q=0}^{{\frac{t-3}{2}}} s_{\pi_q^{(1)}}(X^t,\X^t) \times \oo_{\lambda^{\left(\frac{t-1}{2}\right)}}(X^t),  
\]
where 
\begin{equation}
\label{epsln}
\begin{split}
\epsilon = \frac{t(t-1)}{2}\frac{n(n+1)}{2}+&\left(\ds \sum_{i=t+1-i_0}^{t}  n_{i-1}(\lambda) \right)
+\ds \sum_{i=\floor{\frac{t+3}{2}}}^{t} t n (n_{i-1}(\lambda)-n) \\
+ &\sum_{q=0}^{\floor{\frac{t-2}{2}}} \left( \frac{n_{t-1-q}(\lambda)
(n_{t-1-q}(\lambda)+1)}
{2} - \frac{n(n+1)}
{2} \right).
\end{split}
\end{equation}
Since $\frac{(t+1)(t-1)}{2}\frac{n(n+1)}{2}$ is even for odd $t$, the parity of $\epsilon$  is same as $\epsilon_1(\lambda)$ defined in \eqref{epsilon}.

\underline{Case 2}. If $i_0 \neq \frac{t-1}{2}$, then \eqref{delta31} for the determinant in \eqref{odd-2}, we see that the numerator is
\begin{equation}
\label{num-11}
\begin{split}
   \sgn(\sigma_{\lambda}) &
  (-1)^{\chi_1+t} 
 (\det \Gamma)^n 
 \det \left( \begin{array}{c|c}
     B_{i_0,1}^{\lambda} - \bar{B}_{i_0,0}^{\lambda} 
     &  B_{t-1-i_0,1}^{\lambda}  - \bar{B}_{t-1-i_0,0}^{\lambda} 
     \\[0.2cm]
     \hline \\[-0.3cm]
     A_{i_0,1}^{\lambda}  & -\bar{A}_{t-1-i_0,0}^{\lambda} \\[0.2cm]
     \hline \\[-0.3cm]
-\bar{A}_{i_0,0}^{\lambda}  & A_{t-1-i_0,1}^{\lambda}
 \end{array} \right) \\
\times & \prod_{\substack{q=0\\q \neq i_0}}^{\floor{\frac{t-2}{2}}}
 \det \left(\begin{array}{c|c}
A_{q,1}^{\lambda}  
& -\bar{A}_{t-1-q,0}^{\lambda} \\[0.2cm]
     \hline \\[-0.3cm]
-\bar{A}_{q,0}^{\lambda}  
& A_{t-1-q,1}^{\lambda}
\end{array}\right) \times 
 \begin{cases}
\det \left( A_{\frac{t-1}{2},1}^{\lambda} - \bar{A}_{\frac{t-1}{2},0}^{\lambda}  \right)  & t \text{ is odd}.\\
1 & t \text{ is even}.
 \end{cases}
 \end{split}
\end{equation}
For $q \in [0,\floor{\frac{t-3}{2}}] \setminus \{i_0\}$, multiplying by $x_i^{-q-1}$ to the $i^{\text{th}}$ row in upper blocks and by $\x_i^{-q}$ to the $i^{\text{th}}$ row in lower blocks for $i \in [n]$, both in numerator and denominator, and then by \cref{lem:s-new}, we have 
\begin{equation}
\label{scuri0}
\frac{\det \left(\begin{array}{c|c}
A_{q,1}^{\lambda} 
& -\bar{A}_{t-1-q,0}^{\lambda} \\[0.2cm]
     \hline \\[-0.3cm]
-\bar{A}_{q,0}^{\lambda}  
& A_{t-1-q,1}^{\lambda}
\end{array}\right)}
{\det \left(\begin{array}{c|c}
A_{q,1} 
& -\bar{A}_{t-1-q,0} \\[0.2cm]
     \hline \\[-0.3cm]
-\bar{A}_{q,0}  
& A_{t-1-q,1}
\end{array}\right)}
= \frac{(-1)^{\frac{n_{t-1-q}(\lambda)
(n_{t-1-q}(\lambda)+1)}
{2}}}
{(-1)^{\frac{n
(n+1)}
{2}}} s_{\pi_q^{(1)}}(X^t,\X^t).
\end{equation}
Evaluating one of the factors in \eqref{dnm-11}, we have
\begin{equation}
\label{factor}
\begin{split}
\det \left( \begin{array}{c|c}
     B_{0,1} - \bar{B}_{0,0} 
     &  B_{t-1,1}  - \bar{B}_{t-1,0}
     \\[0.2cm]
     \hline \\[-0.3cm]
     A_{0,1}  & -\bar{A}_{t-1,0} \\[0.2cm]
     \hline \\[-0.3cm]
-\bar{A}_{0,0} & A_{t-1,1}
 \end{array} \right) & \\
= (-1)^{\frac{n(n-1)}{2}} & x_1\dots x_n \left(y^{1-tn}  - y^{-tn} \right) V(X^t,\X^t,y^t) \\
= &(-1)^{\frac{n(n-1)}{2}} x_1\dots x_n \left(y^{1+tn}  - y^{tn} \right)  V(X^t,\X^t, \bar{y}^t),
\end{split}
\end{equation}
where $V(X^t,\X^t, \bar{y}^t)=\prod_{1 \leq i < j \leq n} (x_i^t-x_j^t)(x_i^t-\x_j^t) (x_j^t-\x_i^t) (\x_i^t-\x_j^t) \ds \prod_{i=1}^n (x_i^t-y^t) (x_i^t-\x_i^t) (\x_i^t-y^t).$
Using \cref{lem:s-new-1} 
and \eqref{factor}, we see that 
\begin{multline}
% \begin{equation}
\label{schuri0}
% \begin{split}
\frac{\det \left( \begin{array}{c|c}
     B_{i_0,1}^{\lambda} - \bar{B}_{i_0,0}^{\lambda} 
    &  B_{t-1-i_0,1}^{\lambda}  - \bar{B}_{t-1-i_0,0}^{\lambda} 
     \\\hline
     A_{i_0,1}^{\lambda}  & -\bar{A}_{t-1-i_0,0}^{\lambda} \\\hline
-\bar{A}_{i_0,0}^{\lambda}  & A_{t-1-i_0,1}^{\lambda}
 \end{array} \right)}
 {\det \left( \begin{array}{c|c}
     B_{0,1} - \bar{B}_{0,0}
     &  B_{t-1,1}  - \bar{B}_{t-1,0}
     \\\hline
     A_{0,1}  & -\bar{A}_{t-1,0} \\\hline
-\bar{A}_{0,0}  & A_{t-1,1}
 \end{array} \right)}  = \frac{(-1)^{\frac{n_{t-1-i_0}(\lambda)
(n_{t-1-i_0}(\lambda)+1)}
{2}}}
{(-1)^{\frac{n
(n+1)}
{2}}} \frac{1}{(y-1)} \\
\times \Big( y^{-t(\lambda_1^{(t-1-i_0)} 
+n_{(t-1-i_0)}(\lambda)-n)+i_0+1}  s_{\pi_{i_0}^{(1)}}(X^t,\X^t,y^t)  \\
- 
y^{t(\lambda_1^{(t-1-i_0)} 
+n_{t-1-i_0}(\lambda)-n)-i_0} 
s_{\pi_{i_0}^{(1)}}(X^t,\X^t, \bar{y}^t) \Big).
%  \end{split}
%  \end{equation}
\end{multline}
Thus, using \eqref{scuri0}, \eqref{schuri0} and  
 \[
\frac{\det \left( A_{\frac{t-1}{2},1}^{\lambda} - \bar{A}_{\frac{t-1}{2},0}^{\lambda}  \right)}
{\det \left( A_{\frac{t-1}{2},1} - \bar{A}_{\frac{t-1}{2},0}  \right)} = \oo_{\lambda^{\left(\frac{t-1}{2}\right)}}(X^t),
\]
the ratio of \eqref{num-11} and \eqref{dnm-11} is:
\begin{equation*}
    \begin{split}
        \sgn(\sigma_{\lambda}) \frac{(-1)^{\epsilon} }{(y-1)} 
        \Big( y^{-t(\lambda_1^{(t-1-i_0)} 
+n_{(t-1-i_0)}(\lambda)-n)+i_0+1} & s_{\pi_{i_0}^{(1)}}(X^t,\X^t,y^t)  -
y^{t(\lambda_1^{(t-1-i_0)} 
+n_{t-1-i_0}(\lambda)-n)-i_0} \\
s_{\pi_{i_0}^{(1)}}(X^t,\X^t, \bar{y}^t) \Big)  
 \times  &
 \prod_{\substack{i=0\\i \neq i_0}}^{\floor{\frac{t-2}{2}}} 
 s_{\pi^{(1)}_i} (X^t,\X^t) 
        \times \begin{cases}
     \oo_{\lambda^{\left(\frac{t-1}{2}\right)}}(X^t)   & t \text{ is odd},\\
1 & t \text{ is even},
        \end{cases} 
        \end{split}
    \end{equation*}
where $\epsilon$ is defined in \eqref{epsln}. 
Since $\frac{t(t-2)}{2}\frac{n(n+1)}{2}$ is even for even $t$ and $\frac{(t+1)(t-1)}{2}\frac{n(n+1)}{2}$ is even for odd $t$, the parity of $\epsilon$  is same as $\epsilon_1(\lambda)$ defined in \eqref{epsilon},
completing the proof.
% \begin{equation}
%     (-1)^{\epsilon}  \sgn(\sigma_{\lambda}) \begin{cases}
% \oo_{\lambda^{\left(\frac{t-1}{2}\right)}}(X^t)  & t \text{ is odd},\\
% 1 & t \text{ is even},
%  \end{cases}
% \end{equation}
% where 
% \begin{multline*}
% \epsilon_1 = \frac{t(t-1)}{2}\frac{n(n+1)}{2}+\left(\ds \sum_{i=t+1-i_0}^{t} n_{i-1}(\lambda) \right)
% + \ds \sum_{i=\floor{\frac{t+3}{2}}}^{t} t n (n_{i-1}(\lambda)-n) \\
% + \sum_{q=0}^{{\frac{t-3}{2}}} \left( \frac{n_{t-1-q}(\lambda)
% (n_{t-1-q}(\lambda)+1)}
% {2} - \frac{n(n+1)}
% {2} \right).
% \end{multline*}
\end{proof}
\subsection{Symplectic characters} If $\ds \sum_{i=0}^{t-2}n_{i}(\lambda)=(t-1)n$, then consider the $(t-1)n \times (t-1)n$ matrix
\[
\Pi_2= \left(\omega^{pq}A^{\lambda}_{q-1,1}-\bar{\omega}^{pq}
\bar{A}^{\lambda}_{q-1,1}\right)_{1 \leq p,q \leq t-1}.
\]
Substituting $U_j=A_{j-1,1}^{\lambda}$, $V_j=\bar{A}_{j-1,1}^{\lambda}$ for $1 \leq j \leq t-1$ and
\[
\gamma_{i,j} = \begin{cases}
\omega^{\frac{i(j+1)}{2}} & j \text{ odd},\\
\omega^{-\frac{ij}{2}} & j \text{ even}, 
\end{cases}
\]
in \cref{lem:det-blockmatrix} proves the following corollary.

\begin{cor} 
\label{cor:det-Pi_2} 
\begin{enumerate} 

\item If $n_{i}(\lambda)+n_{t-2-i}(\lambda)\neq 2n$ for some $i \in [0,\floor{\frac{t-2}{2}}]$, then $\det \Pi_2 = 0$.

\item If $n_{i}(\lambda)+n_{t-2-i}(\lambda) = 2n$ for all $i \in \{0,1,\dots,\floor{\frac{t-2}{2}}\}$, then 
\begin{equation}
    \begin{split}
        \label{spdet3}
\det \Pi_2
= (-1)^{\Sigma_2} \left(\det \Gamma
% (\gamma_{i,j})_{1 \leq i,j \leq t-1}
\right)^{n}
\prod_{q=0}^{\floor{\frac{t-3}{2}}}
&\det\left(\begin{array}{c|c}
   A^{\lambda}_{i,1}  & -\bar{A}^{\lambda}_{t-2-i,1} \\[0.2cm]
   \hline \\[-0.3cm]
   -\bar{A}^{\lambda}_{i,1}  & A^{\lambda}_{t-2-i,1}
\end{array}\right) \\
\times &
\begin{cases}
\det \left(A^{\lambda}_{\frac{t}{2}-1,1}- \bar{A}^{\lambda}_{\frac{t}{2}-1,1} \right)  
& t \text{ even,}\\
1 & t \text{ odd,}
\end{cases}  
    \end{split}
\end{equation}
where 
\[
\Sigma_2=
% \sum_{q=1}^{\floor{\frac{t-1}{2}}}
% \left(n+n_{q-1}(\lambda)\right)
% +
\begin{cases}
% n \ds \sum_{q=1}^{\frac{t-2}{2}} n_{q-1}(\lambda) =
n \ds \sum_{q=\frac{t+2}{2}}^{t-1} n_{q-1}(\lambda)
&  t \text{ even}, \\ 
0 & t \text{ odd}.
\end{cases}
\]
\end{enumerate}
\end{cor}

\noindent
If $\ds \sum_{i=0}^{t-2}n_{i}(\lambda)=(t-1)n+1$, then consider the $(t-1)n+1 \times (t-1)n+1$ matrix
\begin{equation}
    \label{delta-2}
    \Delta_2 \coloneqq \left( 
    \begin{array}{c}
         (B_{q-1,1}^{\lambda} - \bar{B}_{q-1,1}^{\lambda})_{1 \leq q \leq t-1}  \\\\\hline \\
          \left(\omega^{pq}A^{\lambda}_{q-1,1}
          -\bar{\omega}^{pq}
\bar{A}^{\lambda}_{q-1,1}\right)_{1 \leq p,q \leq t-1}
    \end{array}
    \right). 
\end{equation}
Substituting $M_j=B_{j-1,1}^{\lambda}$, $N_j=\bar{B}_{j-1,1}^{\lambda}$, $U_j=A_{j-1,1}^{\lambda}$, $V_j=\bar{A}_{j-1,1}^{\lambda}$ for all $1 \leq j \leq t-1$ and
\[
\gamma_{i,j} = \begin{cases}
\omega^{\frac{i(j+1)}{2}} & j \text{ odd},\\
\omega^{-\frac{ij}{2}} & j \text{ even}, 
\end{cases}
\]
in \cref{lem:det-blockmatrix} proves the following corollary.
\begin{cor}
\label{cor:delta-2}
\begin{enumerate}
\item  If  
\[ n_{j}(\lambda)+n_{t-2-j}(\lambda)=\begin{cases}
2n+1+\delta_{i_0,\frac{t-2}{2}} & j=i_0,\\
2n & \text{ otherwise}, 
\end{cases} \quad j \in \left[0,\floor{\frac{t-2}{2}}\right],
\]
for some $i_0 \in [0,\floor{\frac{t-2}{2}}]$, then
\begin{equation}
\label{delta2}
    \det \Delta_2
 = (-1)^{\chi_2} 
 (\det \Gamma)^n \det \left( \begin{array}{c}
      O_{i_0}^{(2)}  \\[0.1cm]
     \hline \\[-0.4cm]
      W_{i_0}^{(2)}
 \end{array} \right)
 \prod_{\substack{i=0\\i \neq i_0}}^{\floor{\frac{t-2}{2}}}
 \det W_i^{(2)},
\end{equation}
where
\[
O_i^{(2)} = \begin{cases}
\left(\begin{array}{c|c}
   B_{i,1}^{\lambda} - \bar{B}_{i,1}^{\lambda} & B_{t-2-i,1}^{\lambda} - \bar{B}_{t-2-i,1}^{\lambda} 
\end{array} \right) & \text{ if } 0 \leq i \leq \floor{\frac{t-3}{2}}\\
\left( B_{\frac{t-2}{2},1}^{\lambda} - \bar{B}_{\frac{t-2}{2},1}^{\lambda} \right) & \text{$t$ even and $i = \frac{t-2}{2}$},
\end{cases}
\]
\[
W_i^{(2)}= \begin{cases}
\left(\begin{array}{c|c}
A_{i,1}^{\lambda}  & 
- \bar{A}_{t-2-i,1}^{\lambda} \\[0.1cm]
     \hline \\[-0.4cm]
- \bar{A}_{i,1}^{\lambda}  
& A_{t-2-i,1}^{\lambda}
\end{array}\right) & 0 \leq i \leq \floor{\frac{t-3}{2}}, \\
\left( A_{\frac{t-2}{2},1}^{\lambda} - \bar{A}_{\frac{t-2}{2},1}^{\lambda}  \right) 
& \text{$t$ even and $i = \frac{t-2}{2}$},
\end{cases}
\]
and
\[
\chi_2=
\left(\ds \sum_{i=t-i_0}^{t-1} n_{i-1}(\lambda) \right)
+ \ds \sum_{i=\floor{\frac{t+2}{2}}}^{t-1} (t-1) n n_{i-1}(\lambda).
\]
% $W_i=\left(\begin{array}{cc}
%   {A}_i  & \bar{A}_{t-i} \\
%   \bar{A}_i  & {A}_{t-i}
% \end{array} \right)$, $W_{\frac{t}{2}}= \left( A_{\frac{t}{2}}+\bar{A}_{\frac{t}{2}}\right)$, $V_i=\left(\begin{array}{cc}
%   {B}_i+{\bar{B}}_i  & {B}_{t-i}+\bar{B}_{t-i} 
% \end{array} \right)$, $V_{\frac{t}{2}}=B_{\frac{t}{2}}+\bar{B}_{\frac{t}{2}}$. 
\item Otherwise 
\[
\det \Delta_2=0.
\]
\end{enumerate}
\end{cor}
\begin{proof}[Proof of \cref{thm:symp}]
By \eqref{spdef}, we see that the numerator of the required symplectic character is given by
\begin{equation}
    % \sp_{\lambda}(X,\, \omega X,  \dots,\omega^{t-1}X,y) \\ 
    % =  \frac{ 
    \det \left(
    \begin{array}{c}
    \left( \left( 
    (\omega^{p-1}x_i)^{\beta_{j}(\lambda)+1}-(\bar{\omega}^{p-1}\x_i)^{\beta_{j}(\lambda)+1}
    \right)_{\substack{1 \leq i\leq n \\ 1\leq j\leq tn+1}}
    \right)_{1 \leq p \leq t} \\[0.3cm]
     \hline \\[-0.3cm]
    \left(  
    y^{\beta_{j}(\lambda)+1}-\bar{y}^{\beta_{j}(\lambda)+1}
    \right)_{1\leq j\leq tn+1}
 \end{array} 
    \right).
    % }
%     {\det \left(
%     \begin{array}{c}
%     \left( \left( 
%     (\omega^{p-1}x_i)^{tn+2-j}-(\bar{\omega}^{p-1}\x_i)^{tn+2-j}
%     \right)_{\substack{1 \leq i\leq n \\ 1\leq j\leq tn+1}}
%     \right)_{1 \leq p \leq t} \\\hline\\[-0.3cm]
%     \left(  
%     y^{tn+2-j}-\bar{y}^{tn+2-j}
%     \right)_{1\leq j\leq tn+1}
%  \end{array} 
%     \right)}.
\end{equation}
% Since the denominator of the right hand side of \eqref{symp-1} is the same as its numerator evaluated at the empty partition, we compute the factorization for the numerator and use that to get factorization for the denominator.
Permuting the columns of the determinant in the numerator by $\sigma_{\lambda}$ from \eqref{sigma-perm-m} $(m=1,d_1=0)$ and applying blockwise row operations $R_1 \to R_1+\dots+R_t$, $R_i \to R_i-\frac{1}{t} R_1$, $2 \leq i \leq t$ and then permute the last $t$ rows cyclically, we see that the numerator is 
\begin{equation}
\label{num-22}
    \sgn(\sigma_{\lambda}) t^n (-1)^{t-1} \det \left(
    \begin{array}{c|c|c|c}
   0& \dots & 0 &  A_{t-1,1}^{\lambda}-\bar{A}^{\lambda}_{t-1,1}
   \\\hline&&&\\
    B_{0,1}^{\lambda}-\bar{B}^{\lambda}_{0,1}
    & \dots &
    B_{t-2,1}^{\lambda}-\bar{B}^{\lambda}_{t-2,1} &  B_{t-1,1}^{\lambda}-\bar{B}^{\lambda}_{t-1,1}
   \\\hline&&&\\
   \omega A_{0,1}^{\lambda}-\omega^{t-1}\bar{A}^{\lambda}_{0,1} 
   & \dots & \omega^{t-1} A_{t-2,1}^{\lambda}-\omega  \bar{A}^{\lambda}_{t-2,1}
    &  0 \\\hline&&&\\
    \vdots &  \dots &  \vdots &  \vdots \\\hline&&&\\
    \omega^{t-1} A_{0,1}^{\lambda}-\omega \bar{A}^{\lambda}_{0,1}  & \dots &  
    \omega A_{t-2,1}^{\lambda} - \omega^{t-1} \bar{A}^{\lambda}_{t-2,1} &  0 
    \end{array} \right),
\end{equation}
where $A_{p,q}^{\lambda}$,  $\bar{A}_{p,q}^{\lambda}$, 
$B_{p,q}^{\lambda}$ and $\bar{B}_{p,q}^{\lambda}$ are defined in \eqref{def AB} and \eqref{def 1AB}.
If $\core \lambda t \not \in \mathcal{Q}^{(t)}_{3,0,k} \cup \mathcal{Q}^{(t)}_{3,2,k}$ for all $k \in [\rk(\lambda)]$, then using \cref{cor:z=1},  \cref{cor:det-Pi_2}  and \cref{cor:delta-2},
\[
    \sp_{\lambda}(X,\, \omega X, \dots,\omega^{t-1}X,y)=0.
    \]
If $\core \lambda t \in \mathcal{Q}^{(t)}_{3,0,k} \cup \mathcal{Q}^{(t)}_{3,2,k}$ for some $k \in [\rk(\lambda)]$, then using \cref{cor:z=1}, we factorize the numerator using \cref{cor:det-Pi_2}  and \cref{cor:delta-2}.
Since the denominator in \eqref{spdef} is its numerator evaluated for the empty partition, and the empty partition is vacuously $(3,0,0)$-asymmetric with $n_0(\emptyset, tn+1)=n+1$ and $n_i(\emptyset, tn+1)=n$ for all $i \in [1,t-1]$, the factorization for the denominator of required symplectic character is 
\begin{equation}
\begin{split}
\label{dnm-21}
    (-1)^{\chi_2^{(0)}} (\det \Gamma)^n  \sgn(\sigma_{\emptyset}) t^n 
(-1)^{(t-1)(n^2+1)+n} \det &\left( A_{t-1,1}-\bar{A}_{t-1,1} \right) \\
\times \det \left( \begin{array}{c}
     \begin{array}{c|c}
   B_{0,1} - \bar{B}_{0,1}  
   & B_{t-2,1} - \bar{B}_{t-2,1} \\[0.2cm]
     \hline \\[-0.3cm]
A_{0,1}  & -\bar{A}_{t-2,1} \\[0.2cm]
     \hline \\[-0.3cm]
-\bar{A}_{0,1}  & A_{t-2,1}
\end{array} 
 \end{array} \right) &\times 
 \prod_{i=1}^{\floor{\frac{t-3}{2}}}
 \det \left(\begin{array}{c|c}
A_{i,1}  & -\bar{A}_{t-2-i,1} \\[0.2cm]
     \hline \\[-0.3cm]
-\bar{A}_{i,1}  & A_{t-2-i,1}
\end{array}\right) \\
&\times \begin{cases}
\left( A_{\frac{t-2}{2},1} - \bar{A}_{\frac{t-2}{2},1} \right) & t \text{ even,}\\
1 & t \text{ odd,}
 \end{cases}
 \end{split}
\end{equation}
where 
\[
\chi_2^{(0)}= \ds \sum_{i=\floor{\frac{t+2}{2}}}^{t-1} (t-1) n^2.
\]
\underline{Case 1}. $i_0=t-1$. In this case $n_{t-1}(\lambda)=n+1$ and the matrix in \eqref{num-22} is block anti-diagonal $2 \times 2$ matrix. Using \eqref{spdet3}, the numerator in this case is
\begin{equation}
\label{num-221}
    \begin{split}
    \sgn(\sigma_{\lambda}) t^n 
(-1)^{(t-1)(n^2+n+1)}
(-1)^{\Sigma_2} \left( \det \Gamma \right)^{n}
 & \det \left( \begin{array}{c}
     A_{t-1,1}^{\lambda}-\bar{A}^{\lambda}_{t-1,1} 
     \\[0.2cm]
     \hline \\[-0.3cm]
     B_{t-1,1}^{\lambda}-\bar{B}^{\lambda}_{t-1,1} 
\end{array} \right)
\\  
\times \prod_{i=0}^{\floor{\frac{t-3}{2}}} \det\left(\begin{array}{c|c}
   A^{\lambda}_{i,1}  & -\bar{A}^{\lambda}_{t-2-i,1} \\[0.2cm]
   \hline \\[-0.3cm]
   -\bar{A}^{\lambda}_{i,1}  & A^{\lambda}_{t-2-i,1}
\end{array}\right) 
\times &
\begin{cases}
\det \left(A^{\lambda}_{\frac{t-2}{2},1}- \bar{A}^{\lambda}_{\frac{t-2}{2},1} \right)  
& t \text{ even,}\\
1 & t \text{ odd}.
\end{cases}  
    \end{split}
\end{equation}
By \cref{lem:DI}, we have 
\begin{multline*}
 \det \left( \begin{array}{c}
     \begin{array}{c|c}
   B_{0,1} - \bar{B}_{0,1}  
   & B_{t-2,1} - \bar{B}_{t-2,1} \\[0.2cm]
     \hline \\[-0.3cm]
A_{0,1}  & -\bar{A}_{t-2,1} \\[0.2cm]
     \hline \\[-0.3cm]
-\bar{A}_{0,1}  & A_{t-2,1}
\end{array} 
 \end{array} \right) \\
 = 
 y^{-tn}(y-\bar{y}) \prod_{i=1}^n \left( (x_i^t-y^t)(\x_i^t-y^t)  \right)
\det \left( \begin{array}{c|c}
     A_{0,1}  & -\bar{A}_{t-2,1} \\[0.2cm]
     \hline \\[-0.3cm]
-\bar{A}_{0,1} & A_{t-2,1}
 \end{array} \right)   
\end{multline*}
and
\[
 \det \left( \begin{array}{c}
A_{0,t}  - \bar{A}_{0,t} \\[0.2cm]
     \hline \\[-0.3cm]
B_{0,t}  - \bar{B}_{0,t}
\end{array} 
\right) 
= (-1)^n 
y^{-tn-t}  (y^{2t}-1)
\prod_{i=1}^n \left( (x_i^t-y^t)(\x_i^t-y^t)  \right) \det \left( A_{t-1,1}-\bar{A}_{t-1,1} \right).
\]
Substituting in \eqref{dnm-21}, the denominator in this case is
\begin{equation}
\label{dnm-22}
    \begin{split}
    (-1)^{\chi_2^{(0)}+n} (\det \Gamma)^n  \sgn(\sigma_{\emptyset}) t^n 
(-1)^{(t-1)(n^2+1)+n} & 
\frac{(y^{2t}-1)}{y^t(y-\bar{y})}
  \det \left( \begin{array}{c}
     A_{0,t}-\bar{A}_{0} 
     \\[0.2cm]
     \hline \\[-0.3cm]
     B_{0,t}-\bar{B}_{0} 
\end{array} \right)
\\  
\times \prod_{i=0}^{\floor{\frac{t-3}{2}}} \det\left(\begin{array}{c|c}
   A_{i,1}  & -\bar{A}_{t-2-i,1} \\[0.2cm]
   \hline \\[-0.3cm]
   -\bar{A}_{i,1}  & A_{t-2-i,1}
\end{array}\right) 
\times &
\begin{cases}
\det \left(A_{\frac{t-2}{2},1}- \bar{A}_{\frac{t-2}{2},1} \right)  
& t \text{ even,}\\
1 & t \text{ odd}.
\end{cases}  
    \end{split}
\end{equation}
So, using \cref{lem:s-new}, \eqref{E3} and \eqref{F2}, we see that the required symplectic character, the ratio of \eqref{num-221} and \eqref{dnm-22}, is
\[
(-1)^{\epsilon_2} \sgn(\sigma_{\lambda}) \frac{(y^{2t}-1)}{y^t(y-\bar{y})} 
\sp_{\lambda^{(t-1)}}(X^t, y^t) \prod_{q=0}^{\floor{\frac{t-3}{2}}} 
s_{\pi_q^{(2)}} (X^t, \X^t) 
\times 
\begin{cases}
\oo_{\lambda^{\left(\frac{t-2}{2}\right)}} (X^t) 
& t \text{ even,}\\
1 & t \text{ odd,}
\end{cases} 
\]
where 
\begin{equation}
\label{epsln2}
\begin{split}
\epsilon_2 = \frac{t(t-1)}{2}\frac{n(n+1)}{2} + (t-1)n
+ \sum_{q=0}^{\floor{\frac{t-3}{2}}} & \left( \frac{n_{t-2-q}(\lambda)
(n_{t-2-q}(\lambda)+1)}
{2} - \frac{n(n+1)}
{2} \right) \\
- 
\ds \sum_{i=\floor{\frac{t+2}{2}}}^{t-1} (t-1) n^2 + & \begin{cases}
n \ds \sum_{q=\frac{t+2}{2}}^{t-1} n_{q-1}(\lambda) &  t \text{ even} \\ 
0 & t \text{ odd}.
\end{cases}
\end{split}
\end{equation}
Since $\frac{(t+1)(t-1)}{2}\frac{n(n+1)}{2}$ is even for odd $t$ and the parity of $\frac{(t^2-2t+2)}{2}\frac{n(n+1)}{2}$ is the same as $\frac{n(n+1)}{2}$ for even $t$, $(-1)^{\epsilon_2}$ is the same as $(-1)^{\epsilon_2(\lambda)}$, defined in \eqref{epsilon22}.

If $i_0 \neq t-1$, then $n_{t-1}=n$ and the numerator in \eqref{num-22} is
\[ \sgn(\sigma_{\lambda}) t^n 
(-1)^{(t-1)(n^2+1)+n} \det \left( A^{\lambda}_{t-1,1}-\bar{A}^{\lambda}_{t-1,1} \right) \det \left( 
    \begin{array}{c}
         (B_{q-1,1}^{\lambda} - \bar{B}_{q-1,1}^{\lambda})_{1 \leq q \leq t-1}  \\\\\hline \\
          \left(\omega^{pq}A^{\lambda}_{q-1,1}
          -\bar{\omega}^{pq}
\bar{A}^{\lambda}_{q-1,1}\right)_{1 \leq p,q \leq t-1}
    \end{array}
    \right).
\]
The last determinant is $\Delta_2$ defined in \eqref{delta-2}.
% \[ (-1)^{\chi_2} (\det \Gamma)^n  \sgn(\sigma_{\lambda}) t^n 
% (-1)^{(t-1)(n^2+1)+n} \det \left( A^{\lambda}_{t-1,1}-\bar{A}^{\lambda}_{t-1,1} \right) 
% \det \left( \begin{array}{c}
%       O_{i_0}  \\
%       W_{i_0}
%  \end{array} \right)
%  \prod_{\substack{i=1\\i \neq i_0}}^{\floor{\frac{t}{2}}}
%  \det W_i,
% \]
% where
% \[
% O_i= \begin{cases}
% \left(\begin{array}{c|c}
%   B_{i-1,1}^{\lambda} - \bar{B}_{i-1,1}^{\lambda}  
%   & B_{t-i-1,1}^{\lambda} - \bar{B}_{t-i-1,1}^{\lambda} 
% \end{array} \right) & \text{ if } 1 \leq i \leq \floor{\frac{t-1}{2}}\\
% \left( B_{\frac{t}{2}-1,1}^{\lambda} - \bar{B}_{\frac{t}{2}-1,1}^{\lambda} \right) & \text{$t$ even and $i = \frac{t}{2}$},
% \end{cases}
% \]
% \[
% W_i= \begin{cases}
% \left(\begin{array}{c|c}
% A_{i-1,1}^{\lambda}  & \bar{A}_{t-i-1,1}^{\lambda} \\\hline
% \bar{A}_{i-1,1}^{\lambda}  & A_{t-i-1,1}^{\lambda}
% \end{array}\right) & 1 \leq i \leq \floor{\frac{t-1}{2}}, \\
% \left( A_{\frac{t}{2}-1,1}^{\lambda} - \bar{A}_{\frac{t}{2}-1,1}^{\lambda}  \right) 
% & \text{$t$ even and $i = \frac{t}{2}$},
% \end{cases}
% \]

\underline{Case 2.} $i_0=\frac{t-2}{2}$, then using \eqref{delta2}, the numerator in this case is
\begin{equation}
\begin{split}
\label{num-23}
    (-1)^{\chi_2} (\det \Gamma)^n  \sgn(\sigma_{\lambda}) &t^n 
(-1)^{(t-1)(n^2+1)+n} \det \left( A^{\lambda}_{t-1,1}-\bar{A}^{\lambda}_{t-1,1} \right) \\ \times &
\det \left( \begin{array}{c}
       B_{\frac{t}{2}-1,1}^{\lambda} - \bar{B}_{\frac{t}{2}-1,1}^{\lambda}   \\[0.2cm]
     \hline \\[-0.3cm]
       A_{\frac{t}{2}-1,1}^{\lambda} - \bar{A}_{\frac{t}{2}-1,1}^{\lambda}  
 \end{array} \right) \times 
 \prod_{i=0}^{\frac{t-4}{2}} 
 \det \left(\begin{array}{c|c}
A_{i,1}^{\lambda}  & -\bar{A}_{t-2-i,1}^{\lambda} \\[0.2cm]
     \hline \\[-0.3cm]
-\bar{A}_{i,1}^{\lambda}  & A_{t-2-i,1}^{\lambda}
\end{array}\right).
\end{split}
\end{equation}
By \cref{lem:DI}, we have 
\begin{multline*}
 \det \left( \begin{array}{c}
     \begin{array}{c|c}
   B_{0,1} - \bar{B}_{0,1}  
   & B_{t-2,1} - \bar{B}_{t-2,1} \\[0.2cm]
     \hline \\[-0.3cm]
A_{0,1}  & -\bar{A}_{t-2,1} \\[0.2cm]
     \hline \\[-0.3cm]
-\bar{A}_{0,1}  & A_{t-2,1}
\end{array} 
 \end{array} \right) \\
 = 
 y^{-tn}(y-\bar{y}) \prod_{i=1}^n \left( (x_i^t-y^t)(\x_i^t-y^t)  \right)
\det \left( \begin{array}{c|c}
     A_{0,1}  & -\bar{A}_{t-2,1} \\[0.2cm]
     \hline \\[-0.3cm]
-\bar{A}_{0,1} & A_{t-2,1}
 \end{array} \right)   
\end{multline*}
and
\[
\det \left(
\begin{array}{c}
     A_{0,\frac{t}{2}}-\bar{A}_{0,\frac{t}{2}}  
     \\[0.2cm]
     \hline \\[-0.3cm]
     B_{0,\frac{t}{2}}-\bar{B}_{0,\frac{t}{2}} 
\end{array}
 \right) = (-1)^n 
y^{-tn-t/2}  (y^t-1)
\prod_{i=1}^{n} (x_i^t-y^t)(\x_i^t-y^t) \det   \left( A_{\frac{t-2}{2},1} - \bar{A}_{\frac{t-2}{2},1} \right).  
\]
Substituting in \eqref{dnm-21}, the denominator in this case is
\begin{equation}
\label{dnm-23}
    \begin{split}
    (-1)^{\chi_2^{(0)}+n} (\det \Gamma)^n  \sgn(\sigma_{\emptyset}) t^n 
(-1)^{(t-1)(n^2+1)+n} & \frac{(y^{t}-1)}{y^{t/2}(y-\bar{y})}
  \det \left( 
     A_{t-1,1}^{\lambda}-\bar{A}^{\lambda}_{t-1,1} 
 \right)
\\  
\times \prod_{i=0}^{\floor{\frac{t-3}{2}}} \det\left(\begin{array}{c|c}
   A_{i,1}  & -\bar{A}_{t-2-i,1} \\[0.2cm]
   \hline \\[-0.3cm]
   -\bar{A}_{i,1}  & A_{t-2-i,1}
\end{array}\right) 
\times &
\begin{cases}
\det \left(
\begin{array}{c}
     A_{0,\frac{t}{2}}-\bar{A}_{0,\frac{t}{2}}  
     \\[0.2cm]
     \hline \\[-0.3cm]
     B_{0,\frac{t}{2}}-\bar{B}_{0,\frac{t}{2}} 
\end{array}
 \right)  
& t \text{ even,}\\
1 & t \text{ odd}.
\end{cases}  
    \end{split}
\end{equation}
So, using \cref{lem:s-new}, \eqref{E2} and \eqref{oo-new-1}, we see that the required symplectic character, the ratio of \eqref{num-23} and \eqref{dnm-23} is
\[
(-1)^{\epsilon_2'+n} \sgn(\sigma_{\lambda}) \frac{(y^t-1)}{y^{t/2}(y-\bar{y})} \sp_{\lambda^{(t-1)}}(X^t) \prod_{q=0}^{\floor{\frac{t-3}{2}}} s_{\pi_q^{(2)}} (X^t, \X^t) 
\times 
\oo_{\lambda^{\left(\frac{t-2}{2}\right)}} (X^t,y^t),
\]
where 
\begin{equation}
\label{epsilon'}
\begin{split}
\epsilon_2'=\frac{t(t-1)}{2}\frac{n(n+1)}{2}+ \ds \sum_{i=t-i_0}^{t-1} & n_{i-1}(\lambda) 
+ \ds \sum_{i=\floor{\frac{t+2}{2}}}^{t-1} (t-1) n (n_{i-1}(\lambda)-n) \\
+ \sum_{q=0}^{\floor{\frac{t-3}{2}}}  &
\left( \frac{n_{t-2-q}(\lambda)
(n_{t-2-q}(\lambda)+1)}
{2} - \frac{n(n+1)}
{2} \right). 
\end{split}
\end{equation}
Since the parity of $\frac{(t^2-2t+2)}{2}\frac{n(n+1)}{2}$ is the same as $\frac{n(n+1)}{2}$ for even $t$, $(-1)^{\epsilon'_2}$ is the same as $(-1)^{\epsilon_2(\lambda)+n+1}$, defined in \eqref{epsilon22}. 

\underline{Case 3}. $i_0 \neq \frac{t-2}{2}$. Using \eqref{delta2}, the numerator in this case is
\begin{multline}
\label{num-25}
(-1)^{\chi_2} (\det \Gamma)^n  \sgn(\sigma_{\lambda}) t^n 
(-1)^{(t-1)(n^2+1)+n} \det \left( A^{\lambda}_{t-1,1}-\bar{A}^{\lambda}_{t-1,1} \right) \\
\times \det \left( \begin{array}{c|c}
   B_{i_0,1}^{\lambda} - \bar{B}_{i_0,1}^{\lambda}  
   & B_{t-2-i_0,1}^{\lambda} - \bar{B}_{t-2-i_0,1}^{\lambda} 
\\[0.2cm]
     \hline \\[-0.3cm]
A_{i_0,1}^{\lambda}  & -\bar{A}_{t-2-i_0,1}^{\lambda} \\[0.2cm]
     \hline \\[-0.3cm]
-\bar{A}_{i_0,1}^{\lambda}  & A_{t-2-i_0,1}^{\lambda}
 \end{array} \right) 
 \times 
 \prod_{\substack{i=0\\i \neq i_0}}^{\floor{\frac{t-3}{2}}}
 \det \left(\begin{array}{c|c}
A_{i,1}^{\lambda}  & -\bar{A}_{t-2-i,1}^{\lambda} \\[0.2cm]
\hline \\[-0.3cm]
-\bar{A}_{i,1}^{\lambda}  & A_{t-2-i,1}^{\lambda}
\end{array}\right)\\
\times \begin{cases}
\left( A_{\frac{t-2}{2},1}^{\lambda} - \bar{A}_{\frac{t-2}{2},1}^{\lambda}  \right) & t \text{ even,}\\
1 & t \text{ odd}.
 \end{cases}   
\end{multline}
By \cref{lem:DI}, we have  
\begin{multline}
\label{factor2}
\det \left( \begin{array}{c|c}
     B_{0,1} - \bar{B}_{0,1} 
     &  B_{t-2,1}  - \bar{B}_{t-2,1}
     \\[0.2cm]
     \hline \\[-0.3cm]
     A_{0,1}  & -\bar{A}_{t-2,1} \\[0.2cm]
     \hline \\[-0.3cm]
-\bar{A}_{0,1} & A_{t-2,1}
 \end{array} \right) \\
= (-1)^{\frac{n(n-1)}{2}} x_1\dots x_n \left(y^{1-tn}  - y^{-tn-1} \right) V(X^t,\X^t,y^t) \\
= (-1)^{\frac{n(n-1)}{2}} x_1\dots x_n \left(y^{1+tn}  - y^{tn-1} \right) V(X^t,\X^t, \bar{y}^t),
% = (-1)^{\frac{n(n-1)}{2}}\prod_i  x_i V(X^t,\X^t) \left(y^{1-tn} \prod_i (x_i^t-y^t)(\x_i^t-y^t) + y^{tn} \prod_i (x_i^t-\bar{y}^t)(\x_i^t-\bar{y}^t) \right)\\
% = \bar{y}^{tn}  (y  + 1) 
%  \prod_{i=1}^n \left( (x_i^t-y^t)(\x_i^t-y^t)  \right)
% \det \left( \begin{array}{c|c}
%      A_{0,1}  & -\bar{A}_{t-1,0} \\\hline
% -\bar{A}_{0,0} & A_{t-1,1}
%  \end{array} \right)  
\end{multline}
where $V(X^t,\X^t, \bar{y}^t)=\prod_{1 \leq i < j \leq n} (x_i^t-x_j^t)(x_i^t-\x_j^t) (x_j^t-\x_i^t) (\x_i^t-\x_j^t) \ds \prod_{i=1}^n (x_i^t-y^t) (x_i^t-\x_i^t) (\x_i^t-y^t)$.
Using \cref{lem:s-new-1} 
and \eqref{factor2}, we see that 
\begin{multline}
% \begin{equation}
\label{schuri02}
% \begin{split}
\frac{\det \left( \begin{array}{c|c}
     B_{i_0,1}^{\lambda} - \bar{B}_{i_0,1}^{\lambda} 
    &  B_{t-2-i_0,1}^{\lambda}  - \bar{B}_{t-2-i_0,1}^{\lambda} 
     \\\hline
     A_{i_0,1}^{\lambda}  & -\bar{A}_{t-2-i_0,1}^{\lambda} \\\hline
-\bar{A}_{i_0,1}^{\lambda}  & A_{t-2-i_0,1}^{\lambda}
 \end{array} \right)}
 {\det \left( \begin{array}{c|c}
     B_{0,1} - \bar{B}_{0,1}
     &  B_{t-2,1}  - \bar{B}_{t-2,1}
     \\\hline
     A_{0,1}  & -\bar{A}_{t-2,1} \\\hline
-\bar{A}_{0,1}  & A_{t-2,1}
 \end{array} \right)}  = \frac{(-1)^{\frac{n_{t-2-i_0}(\lambda)
(n_{t-2-i_0}(\lambda)+1)}
{2}}}
{(-1)^{\frac{n
(n+1)}
{2}}} \frac{1}{y-\bar{y}}\\
\times \Big( y^{-t(\lambda_1^{(t-2-i_0)} 
+n_{(t-2-i_0)}(\lambda)-n)+i_0}  s_{\pi_{i_0}^{(1)}}(X^t,\X^t,y^t)  \\
-
y^{t(\lambda_1^{(t-1-i_0)} 
+n_{t-1-i_0}(\lambda)-n)-i_0} 
s_{\pi_{i_0}^{(1)}}(X^t,\X^t, \bar{y}^t) \Big).
%  \end{split}
%  \end{equation}
\end{multline}
Thus, using \cref{lem:s-new}, \eqref{schuri02} and  
 \[
\frac{\det \left( A_{\frac{t-2}{2},1}^{\lambda} - \bar{A}_{\frac{t-2}{2},1}^{\lambda}  \right)}
{\det \left( A_{\frac{t-2}{2},1} - \bar{A}_{\frac{t-2}{2},1}  \right)} = \oo_{\lambda^{\left(\frac{t-2}{2}\right)}}(X^t),
\]
the ratio of \eqref{num-25} and \eqref{dnm-21} is:
\begin{equation*}
    \begin{split}
        \sgn(\sigma_{\lambda}) \frac{(-1)^{\epsilon_2'}}{y-\bar{y}} 
        \Big( y^{-t(\lambda_1^{(t-2-i_0)} 
+n_{(t-2-i_0)}(\lambda)-n)+i_0} & s_{\pi_{i_0}^{(2)}}(X^t,\X^t,y^t)  -
y^{t(\lambda_1^{(t-2-i_0)} 
+n_{t-2-i_0}(\lambda)-n)-i_0} \\
s_{\pi_{i_0}^{(2)}}(X^t,\X^t, \bar{y}^t) \Big)  
 \times  &
 \prod_{\substack{i=0\\i \neq i_0}}^{\floor{\frac{t-2}{2}}} 
 s_{\pi^{(2)}_i} (X^t,\X^t) 
        \times \begin{cases}
     \oo_{\lambda^{\left(\frac{t-1}{2}\right)}}(X^t)   & t \text{ is odd},\\
1 & t \text{ is even},
        \end{cases} 
        \end{split}
    \end{equation*}
where $\epsilon_2'$ is defined in \eqref{epsilon'}. Since $\frac{(t+1)(t-1)}{2}\frac{n(n+1)}{2}$ is even for odd $t$, $(-1)^{\epsilon_2'}$ is the same as $(-1)^{\epsilon_2(\lambda)}$, defined in \eqref{epsilon22}.  This completes the proof.
% \[
% \epsilon_2=\frac{t(t-1)}{2}\frac{n(n+1)}{2}+ \left(\ds \sum_{i=t-i_0}^{t-1} n_{i-1}(\lambda) \right)
% + \ds \sum_{i=\floor{\frac{t+2}{2}}}^{t-1} (t-1) n (n_{i-1}(\lambda)-n).
% \]
% Since $\frac{t(t-2)}{2}\frac{n(n+1)}{2}$ is even for even $t$ and $\frac{(t+1)(t-1)}{2}\frac{n(n+1)}{2}$ is even for odd $t$, the parity of $\epsilon$  is same as $\epsilon_2(\lambda)$ defined in \eqref{epsilon22}, completing the proof.
\end{proof}
\subsection{even orthogonal characters} 
If $\ds \sum_{i=1}^{t-1}n_{i}(\lambda) \allowbreak = (t-1)n$, 
then 
consider the $(t-1)n \times (t-1)n$ block matrix
\begin{equation}
    \label{p11}
    \Pi_3 \coloneqq \left(\omega^{pq}A^{\lambda}_{q}+\bar{\omega}^{pq}
\bar{A}^{\lambda}_{q}\right)_{1 \leq p,q \leq t-1}.
\end{equation}
Substituting $U_j=A_j^{\lambda}$, $V_j=-\bar{A}_j^{\lambda}$ and for $1 \leq j \leq t-1$,
\[
\gamma_{i,j} = \begin{cases}
\omega^{\frac{i(j+1)}{2}} & j \text{ odd},\\
\omega^{-\frac{ij}{2}} & j \text{ even}, 
\end{cases}
\]
in \cref{lem:det-blockmatrix}, we get the following corollary.
\begin{cor}
% [{\cite{ayyer-2021}}] 
\label{cor:det-even} \begin{enumerate}
\item If $n_{i}(\lambda)+n_{t-i}(\lambda)\neq 2n$ 
for some $i \in [\floor{\frac{t}{2}}]$, then
$\det \Pi_3 = 0$.

\item If $n_{i}(\lambda)+n_{t-i}(\lambda) = 2n$ for all $i \in [\floor{\frac{t}{2}}]$, then 
\begin{equation}
\begin{split}
   \label{eodet3}
 \det \Pi_3
= (-1)^{\Sigma_3}  \left(\det (\gamma_{i,j})_{1 \leq i,j \leq t-1} \right)^{n} &
\prod_{q=1}^{\floor{\frac{t-1}{2}}}
\det\left(\begin{array}{c|c}
   A^{\lambda}_{q}  & \bar{A}^{\lambda}_{t-q} \\[0.2cm]
   \hline \\[-0.3cm]
   \bar{A}^{\lambda}_{q}  & A^{\lambda}_{t-q}
   \end{array}\right) \\ &
\times
\begin{cases}
\det \left(A^{\lambda}_{\frac{t}{2}}+ \bar{A}^{\lambda}_{\frac{t}{2}} \right)  & t \text{ even},\\
1 & t \text{ odd},
\end{cases}    
\end{split}
\end{equation}
\end{enumerate}
where 
\[
\Sigma_3 =
\begin{cases}
% n \ds \sum_{q=1}^{\frac{t-2}{2}} n_{q}(\lambda) =
n \ds \sum_{q=\frac{t+2}{2}}^{t-1} n_{q}(\lambda)
&  t \text{ even},\\
0 &  t \text{ odd}.
\end{cases}
\]
\end{cor}
\noindent
If $\ds \sum_{i=1}^{t-1}n_{i}(\lambda) \allowbreak = (t-1)n+1$, 
then 
consider the $((t-1)n+1) \times ((t-1)n+1)$ block matrix
\begin{equation}
\label{delta-1}
 \Delta_3 \coloneqq \left( 
    \begin{array}{c}
         (B_q^{\lambda}+ \bar{B}_q^{\lambda})_{1 \leq q \leq t-1}  \\\\\hline \\
          \left(
 \omega^{pq} A_q^{\lambda}+ \omega^{t-pq} \bar{A}_q^{\lambda}
    \right)_{1 \leq p,q \leq t-1}
    \end{array}
    \right).   
\end{equation}
Substituting $M_j= B_{j}^{\lambda}$, $N_j= \bar{B}_{j}^{\lambda}$,  $U_j=A_j^{\lambda}$, $V_j=-\bar{A}_j^{\lambda}$ and for $1 \leq j \leq t-1$,
\[
\gamma_{i,j} = \begin{cases}
\omega^{\frac{i(j+1)}{2}} & j \text{ odd},\\
\omega^{-\frac{ij}{2}} & j \text{ even}, 
\end{cases}
\]
in \cref{lem:det-blockmatrix-1}, we get the following corollary.
\begin{cor}
\label{det-even-1}
\begin{enumerate} \item If
\begin{equation}
\label{i0}
n_j(\lambda,tn+1)+n_{t-j}(\lambda,tn+1)= 
\begin{cases}
2n+1+\delta_{i_0,\frac{t}{2}} & j=i_0,\\
2n & \text{ otherwise},
\end{cases}  \quad j \in [t-1],
\end{equation}
for some $i_0 \in [\floor{\frac{t}{2}}]$, then
\begin{equation}
    \det \Delta_3
 = (-1)^{\chi_3} 
 (\det \Gamma)^n \det \left( \begin{array}{c}
      O_{i_0}^{(3)}  \\[0.1cm]
     \hline \\[-0.4cm]
      W_{i_0}^{(3)}
 \end{array} \right)
 \prod_{\substack{i=1\\i \neq i_0}}^{\floor{\frac{t}{2}}}
 \det W_i^{(3)},
\end{equation}
where
\[
O_i^{(3)}= \begin{cases}
\left(\begin{array}{c|c}
   B_i^{\lambda} + \bar{B}_i^{\lambda}  & B_{t-i}^{\lambda} + \bar{B}_{t-i}^{\lambda} 
\end{array} \right) & \text{ if } 1 \leq i \leq \floor{\frac{t-1}{2}},\\
\left(
B_{\frac{t}{2}}^{\lambda} + \bar{B}_{\frac{t}{2}}^{\lambda}
\right) & \text{$t$ even and $i = \frac{t}{2}$},
\end{cases}
\]
\[
W_i^{(3)}= \begin{cases}
\left(\begin{array}{c|c}
A_i^{\lambda}  & \bar{A}_{t-i}^{\lambda} \\[0.1cm]
     \hline \\[-0.4cm]
\bar{A}_i^{\lambda}  & A_{t-i}^{\lambda}
\end{array}\right)
& 1 \leq i \leq \floor{\frac{t-1}{2}}, \\
\left( A_{\frac{t}{2}}^{\lambda} +\bar{A}_{\frac{t}{2}}^{\lambda}  \right) 
& \text{$t$ even and $i = \frac{t}{2}$},
\end{cases}
\]
and
\[
\chi_3=\left(\ds \sum_{i=t+1-i_0}^{t-1} n_i(\lambda) \right)
+ \ds \sum_{i=\floor{\frac{t+2}{2}}}^{t-1} (t-1) n n_i(\lambda).
\]
\item Otherwise 
\[
\det \Delta_3=0.
\]
\end{enumerate}
% \begin{proof} 
% \[
% \det \Pi_2 = (-1)^{\Sigma_1}  \left(\det (\gamma_{i,j})_{1 \leq i,j \leq t-1} \right)^{n}  \det \left(
% \begin{array}{ccccc}
% V_1 & V_2 & \dots & V_{\floor{\frac{t}{2}}}\\
% W_1 &&& \\ 
%  & W_2  & &\text{\huge0} \\
%  &     & \ddots \\
% \text{\huge0} &    &   & W_{\floor{\frac{t}{2}}}
% \end{array}
% \right),
% \]
% \end{proof}
% where 
% \[
% W_i= \begin{cases}
% \left(\begin{array}{c|c}
%   {A}_i  & \bar{A}_{t-i} \\\hline
%   \bar{A}_i  & {A}_{t-i}
% \end{array} \right) & \text{ if } 1 \leq i < \floor{\frac{t}{2}}\\
% \left( A_{\frac{t}{2}}+\bar{A}_{\frac{t}{2}}\right) & \text{ if } i= \frac{t}{2},
% \end{cases}
% \]
% and 
% \[
% V_i= \begin{cases}
% \left(\begin{array}{c|c}
%   {B}_i+{\bar{B}}_i  & {B}_{t-i}+\bar{B}_{t-i} 
% \end{array} \right) & \text{ if } 1 \leq i < \floor{\frac{t}{2}}\\
% B_{\frac{t}{2}}+\bar{B}_{\frac{t}{2}} & \text{ if } i= \frac{t}{2}.
% \end{cases}
% \]
 \end{cor}
%  \[
% Y_{(i)}=\begin{cases}
% X & i \neq \alpha_k+i \pmod{t}\\
% (X,y) & i = \alpha_k+i \pmod{t}
% \end{cases}, (Y_{(i)},\bar{Y}_{(i)})=\begin{cases}
% (X,\X) & i \neq \alpha_k+i \pmod{t}\\
% (X,\X,y) & i = \alpha_k+i \pmod{t}
% \end{cases},\]
% \[
% (\bar{Y}_{(i)},Y_{(i)})=\begin{cases}
% (X,\X) & i \neq \alpha_k+i \pmod{t}\\
% (X,\X, \bar{y}) & i = \alpha_k+i \pmod{t}
% \end{cases}
% \]
\begin{proof}[Proof of \cref{thm:even}] By \eqref{oedef}, we see that the numerator of the required even orthogonal character is:
\begin{equation*}
2 \det \left(
    \begin{array}{c}
    \left( \left( 
    (\omega^{p-1}x_i)^{\beta_{j}(\lambda)}+(\bar{\omega}^{p-1}\x_i)^{\beta_{j}(\lambda)}
    \right)_{\substack{1 \leq i\leq n \\ 1\leq j\leq tn+1}}
    \right)_{1\leq p \leq t}  \\[0.4cm]  
      \hline     \\[-0.3cm]
       \left( 
    y^{\beta_{j}(\lambda)}+\bar{y}^{\beta_{j}(\lambda)}
    \right)_{1\leq j\leq tn+1}
    \end{array} 
    \right)
\end{equation*}
% \begin{equation}
% \begin{split}
%     \oe_{\lambda}(X,\, & \omega X,  \dots,\omega^{t-1}X,y) \\ 
%     = & \frac{2 \det \left(
%     \begin{array}{c}
%     \left( \left( 
%     (\omega^{p-1}x_i)^{\beta_{j}(\lambda)}+(\bar{\omega}^{p-1}\x_i)^{\beta_{j}(\lambda)}
%     \right)_{\substack{1 \leq i\leq n \\ 1\leq j\leq tn+1}}
%     \right)_{1\leq p \leq t}  \\[0.4cm]  
%       \hline     \\[-0.3cm]
%       \left( 
%     y^{\beta_{j}(\lambda)}+\bar{y}^{\beta_{j}(\lambda)}
%     \right)_{1\leq j\leq tn+1}
%     \end{array} 
%     \right)}
%     {(1+\delta_{\lambda_{tn+1},0}) \det \left(
%     \begin{array}{c}
%     \left( \left( 
%     (\omega^{p-1}x_i)^{tn+1-j}+(\bar{\omega}^{p-1}\x_i)^{tn+1-j}
%     \right)_{\substack{1 \leq i\leq n \\ 1\leq j\leq tn+1}}
%     \right)_{1\leq p \leq t}  \\  
%       \hline     \\[-0.3cm]
%      \left( y^{tn+1-j}+\bar{y}^{tn+1-j} \right)_{1\leq j\leq tn+1}
%     \end{array} 
%     \right)}.
%     \end{split}
% \end{equation}
First permuting the columns of the matrix in the numerator by $\sigma_{\lambda}$ from \eqref{sigma-perm-m} $(m=1,d_1=0)$ and then permuting the last $t$ rows cyclically, we see that the numerator is
\begin{equation*}
   2 \sgn(\sigma_{\lambda}) (-1)^{(t-1)} \det \left(
    \begin{array}{c|c|c|c}
    A_0^{\lambda}+\bar{A}^{\lambda}_0 & A_1^{\lambda}+\bar{A}^{\lambda}_1 & \dots & A_{t-1}^{\lambda}+\bar{A}^{\lambda}_{t-1} \\\hline&&&\\
     B_0^{\lambda}+\bar{B}^{\lambda}_0 & B_1^{\lambda}+\bar{B}^{\lambda}_1 & \dots & B_{t-1}^{\lambda}+\bar{B}^{\lambda}_{t-1} \\\hline&&&\\
    A_0^{\lambda}+\bar{A}^{\lambda}_0 & \omega A_1^{\lambda}+\omega^{t-1} \bar{A}^{\lambda}_1 & \dots & \omega^{t-1} A^{\lambda}_{t-1}+\omega \bar{A}^{\lambda}_{t-1} \\\hline&&&\\
    \vdots &  \vdots &  \dots &  \vdots  \\&&&\\\hline&&&\\
    A_0^{\lambda}+\bar{A}^{\lambda}_0 & \omega^{t-1} A_1^{\lambda} +\omega \bar{A}^{\lambda}_1 &  \dots &  \omega A^{\lambda}_{t-1}+\omega^{t-1} \bar{A}^{\lambda}_{t-1},
    \end{array} \right),
\end{equation*}
where $A_{p,q}^{\lambda}$,  $\bar{A}_{p}^{\lambda}$, 
$B_{p,q}^{\lambda}$ and $\bar{B}_{p}^{\lambda}$ are defined in \eqref{def AB} and \eqref{def 1AB}.
Applying blockwise row operations $R_1 \to R_1+R_3+\dots+R_{t+1}$ and then $R_i \to R_i-\frac{1}{t} R_1$, $3 \leq i \leq t+1$, we get
\begin{equation}
  \label{even-2}
  2 t^n \sgn(\sigma_{\lambda}) (-1)^{(t-1)} \det \left(
    \begin{array}{c|c|c|c}
    A_0^{\lambda}+\bar{A}^{\lambda}_0 & 0 &  \dots &  0 \\\hline&&&\\
    B_0^{\lambda}+\bar{B}^{\lambda}_0 & B_1^{\lambda}+\bar{B}^{\lambda}_1 & \dots & B_{t-1}^{\lambda}+\bar{B}^{\lambda}_{t-1}  \\\hline&&&\\
    0 & \omega A^{\lambda}_1+ \omega^{t-1} \bar{A}^{\lambda}_1 & \dots & \omega^{t-1} A^{\lambda}_{t-1}+\omega \bar{A}^{\lambda}_{t-1} \\\hline&&&\\
    \vdots &  \vdots &  \dots &  \vdots 
     \\&&&
    \\\hline&&&\\
    0 & \omega^{t-1} A^{\lambda}_1+\omega \bar{A}^{\lambda}_1 & \dots & \omega A_{t-1}^{\lambda}+\omega^{t-1} \bar{A}_{t-1}^{\lambda},
    \end{array} \right).
\end{equation}
Since the denominator in \eqref{oedef} is the numerator evaluated at the empty partition and $n_0(\emptyset,tn+1)=n+1$, $n_i(\emptyset,tn+1)=n$ for all $i \in [1,t-1]$, evaluating the numerator in \eqref{even-2} and then using \cref{cor:det-even}, we see that the denominator of the required even orthogonal character is:
\begin{equation}
\begin{split}
\label{denomi}
(1+\delta_{tn+1}) t^n \sgn(\sigma_{\emptyset}) & (-1)^{(t-1)+\Sigma_3^0} (\det \Gamma)^{n} \det \left(
\begin{array}{cc}
     A_0+\bar{A}_0  \\[0.1cm]
   \hline \\[-0.4cm]
      B_0+\bar{B}_0
\end{array}
\right)\\
   & \times \prod_{q=1}^{\floor{\frac{t-1}{2}}}
\det\left(\begin{array}{c|c}
   A_{q}  & \bar{A}_{t-q} \\[0.2cm]
   \hline \\[-0.3cm]
   \bar{A}_{q}  & A_{t-q}
   \end{array}\right) 
\times
\begin{cases}
\det \left(A_{\frac{t}{2}}+ \bar{A}_{\frac{t}{2}} \right)  & t \text{ even,}\\
1 & t \text{ odd},
\end{cases}     
\end{split}
\end{equation}
where 
\[
\Sigma_3^0=\begin{cases}
 \ds \sum_{q=\frac{t+2}{2}}^{t-1} n^2
&  t \text{ even},\\
0 &  t \text{ odd}.
\end{cases}
\]
If $\core \lambda t \not \in \mathcal{Q}^{(t)}_{1,0,0} \cup \mathcal{Q}^{(t)}_{2,1,k}$ for all $k \in [\rk(\core \lambda t)]$, then by \cref{lem:symp-1}, \cref{cor:z=-1}, \cref{cor:det-even} and \cref{det-even-1}, the numerator in \eqref{even-2} is 0. So,
% If $n_0(\lambda) >n+1,$ then 
\[
 \oe_{\lambda}(X,\, \omega X,\dots,\omega^{t-1}X,y)=0.
\]
If $\core \lambda t \in \mathcal{Q}^{(t)}_{1,0,0}$, then by \cref{lem:symp-1}, $n_0(\lambda)=n+1$ and $n_i(\lambda)+n_{t-i}(\lambda)=2n$ and the numerator in \eqref{even-2} is the same as:  
% Using \cref{cor:det-even}, 
% we see that the 
\begin{equation}
\label{even-1}
2 t^n \sgn(\sigma_{\lambda}) (-1)^{(t-1)} \det \left(
\begin{array}{cc}
     A_0^{\lambda}+\bar{A}_0^{\lambda}  \\[0.1cm]
   \hline \\[-0.4cm]
      B_0^{\lambda}+\bar{B}_0^{\lambda}
\end{array}
\right)
    \times \det \left(
    \begin{array}{c}
 \omega^{ij} A_j^{\lambda}+ \omega^{t-ij} \bar{A}_j^{\lambda}
    \end{array}
    \right)_{1 \leq i,j \leq t-1}.
\end{equation}
The last matrix in \eqref{even-1} is $\Pi_3$ defined in \eqref{p11}. Using \cref{cor:det-even}, we see that the numerator is 
\begin{multline}
\label{numer}
    2 t^n \sgn(\sigma_{\lambda}) (-1)^{(t-1)+\Sigma_3} \left(\det \Gamma \right)^{n} \det \left(
\begin{array}{cc}
     A_0^{\lambda}+\bar{A}_0^{\lambda}  \\[0.1cm]
   \hline \\[-0.4cm]
      B_0^{\lambda}+\bar{B}_0^{\lambda}
\end{array}
\right) \\
  \times   \prod_{q=1}^{\floor{\frac{t-1}{2}}}
\det\left(\begin{array}{c|c}
   A^{\lambda}_{q}  & \bar{A}^{\lambda}_{t-q} \\[0.2cm]
   \hline \\[-0.3cm]
   \bar{A}^{\lambda}_{q}  & A^{\lambda}_{t-q}
   \end{array}\right)  
\times
\begin{cases}
\det \left(A^{\lambda}_{\frac{t}{2}}+ \bar{A}^{\lambda}_{\frac{t}{2}} \right)  & t \text{ even,}\\
1 & t \text{ odd.}
\end{cases}      
\end{multline}
By \eqref{E4} and \eqref{oeshifted}, we note that 
\begin{equation*}
    \frac{\det \left(A^{\lambda}_{\frac{t}{2}}+ \bar{A}^{\lambda}_{\frac{t}{2}} \right)}
    {\det \left(A_{\frac{t}{2}}+ \bar{A}_{\frac{t}{2}} \right)} =  (-1)^{\sum_i \lambda_i^{(t/2)}} \oo_{\lambda^{(t/2)}}(-X^t).
\end{equation*}
So, applying \eqref{F4} and \cref{lem:s-new}, the ratio of \eqref{numer} and \eqref{denomi}, the required even orthogonal character is
\begin{equation*}
% \tcr{\frac{1+\delta_{\lambda_n^{(t/2)}}}{2}} 
(-1)^{\epsilon_3} \, \sgn(\sigma_{\lambda}) \oe_{\lambda^{(0)}}(X^t,y^t) \prod_{i=1}^{\floor{\frac{t-1}{2}}} s_{\pi_i^{(3)}}(X^t,\X^t)
 \times \begin{cases}
(-1)^{\sum_i \lambda_i^{(t/2)}} \oo_{\lambda^{(t/2)}}(-X^t)& t \text{ even,}\\
1 & t \text{ odd,}
\end{cases}   
\end{equation*}
where 
\begin{equation}
\label{epsilo3}
\begin{split}
\epsilon_3=\frac{t(t-1)}{2}\frac{n(n+1)}{2} + & \begin{cases}
n \ds \sum_{q=\frac{t+2}{2}}^{t-1} 
(n_{q}(\lambda)-n) 
&  t \text{ even},\\
0 &  t \text{ odd}
\end{cases}  \\
& \qquad+  \sum_{i=1}^{\floor{\frac{t-1}{2}}} \left(
\frac{n_{t-i}(\lambda)
(n_{t-i}(\lambda)-1)}{2}-\frac{n(n-1)}{2} 
\right). 
\end{split}
\end{equation}
If $n_0(\lambda)=n$, then the numerator is
\begin{equation*}
 2 t^n \sgn(\sigma_{\lambda}) (-1)^{(t-1)} \det \left(
    A_0^{\lambda}+\Bar{A}_0^{\lambda} \right)
    \times \det \left( 
    \begin{array}{c}
         (B_j+ \bar{B}_j)_{1 \leq j \leq t-1}  \\\\\hline \\
          \left(
 \omega^{ij} A_j^{\lambda}+ \omega^{t-ij} \bar{A}_j^{\lambda}
    \right)_{1 \leq i,j \leq t-1}
    \end{array}
    \right)
%     \det \left(
%     \begin{array}{c|c|c}
%     B_1+ \bar{B}_1 & \dots &  B_{t-1}+ \bar{B}_{t-1}  \\\hline&\\
%       \omega A_1+ \omega^{t-1} \bar{A}_1 & \dots & \omega^{t-1} A_{t-1}+\omega \bar{A}_{t-1} \\\hline&\\
%      \omega^2 A_1+\omega^{t-2} \bar{A}_1 & \dots & \omega^{t-2} A_{t-1}+\omega^2 \bar{A}_{t-1} \\\hline&\\
%   \vdots &  \dots &  \vdots \\\hline&\\
%      \omega^{t-1} A_1+\omega \bar{A}_1 & \dots & \omega A_{t-1}+\omega^{t-1} \bar{A}_{t-1} 
%     \end{array}
%     \right)
\end{equation*}
% We note that 
% \[
% \det \left( 
%     \begin{array}{c}
%          (B_j+ \bar{B}_j)_{1 \leq j \leq t-1}  \\\\\hline \\
%           \left(
%  \omega^{ij} A_j^{\lambda}+ \omega^{t-ij} \bar{A}_j^{\lambda}
%     \right)_{1 \leq i,j \leq t-1}
%     \end{array}
%     \right) = \det(\gamma_{i,j})  \det \left(
% \begin{array}{ccccc}
% V_1 & V_2 & \dots & V_{\floor{\frac{t}{2}}}\\
% W_1 &&& \\ 
%  & W_2  & &\text{\huge0} \\
%  &     & \ddots \\
% \text{\huge0} &    &   & W_{\floor{\frac{t}{2}}}
% \end{array}
% \right),
% \]
% where 
% \[
% W_i= \begin{cases}
% \left(\begin{array}{c|c}
%   {A}_i  & \bar{A}_{t-i} \\\hline
%   \bar{A}_i  & {A}_{t-i}
% \end{array} \right) & \text{ if } 1 \leq i < \floor{\frac{t}{2}}\\
% \left( A_{\frac{t}{2}}+\bar{A}_{\frac{t}{2}}\right) & \text{ if } i= \frac{t}{2},
% \end{cases}
% \]
% and 
% \[
% V_i= \begin{cases}
% \left(\begin{array}{c|c}
%   {B}_i+{\bar{B}}_i  & {B}_{t-i}+\bar{B}_{t-i} 
% \end{array} \right) & \text{ if } 1 \leq i < \floor{\frac{t}{2}}\\
% B_{\frac{t}{2}}+\bar{B}_{\frac{t}{2}} & \text{ if } i= \frac{t}{2}.
% \end{cases}
% \]
The last matrix is the same as $\Delta_3$ defined in \eqref{delta-1}. We use \cref{det-even-1} to factorize the determinant. 
% the numerator is
% \begin{equation*}
%  2 t^n \sgn(\sigma_{\lambda}) (-1)^{(t-1)+\chi_3} (\det \Gamma)^n \det \left(
%     A_0^{\lambda}+\Bar{A}_0^{\lambda} \right)
%   \det \left( \begin{array}{c}
%       O_{i_0}^{(3)}  \\[0.1cm]
%   \hline \\[-0.4cm]
%       W_{i_0}^{(3)}
%  \end{array} \right)
%  \prod_{\substack{i=1\\i \neq i_0}}^{\floor{\frac{t}{2}}}
%  \det W_i^{(3)}.   
% \end{equation*}
If $\core \lambda t \in \mathcal{Q}^{(t)}_{2,1,k}$ for all $k \in [\rk(\core \lambda t)]$, then by \cref{cor:z=-1}, \eqref{i0} holds.
\underline{Case 1}. $t$ is even and $i_0=\frac{t}{2}$. Then $n_{t/2}(\lambda)=n+1$ and using \cref{det-even-1}, the factorization for the numerator is 
\begin{equation}
\label{1num}
2 t^n \sgn(\sigma_{\lambda}) (-1)^{(t-1)+\chi_3} (\det \Gamma)^n  \det
   \left(\begin{array}{c}
         A_0^{\lambda}+\bar{A}^{\lambda}_0 
    \end{array}\right) \det \left(
\begin{array}{c}
     B_{\frac{t}{2}}^{\lambda} + \bar{B}_{\frac{t}{2}}^{\lambda} \\[0.2cm]
     \hline\\[-0.3cm]
     A_{\frac{t}{2}}^{\lambda} + \bar{A}_{\frac{t}{2}}^{\lambda}
\end{array} \right)
\prod_{q=1}^{\frac{t-2}{2}} \det \left(\begin{array}{c|c}
A_q^{\lambda}  & \bar{A}_{t-q}^{\lambda} \\[0.2cm]
\hline\\[-0.3cm]
\bar{A}_q^{\lambda}  & A_{t-q}^{\lambda}
\end{array}\right).     
\end{equation}
By \cref{lem:DI}, we have
\begin{equation}
\label{ev-1}
\det \left(A_{\frac{t}{2}}+ \bar{A}_{\frac{t}{2}} \right)  =  \frac{1}{2} \det (x_i^{t(n-j)}+\x_i^{t(n-j)}) {\ds \prod_{i=1}^n (x_i^{t/2}+\bar{x}_i^{t/2})}.  
\end{equation}
Substituting in \eqref{denomi}, the denominator is
\begin{equation}
\begin{split}
\label{1denomi}
(1+\delta_{tn+1}) t^n &\sgn(\sigma_{\emptyset})  (-1)^{(t-1)+\Sigma_3^0} (\det \Gamma)^{n} \det \left(
\begin{array}{cc}
     A_0+\bar{A}_0  \\[0.1cm]
   \hline \\[-0.4cm]
      B_0+\bar{B}_0
\end{array}
\right)\\
   & \times \prod_{q=1}^{\floor{\frac{t-1}{2}}}
\det\left(\begin{array}{c|c}
   A_{q}  & \bar{A}_{t-q} \\[0.2cm]
   \hline \\[-0.3cm]
   \bar{A}_{q}  & A_{t-q}
   \end{array}\right) 
\times
\frac{1}{2} \det (x_i^{t(n-j)}+\x_i^{t(n-j)}) {\ds \prod_{i=1}^n (x_i^{t/2}+\bar{x}_i^{t/2})}. 
\end{split}
\end{equation}
% \begin{multline*}
% =\det \left(
% \begin{array}{c}
%     (x_i^{t(n+1-j)}+\x_i^{t(n+1-j)})_{\substack{1 \leq i \leq n\\1 \leq j \leq n+1}} \\\hline 
%      (y^{t(n+1-j)}+\bar{y}^{t(n+1-j)})_{1 \leq j \leq n+1} 
% \end{array}
% \right) \times \det \left(x_i^{t/2+t(n-j)} + \x_i^{t/2+t(n-j)} \right)_{1 \leq i,j \leq n}\\
%     \times \prod_{q=1}^{\frac{t-2}{2}}
% \det\left(\begin{array}{c|c}
%   \left(x_i^{q+t(n-j)} \right)_{1 \leq i,j \leq n}  & \left(\x_i^{t-q+t(n-j)} \right)_{1 \leq i,j \leq n} \\[0.2cm]
%   \hline \\[-0.3cm]
%   \left(\x_i^{q+t(n-j)} \right)_{1 \leq i,j \leq n}  & \left(x_i^{t-q+t(n-j)} \right)_{1 \leq i,j \leq n}
%   \end{array}\right).
% \end{multline*}
By \eqref{E4} and \eqref{oeshifted}, we note that 
\begin{equation*}
    \frac{\det \left(
\begin{array}{c}
     B_{\frac{t}{2}}^{\lambda} + \bar{B}_{\frac{t}{2}}^{\lambda} \\[0.2cm]
     \hline\\[-0.3cm]
     A_{\frac{t}{2}}^{\lambda} + \bar{A}_{\frac{t}{2}}^{\lambda}
\end{array} \right)}
    {\prod_{i=1}^n (x_i^{t/2}+\bar{x}_i^{t/2}) \det \left(
\begin{array}{c}
     B_{0} + \bar{B}_{0} \\[0.2cm]
     \hline\\[-0.3cm]
     A_{0} + \bar{A}_{0}
\end{array} \right)} =  (-1)^{\sum_i \lambda_i^{(t/2)}} (y^{t/2}+\bar{y}^{t/2}) \oo_{\lambda^{(t/2)}}(-X^t,-y^t).
\end{equation*}
Note that $\lambda_{tn+1}$ is zero iff $\lambda^{(0)}_n$.
Hence, using \eqref{E4}, \cref{lem:s-new} and the ratio of \eqref{1num} and \eqref{1denomi}, the even orthogonal character is
\[
(-1)^{\epsilon_3+n} \, \sgn(\sigma_{\lambda})
% (1+\delta_{\lambda^{(t/2)}_{n+1},0})
(y^{t/2}+\bar{y}^{t/2}) \, \oe_{\lambda^{(0)}}(X^t) \, (-1)^{ \sum_i \lambda_i^{(t/2)}} \oo_{\lambda^{(t/2)}}(-X^t,-y^t) \prod_{q=1}^{\frac{t-2}{2}} s_{\pi_q^{(1)}}(X^t,\X^t),
\]
where 
\begin{multline*}
\epsilon_3= \frac{t(t-1)}{2} \frac{n(n+1)}{2} +\left(\ds 
\sum_{i=t+1-i_0}^{t-1} n_{i}(\lambda) \right)
+ \ds \sum_{i=\floor{\frac{t+2}{2}}}^{t-1} 
(t-1) n (n_{i}(\lambda)-n) \\\
+ \sum_{q=1}^{\frac{t-2}{2}} \left( \frac{n_{t-q}(\lambda)(n_{t-q}(\lambda)-1)}{2}-\frac{n(n-1)}{2} \right).    
\end{multline*}
Case 2. $i_0 \neq \frac{t}{2}$. In this case,  the factorization for the numerator is 
\begin{multline*}
 2 t^n \sgn(\sigma_{\lambda}) (-1)^{(t-1)+\chi_3} (\det \Gamma)^n \det
   \left(\begin{array}{c}
         A^{\lambda}_0+\bar{A}^{\lambda}_0
    \end{array}\right) \det \left(
\begin{array}{c|c}
     B_{i_0}^{\lambda}+ \bar{B}_{i_0}^{\lambda}
     & B_{t-i_0}^{\lambda} + \bar{B}_{t-i_0}^{\lambda} \\[0.2cm]
     \hline \\[-0.3cm]
     A_{i_0}^{\lambda} & \bar{A}_{t-i_0}^{\lambda} \\[0.2cm]
     \hline \\[-0.3cm]
     \bar{A}_{t-i_0}^{\lambda} & A_{i_0}^{\lambda}
\end{array} 
\right)_{2n+1 \times 2n+1} \\
\times \prod_{\substack{q=1 \\ q \neq i_0}}^{\floor{\frac{t-1}{2}}}  \det\left(\begin{array}{c|c}
   A_{q}  & \bar{A}_{t-q} \\[0.2cm]
   \hline \\[-0.3cm]
   \bar{A}_{q}  & A_{t-q}
   \end{array}\right) \times 
   \begin{cases}
   \det(A^{\lambda}_{t/2}+\bar{A}^{\lambda}_{t/2})  & t \text{ is even},\\
1 & t \text{ is odd}.
   \end{cases}
     \end{multline*}
By \cref{lem:DI}, we have
\[
 \det \left(
\begin{array}{c}
     A_0+\bar{A}_0  \\\hline 
     B_0+\bar{B}_0 
\end{array}
 \right)= (-1)^n y^{-n} {\ds \prod_{i=1}^{n} (y-x_i)(y-\x_i)} \det (x_i^{n-j}+\x_i^{n-j}),
\]
\[
 \det \left(A_{\frac{t}{2}}+ \bar{A}_{\frac{t}{2}} \right) =  \frac{1}{2} {\ds \prod_{i=1}^n (x_i^{t/2}+\bar{x}_i^{t/2})}  \det (x_i^{n-j}+\x_i^{n-j}),
\]
\[
\det \left(\begin{array}{c|c}
   A_{i_0}  & \bar{A}_{t-i_0} \\[0.2cm]
   \hline \\[-0.3cm]
   \bar{A}_{i_0}  & A_{t-i_0}
   \end{array}\right) = (-1)^{\frac{n(n-1)}{2}} \prod_{1 \leq i < j \leq n} (x_i^t-x_j^t)(x_i^t-\x_j^t) (x_j^t-\x_i^t) (\x_i^t-\x_j^t) \ds \prod_{i=1}^n (x_i^t-\x_i^t).
\]
% \begin{multline*}
% =\det \left(
% \begin{array}{c}
%     (x_i^{t(n+1-j)}+\x_i^{t(n+1-j)})_{\substack{1 \leq i \leq n\\1 \leq j \leq n+1}} \\\hline 
%      (y^{t(n+1-j)}+\bar{y}^{t(n+1-j)})_{1 \leq j \leq n+1} 
% \end{array}
% \right) \times \det \left(x_i^{t/2+t(n-j)} + \x_i^{t/2+t(n-j)} \right)_{1 \leq i,j \leq n}\\
%     \times \prod_{q=1}^{\frac{t-2}{2}}
% \det\left(\begin{array}{c|c}
%   \left(x_i^{q+t(n-j)} \right)_{1 \leq i,j \leq n}  & \left(\x_i^{t-q+t(n-j)} \right)_{1 \leq i,j \leq n} \\[0.2cm]
%   \hline \\[-0.3cm]
%   \left(\x_i^{q+t(n-j)} \right)_{1 \leq i,j \leq n}  & \left(x_i^{t-q+t(n-j)} \right)_{1 \leq i,j \leq n}
%   \end{array}\right).
% \end{multline*}
% In this case, using \eqref{ev-1} and \eqref{ev-2}, the denominator is
% \begin{multline*}
% \pm \frac{1}{2} \prod_{i=1}^{n}(1-\x_i)(1-x_i)  \det \left(x_i^{n-j}+\x_i^{n-j} \right)_{1 \leq i,j \leq n} \prod_{q=1}^{\floor{\frac{t-1}{2}}}
% \det\left(\begin{array}{c|c}
%   \left(x_i^{q+t(n-j)} \right)_{1 \leq i,j \leq n}  & \left(\x_i^{t-q+t(n-j)} \right)_{1 \leq i,j \leq n} \\[0.2cm]
%   \hline \\[-0.3cm]
%   \left(\x_i^{q+t(n-j)} \right)_{1 \leq i,j \leq n}  & \left(x_i^{t-q+t(n-j)} \right)_{1 \leq i,j \leq n}
%   \end{array}\right) \\
%  \times  \det (x_i^{n-j}+\x_i^{n-j}) {\ds \prod_{i=1}^n (x_i^{1/2}+\bar{x}_i^{1/2})}
% \end{multline*}
In this case, using \cref{lem:s-new-1} and \cref{lem:s-new}, the required even orthogonal character is
\begin{multline*}
(-1)^{\epsilon_3} \, \sgn(\sigma_{\lambda})  \oe_{\lambda^{(0)}}(X^t)  
\Big( y^{-t(\lambda_1^{(t-i_0)}
+n_{(t-i_0)}(\lambda)-n)+i_0} s_{\pi_{i_0}^{(3)}}(X^t,\X^t,y^t) +
y^{t(\lambda_1^{(t-i_0)}
+n_{t-i_0}(\lambda)+n)-i_0} \\
\times s_{\pi_{i_0}^{(3)}}(X^t,\X^t, \bar{y}^t) \Big) \times 
 \prod_{\substack{j=1 \\ j \neq i_0}}^{\floor{\frac{t-1}{2}}}  s_{\pi_i^{(3)}}(X^t,\X^t) \times 
 \begin{cases}
% (1+\delta_{\lambda^{(t/2)}_n,0})
(-1)^{\sum_i \lambda_i^{(t/2)}} \oo_{\lambda^{(t/2)}}(-X^t) & t \text{ is even},\\
1 & t \text{ is odd}.
 \end{cases}
\end{multline*}
This completes the proof.
% where 
% \[
% \epsilon= \frac{t(t-1)}{2}\frac{n(n+1)}{2} + \chi_3-\Sigma_3^0+ \sum_{q=1}^{\frac{t-2}{2}} \left( \frac{n_{t-q}(\lambda)(n_{t-q}(\lambda)-1)}{2}-\frac{n(n-1)}{2} \right)
% \]
% and 
% $\ds \pi^{(3)}_i=  \lambda^{(t-i)}_1 +
% \left(\lambda^{(i)}, 0, -\rev(\lambda^{(t-i)})\right)$ has $2n+\delta_{i,i_0}$ parts for $1 \leq i \leq \floor{\frac{t-1}{2}}$.
\end{proof}
\section{Generating Functions}
\label{sec:gf}

We now give enumerative results for $(z_1,z_2,k)$-asymmetric partitions defined in \eqref{defn:asym}.
\begin{prop}
\label{prop:bij} 
Fix $z_1 > z_2 \geq 0$ and $k \geq 1$. The number of $(z_1,z_2,k)$-asymmetric partitions of $m$ is equal to the number of partitions of $m$
of the form
\begin{multline*}
\lambda^{\{a_1,\dots,a_r\}} = \left(z_1(r-1)+r,\underbrace{2r-1,\dots,2r-1}_{a_r},\underbrace{2r-3,\dots,2r-3}_{a_{r-1}-a_{r}},
\dots,\underbrace{2k-1,\dots,2k-1}_{a_k-a_{k+1}},\right. \\
\left. \underbrace{2k-2,\dots,2k-2}_{a_{k-1}-a_{k}},\dots,\underbrace{2,\dots,2}_{a_1-a_2},\underbrace{1,\dots,1}_{z_2}\right),
\end{multline*}
for $r \geq 1$ and $ \{a_1,\dots,a_r\}_{>} \subset \mathbb{Z}_{\geq 0}$.
\end{prop}
\begin{proof} Let $\lambda=(\alpha_1,\dots,\alpha_k,\dots,\alpha_r|\alpha_1+z_1,\dots,\widehat{\alpha_k+z_1},\dots,\alpha_r+z_1,z_2)$ be a $(z_1,z_2,k)$-asymmetric partition of $m$. It is easy to see that mapping $\lambda$ to $\lambda^{\{\alpha_1,\dots,\alpha_r\}}$ gives the required bijection. 
\end{proof}
\cref{prop:bij} gives an expression for the generating function:
\begin{cor}
\[ \sum_{\lambda \in \mathcal{Q}_{z_1,z_2,k}} q^{|\lambda|} =
\sum_{n \geq 1} \frac{q^{z_2+z_1(n-1)+n+n(n-2)+k}}{(1-q^2) \cdots (1-q^{2k-2})(1-q^{2k-1})\cdots(1-q^{2n-1})}.
\]
\end{cor}

\begin{cor} The number of $(z_1,z_2,k)$-asymmetric partitions $N$ of $n$ with the Frobenius rank $1$ is 
\[
N = \begin{cases}
1 & n \geq z_2+1,\\
0 & n \leq z_2.
\end{cases}
\]
\end{cor}

\noindent 
Recall, $\mathcal{Q}^{(t)}_{z_1,z_2,k}$ from \cref{defn:asym}. For $z_1>z_2$, let 
\[
\mathcal{Q}^{(t)}_{z_1,z_2} = \ds \bigcup_k \mathcal{Q}^{(t)}_{z_1,z_2,k}.
\]
We now enumerate the $t$-core partitions 
in $\mathcal{Q}^{(t)}_{z+2,0} \cup \mathcal{Q}^{(t)}_{z+2,z+1}$.
% be the set of $(z_1,z_2,k)$ asymmetric $t$-core partitions for all $k$. 
% $\lambda$ is either $(z+2,0,k)$ or $(z+2,z+1,k)$-asymmetric
%  $\lambda \in \mathcal{Q}_{z+2,0,k} \cup \mathcal{Q}_{z+2,z+1,k}$ 
%  for some $k \in [\rk(\lambda)]$. Let $\theta_k=(\alpha_k+1) \pmod{t}$ and
% \[
% i_0(\lambda) = \begin{cases}
% t-z-1-\theta_k & \theta_k \in \left[\floor{\frac{t-z-1}{2}}+1,t-z-1\right],\\
% \theta_k & \text{otherwise}.
% \end{cases}
% \]
Represent the elements of $\mathbb{Z}^{\floor{\frac{t-z}{2}}} \times \Big\{0,\dots,\floor{\frac{t-z-1}{2}},t-z,\dots,t-1\Big\}$ by $(\vec{v},\check{v}) \coloneqq (v_0,\dots, v_{\floor{\frac{t-z-2}{2}}}, \Check{v})$. 

\begin{thm}
\label{thm:bijection}
Fix $0 < z+2 \leq t+2$.
Define $\Vec{b} \in \mathbb{Z}^{\floor{\frac{t-z}{2}}}$ by $\Vec{b}_i \coloneqq t-z-1-2i$.
Then there exists a bijection 
$\psi : \mathcal{Q}^{(t)}_{z+2,0} \cup \mathcal{Q}^{(t)}_{z+2,z+1} \rightarrow \mathbb{Z}^{\floor{\frac{t-z}{2}}} \times \Big\{0,\dots,\floor{\frac{t-z-1}{2}},t-z,\dots,t-1\Big\}$
satisfying 
\[
|\lambda|=t ||\Vec{\psi(\lambda)}||^2-\Vec{b} \cdot \Vec{\psi(\lambda)}
% -\frac{i_0}{2}(1+(-1)^{t-z})
+ 
\begin{cases}
\widecheck{\psi(\lambda)} & \widecheck{\psi(\lambda)} \in [t-z,t-1] \cup \{\frac{t-z-1}{2}\}, \\
t(n_{\widecheck{\psi(\lambda)}}(\lambda)-n)+t-z-1 & \text{otherwise,}
\end{cases}
\]
where $\cdot$ represents the standard inner product.
\end{thm}
% \begin{lem}
% Assume $t$ is odd. There is a bijection $\psi$ between $\mathcal{Q}_{1,t}$ and $\mathbb{Z}^{\frac{t-1}{2}} \times [\frac{t-1}{2}]$.
% \end{lem}
\begin{proof} Suppose $\lambda \in \mathcal{Q}^{(t)}_{z+2,0} \cup \mathcal{Q}^{(t)}_{z+2,z+1}$ such that $\ell(\lambda) \leq tn+1$ for some $n \geq 1$. Then by \cref{lem:converse: sym}, there exists a unique $i_0 \in [0,\floor{\frac{t-z-1}{2}}] \cup [t-z,t-1]$ such that \eqref{n_{t-1}} holds. 
Define the map $\psi$ by
\[
(\psi(\lambda))_i  
\coloneqq
n_{i}(\lambda)-n, 
\quad 0 \leq i \leq \floor{\frac{t-z-2}{2}}, \quad \widecheck{\psi(\lambda)} \coloneqq i_0.
\]
Since $n$ is not unique, it is not a priori clear that $\psi$ is well-defined.
But from the definition of $n_{i}(\lambda)$), it is easy to see that
$n_{i}(\lambda,tn+1)-n=n_{i}(\lambda,tn+t+1)-n-1$. 
Hence, $\psi(\lambda)$ is indeed well-defined. 

To show that $\psi$ is a bijection, we define the inverse of $\psi$ as follows. 
For a vector $(\Vec{v},\check{v}) = \left(v_0,\dots,v_{\floor{\frac{t-z-2}{2}}},\check{v} \right)$,
let $n=\text{max}\{|v_0|,|v_1|,\dots,|v_{\floor{\frac{t-z-2}{2}}}|\}$ 
and for $0 \leq i \leq t-1$, 
\[
m_i=
\begin{cases}
n+v_i &  0 \leq i \leq \floor{\frac{t-z-2}{2}},\\
n-v_{t-z-1-i}  & \floor{\frac{t-z+1}{2}} \leq i \leq t-z-1,\\
n & \text{ otherwise},
\end{cases} \text{ and } r_i=
\begin{cases}
m_i+1 &  i=\check{v},\\
m_i & \text{ otherwise}.
\end{cases}
\]
% \[
% r_i=
% \begin{cases}
% m_i+1 &  i=v_{\floor{\frac{t-z}{2}}},\\
% m_i & \text{ otherwise}.
% \end{cases}
% \]
By construction, $\ds \sum_{i=0}^{t-1} r_i=tn+1$,
\begin{equation*}
\begin{split}
    r_{i}+r_{t-z-1-i} =& 
    \begin{cases}
    2n+1+\delta_{\check{v},\frac{t-z-1}{2}} & \text{ if } i=\check{v}\\
    2n & \text{ otherwise}
    \end{cases}
    \quad \text{for} \quad  0 \leq i \leq  \floor{\frac{t-z-1}{2}}, \\ 
    \text{and} \quad r_{i} =& \begin{cases}
    n+1 & \text{ if } i=\check{v}\\
    n & \text{ otherwise}
    \end{cases} \quad \text{for} \quad t-z \leq i \leq t-1,
\end{split}
\end{equation*}
By \cref{lem:converse: sym}, there is a unique $t$-core $\lambda \in \mathcal{Q}^{(t)}_{z+2,0} \cup \mathcal{Q}^{(t)}_{z+2,z+1}$ satisfying $n_{i}(\lambda)=r_i$. and we set
$\psi^{-1}(\Vec{v},\check{v}) = \lambda$.
Moreover the size of $\lambda$ is computed as
\begin{equation}
\label{lbda1}
|\lambda|
= \sum_{i=1}^{tn+1} \beta_i(\lambda)-\frac{tn(tn+1)}{2}.    
\end{equation}
Since $\lambda$ is a $t$-core, $tj+i$, $0 \leq j \leq n_{i}(\lambda)-1$, $0 \leq i \leq t-1$ are the parts of $\beta(\lambda)$ (see \cref{prop:mcd-t-core-quo}). So, 
\begin{multline*}
 \sum_{i=1}^{tn+1} \beta_i(\lambda)=\sum_{i=0}^{t-1}
 \left(
 i(n_{i}(\lambda))
 +\frac{n_{i}(\lambda)(n_{i}(\lambda)-1)t}{2}
 \right)\\
= \sum_{i=0}^{t-1}
 \left(
 i(n_{i}(\lambda)-n)
 \right)
 +\frac{tn(t-1)}{2}
 + \frac{t}{2} \sum_{i=0}^{t-1}
 n_{i}(\lambda)^2-\frac{t(tn+1)}{2}.
\end{multline*}
Substituting this in \eqref{lbda1} gives
\begin{multline*}
|\lambda|= \sum_{i=0}^{t-1}
\left(i(n_{i}(\lambda)-n)\right)
+\frac{t}{2} \left(\sum_{i=0}^{t-1} n_{i}(\lambda)^2-tn^2-2n-1\right) \\
= \sum_{i=0}^{t-1}
\left(i(n_{i}(\lambda)-n)\right)
+\frac{t}{2}\sum_{i=0}^{t-1}  \left( n_{i}(\lambda)-n\right)^2-\frac{t}{2}.
\end{multline*}
Since $\lambda \in \mathcal{Q}^{(t)}_{z+2,0} \cup \mathcal{Q}^{(t)}_{z+2,z+1}$, 
using \cref{lem:converse: sym},  we have
\[ \sum_{i=0}^{t-1}\left(n_{i}(\lambda)-n \right)^2 = 2 ||\Vec{v}||^2 +1-
\begin{cases}
0 & \check{v} \in [t-z,t-1] \cup \{\frac{t-z-1}{2}\}, \\
2(n_{\check{v}}(\lambda)-n) & \text{otherwise,}
\end{cases}
\]
% If $i_0 \in [t-z,t-1]$ or $i_0=\frac{t-z-1}{2}$, then 
% \[
% \sum_{i=0}^{t-1}\left(n_{i}(\lambda)-n \right)^2 = 1+2 \sum_{i=0}^{\floor{\frac{t-z-2}{2}}}(n_{i}(\lambda)-n)^2
% = 2 (||\Vec{v}||^2 -i_0^2) +1.  
% \]
% Otherwise 
% \begin{multline*}
%  \sum_{i=0}^{t-1}\left(n_{i}(\lambda)-n \right)^2
% = \left(n_{i_0}(\lambda)-n \right)^2+
% \left(n_{t-z-1-i_0}(\lambda)-n \right)^2+
% 2 \sum_{\substack{i=0\\i \neq i_0}}^{\floor{\frac{t-z-2}{2}}}(n_{i}(\lambda)-n)^2\\
% = 1-2(n_{i_0}(\lambda)-n)+
% 2 \sum_{i=0}^{\floor{\frac{t-z-2}{2}}}(n_{i}(\lambda)-n)^2
% =
% 2 (||\Vec{v}||^2 -i_0^2) +1-2(n_{i_0}(\lambda)-n).
% \end{multline*}
and
\[
\sum_{i=0}^{t-1}
i\left(n_{i}(\lambda)-n\right) = 
\ds \sum_{i=0}^{\floor{\frac{t-z-2}{2}}}
(2i+z+1-t)(n_i(\lambda)-n) +
\begin{cases}
\check{v} & \check{v} \in [t-z,t-1] \cup \{\frac{t-z-1}{2}\}, \\
t-z-1 & \text{otherwise.}
\end{cases}
\]
Now observe that
\[
-\Vec{b} \cdot \Vec{v}= \ds \sum_{i=0}^{\floor{\frac{t-z-2}{2}}}
(2i+z+1-t)(n_i(\lambda)-n) 
% + \frac{i_0}{2}(1+(-1)^{t-z})
\]
\[
-\Vec{b} \cdot \Vec{v} = \sum_{i=0}^{t-1}i\left(n_{i}(\lambda)-n\right) 
% + \frac{i_0}{2}(1+(-1)^{t-z}) 
- 
\begin{cases}
\check{v} & \check{v} \in [t-z,t-1] \cup \{\frac{t-z-1}{2}\}, \\
t-z-1 & \text{otherwise.}
\end{cases}
\]
\[
\frac{t}{2} \sum_{i=0}^{t-1}\left(n_{i}(\lambda)-n \right)^2 - \frac{t}{2} 
=
t||\Vec{v}||^2- \begin{cases}
0 & \check{v} \in [t-z,t-1] \cup \{\frac{t-z-1}{2}\}, \\
t(n_{\check{v}}(\lambda)-n) & \text{otherwise}.
\end{cases}
\]
Hence \[
|\lambda|=t ||\Vec{v}||^2-\Vec{b} \cdot \Vec{v}
% -\frac{i_0}{2}(1+(-1)^{t-z})
+ 
\begin{cases}
\check{v} & \check{v} \in [t-z,t-1] \cup \{\frac{t-z-1}{2}\}, \\
t(n_{\check{v}}(\lambda)-n)+t-z-1 & \text{otherwise,}
\end{cases}
\]
completing the proof.
\end{proof}
\begin{cor}
There are infinitely many $t$-cores in $\mathcal{Q}^{(t)}_{z+2,0} \cup \mathcal{Q}^{(t)}_{z+2,z+1}$ for $t \geq z$.
\end{cor}
\bibliographystyle{alpha}
\bibliography{Bibliography}
\end{document}